\documentclass[11pt]{article}

\usepackage{amsmath, amssymb, amsfonts,amsthm}
\usepackage{mathtools,float, keyval, subfig}  
\usepackage{commath}              
\usepackage{graphicx}
\usepackage[margin=2cm]{geometry}

\usepackage{algorithmic}			
\usepackage{algorithm}
\usepackage{placeins}                
\usepackage{url}
\usepackage{enumerate}
\usepackage{physics}
\usepackage{texdraw,epsfig,bm}
\usepackage{bigints}
\usepackage{multirow}
\usepackage{capt-of}
\usepackage{url}
\usepackage{rotating}
\usepackage{array}
\usepackage{setspace}
\usepackage{authblk}       			  
\usepackage[bookmarks=false]{hyperref}

\newtheorem{theorem}{Theorem}[section]
\newtheorem{lemma}[theorem]{Lemma}
\newtheorem{corollary}[theorem]{Corollary}
\newtheorem{prop}[theorem]{Proposition}

\newtheorem{example}{Example}[section]

\newcommand{\Gb}{\ensuremath{\mathrm{G}}}       
\newcommand{\Gc}{\ensuremath{\mathrm{G_0}}}       
\newcommand{\mmg}[2]{{\mbox{MMG}}{#1}{#2}}

\newcommand{\se}[1]{{\mbox{SE}} {#1}} 
\newcommand{\gl}[1]{{\mbox{GL}} {#1}}
\newcommand{\so}[1]{{\mbox{SO}} {#1}}
\newcommand{\su}[1]{{\mbox{SU}} {#1}}
\newcommand{\on}[1]{{\mbox{O}} {#1}}
\newcommand{\m}[1]{{\mbox{M}} {#1}}
\newcommand{\ad}[2]{{\mbox{Ad}}_{{#1}}{\left(#2\right)}}
\newcommand{\mse}[1]{{\mbox{MSE}} {\left( #1 \right) }}
\newcommand{\rd}[1]{{\operatorname{ROUND} } {#1}}

\newcommand{\E}[1]{\ensuremath{\mathbb{E}\left[#1\right] }}   
\newcommand{\R}{\ensuremath{\mathbb{R}}}
\newcommand{\C}{\ensuremath{\mathbb{C}}}
\newcommand{\eye}{\ensuremath{\mathrm{I}}}

\newcommand{\Rmnum}[1]{\expandafter\@slowromancap\romannumeral #1@}
\newcommand{\twopartdef}[4]
{
	\left\{
		\begin{array}{ll}
			#1 & \mbox{if } #2 \\
			#3 & \mbox{if } #4
		\end{array}
	\right.
}

\DeclareMathAlphabet{\mathcalligra}{T1}{calligra}{m}{n}
\DeclareFontFamily{OT1}{pzc}{}
\DeclareFontShape{OT1}{pzc}{m}{it}{<-> s * [1.10] pzcmi7t}{}
\DeclareMathAlphabet{\mathfrc}{OT1}{pzc}{m}{it}

\begin{document}

\title{Synchronization over Cartan motion groups via contraction}

\author[1]{Onur \"{O}zye\c{s}il \thanks{oozyesil@intechjanus.com}}
\author[2]{Nir Sharon\thanks{nsharon@math.princeton.edu} }
\author[3]{Amit Singer \thanks{amits@math.princeton.edu} }

\affil[1]{\footnotesize{INTECH Investment Management LLC, One Palmer Square, Suite 441, Princeton, NJ 08542, USA}}
\affil[2]{\footnotesize{Program in Applied and Computational Mathematics (PACM), Princeton University, Princeton, NJ 08544-1000}}
\affil[3]{\footnotesize{Department of Mathematics and PACM,  Princeton University, Princeton, NJ 08544-1000}}

\date{} 
\maketitle


\maketitle

\begin{abstract}
Group contraction is an algebraic map that relates two classes of Lie groups by a limiting process. We utilize this notion for the compactification of the class of Cartan motion groups, which includes the important special case of rigid motions. The compactification process is then applied to reduce a non-compact synchronization problem to a problem where the solution can be obtained by means of a unitary, faithful representation. We describe this method of synchronization via contraction in detail and analyze several important aspects of this application. We then show numerically the advantages of our approach compared to some current state-of-the-art synchronization methods on both synthetic and real data.
\end{abstract}


\section{Introduction} \label{sec:intro}

The synchronization (also known as registration or averaging) problem over a group is to estimate $n$ unknown group elements $\left\{ g_i \right\}_{i=1}^n$ from a set of measurements ${g}_{ij}$, potentially partial and noisy, of their ratios $g_ig_j^{-1}$, $ 1 \le i < j \le n$. The solution of this problem is not unique as the data is invariant to right action of an arbitrary group element $g$ (also termed global alignment) since $g_ig_j^{-1} =(g_ig)(g_jg)^{-1}$.

When all measurements are exact and noiseless, it is possible to solve the above synchronization problem as follows. Define an undirected graph $ \mathcal{G} = (\mathcal{V},\mathcal{E})$ with the set of vertices $\mathcal{V}$ corresponding to group elements $\left\{g_i\right\}_{i=1}^n$ and the edges representing the measurements. Then, if the graph is connected the elements are determined by traversing a spanning tree, up to an arbitrary group element in the root of the tree. Nevertheless, in the presence of noise, this solution accumulates errors while traversing the spanning tree. Since measurement errors naturally arise in real world problems, such as in dynamical system data \cite{sonday2013noisy}, in networks time distribution \cite{giridhar2006distributed}, or in Cryo-EM imaging \cite{singer2011three}, we henceforth focus on addressing the synchronization problem with noisy measurements.

We roughly divide the synchronization problem into two main cases, corresponding to two types of groups: compact groups and non-compact groups. In the case of compact groups, there exist a faithful, finite, and orthogonal representation of the group (that is an injective map to orthogonal elements of a finite dimensional linear space). Therefore, the problem can be reduced to one in linear algebra. For example, in~\cite{singer2011angular}, two main approaches are suggested: to define a proper (kernel) matrix and extract its (top) eigenvectors, or to solve an appropriate optimization problem. An extensive survey is given in Subsection~\ref{subsec:synch_compact_group}. While the study of synchronization over compact group has attracted significant attention in the literature, with a special attention to the special orthogonal groups, e.g. \cite{boumal2013robust, wang2013exact}, the case of non-compact groups has not received as much attention. A fundamental difficulty in the non-compact case is that many of the methods for compact groups are no more valid in their original form as there is no faithful, orthogonal representation available in a finite linear space. In addition, the Riemannian structure is more complicated. 

One important class of non-compact groups is the Cartan motion groups. These groups have the structure of a semidirect product $K \ltimes V $, where $K$ is a compact group and $V$ is a linear space (and thus the non-compactness). An important special case from the application standpoint is the special Euclidean group, denoted by $\se(d)$, and consists of oriented rotations and translations in $d$-dimensional Euclidean space, that is $\se(d) = \so(d) \ltimes \mathbb{R}^{d} $. This group, also known as the group of rigid motions, is ubiquitous in various research fields, and problems related to synchronization over this group appear in many areas, for example in sensor network localization~\cite{cucuringu2012sensor,peters2015sensor}, object alignment in computer vision~\cite{bernard2015solution}, estimation of location and orientation of mobile robots \cite{rosen2015convex}, the fresco reconstruction problem~\cite{brown2008system} and more. 

Synchronization over $\se(d)$ has gained considerable attention lately, see e.g., \cite{peters2015sensor, rosen2015convex, tron2009distributed}. Most work is concentrated on applications in $d=2$ and $d=3$, but some also refer to general $d$. One such general approach is to form a solution separately in the rotational part and the translational part of the data, see e.g. \cite{jiang2013global}. The linear part in those approaches is solved by least squares or consistency search. The consistency of the data in synchronization is the relation $g_{ij}g_{j\ell} = g_{i\ell}$ for noiseless measurements. Other solutions for arbitrary $d$ generalize the approaches in~\cite{singer2011angular}, either directly based on spectral approach \cite{arrigoni2015spectral, bernard2015solution} or by using modified optimization formulations and different relaxations to solve them, e.g., \cite{ozyesil2015stable}. There is a variety of solutions for the cases $d=2$ and $d=3$, as rigid motion synchronization can facilitate solving several classical problems such as Structure from Motion (SfM) in computer vision and pose graph optimization in robotics. Furthermore, some mathematical relations are only valid for these cases of low dimension, for example, the tight connection between $\se(3)$ and dual quaternions, which is used for synchronization in \cite{torsello2011multiview} to solve registration problems in computer graphics. A more comprehensive survey, including these cases and more, is given in Subsection~\ref{subsec:available_solutions_se_group}. Our approach to addressing the problem of synchronization over $\se(d)$ stems from the algebraic structure of Cartan motion groups and their relation to compact groups. Specifically, we use the notion of group contraction.  

The seminal paper~\cite{inonu1953contraction} introduced the notion of contraction of a Lie group. The motivation was to link two theories in physics which are related via a limiting process. The results show algebraically how the associated groups (and their representations) are linked, for example, the Galilean group and the Poincar\'e group. We use a specific family of contraction maps, presented in~\cite{dooley1983contractions}, which is based on Cartan decomposition of groups, and relates a group to its associated Cartan motion group. For example, in the case of the special Euclidean group, this contraction mapping acts from $\se(d)$ to $\so(d+1)$. These maps facilitate rewriting the synchronization over the Cartan motion group in terms of synchronization problem in its associated group. When the associated group is compact it means one can approximate the solutions of the original synchronization by solving synchronization over the compact group using any existing tools available. The procedure of reducing the problem by means of compact groups is also termed compactification.

The compactification implies a solution in the compact domain, which needs then to be mapped back to the original non-compact domain. The invertibility is not the only requirement on the map as we also want to guarantee the quality of the solution. To this end, we impose an additional algebraic condition on the maps which we term approximated homomorphism. This condition enables to relate the two problems, the compact and the non-compact one. Also, it gives rise to other variants for compactification; we provide one such example for the special case of $\se(d)$, where we use an invertible map for projecting data from $\se(d)$ to $\so(d+1)$. In contrary to contractions which are based on group level decomposition, this variant is based on a matrix decomposition over the group representations. 

Our main focus is on the application of the compactification methods to synchronization. Thus, we present a full, detailed algorithm for synchronization via contraction and analyze several important aspects of it. One such aspect is the invariance of synchronization to global alignment. Apparently, this invariant does not directly pass from the compact synchronization to the non-compact problem. We accurately characterize the correct invariance that the two problems share, and based on that, suggest a method to choose the right global alignment in the solution of the compact synchronization for maximizing the quality of solution in the non-compact side, after back mapping the solution. We also provide some analysis for the case of noisy data. In particular, we show a condition that guarantees, under certain assumptions on the noise, that meaningful information can be extracted to form a solution for synchronization over data contaminated with noise.

We conclude the paper with a numerical part, where we discuss the implementation of our method and compare it numerically with several other methods for synchronization. The examples are performed over two groups: the group of rigid motions and matrix motion group. We provide numerical examples for synthetic data where the noise model consists of multiplicative noise, outliers, and missing data. The results show that the contraction approach, which presents great flexibility in choosing the solver for synchronization over the compact group of rotations, leads to superior performance in almost any tested scenario. A second part is devoted to real data, coming from the problem of 3D registration, that is to estimate a 3D object from raw partial scans of it. This problem presents a realistic mix of noise and missing data and provides a solution that can be evaluated visually. For two different objects, in three different scenarios, synchronization via contraction exhibits superior performance. Furthermore, these results demonstrate that contraction in the group level outperforms contraction in the matrix level, and so justifies the use of algebraic structures.

The paper is organized as follows. In Section~\ref{sec:SyncBackground} we provide a survey on synchronization over groups, including the case of $\se(d)$. Then, in Section~\ref{sec:contraction} we introduce the Cartan motion groups and the notion of group contraction. We also provide a geometrical interpretation in low dimension and a linear variant of compactification on the matrix level. Section~\ref{sec:ApplicationForSync} presents how to apply compactification for solving synchronization and its analysis. In Section~\ref{sec:numerical_examples} we show numerical examples, starting with conventions about measuring error and noise and how to address several implementation issues such as parameter selection. Finally, we numerically compare the performances of various algorithms for synchronization using synthetic and real data. 

\section{Synchronization over groups -- background} \label{sec:SyncBackground}

We provide some background on the synchronization problem over groups. This problem is to estimate $n$ unknown elements $\left\{ g_i \right\}_{i=1}^n$  from a group $G$, given a set of ratio measurements $\{ {g}_{ij} \}_{1 \le i < j \le n} $, $g_{ij} \approx g_ig_j^{-1}$. This set might include only partial collection of measurements, outliers, and noisy ratio samples. As above, we denote the undirected graph of measurements by $\mathcal{G} = (\mathcal{V},\mathcal{E})$, where any edge $\left(i,j\right)$ in $\mathcal{E}$ corresponds to $g_{ij}$ and the set of vertices $\mathcal{V}$ corresponds to the unknown group elements $\left\{g_i\right\}_{i=1}^n$. We also assign weights $w_{ij}\ge 0$ to edges, where in a naive model these weights indicate whether $\left(i,j\right)\in \mathcal{E}$, that is $1$ if such edge exists and $0$ otherwise. In other more advanced models, these weights can be correlated with the noise or suggest our confidence in the measurements. 

\subsection{Synchronization over compact groups: a short survey and an example}  \label{subsec:synch_compact_group} 

The basic assumption in this part is that $G$ is a compact group, and thus having a faithful, finite dimensional and orthogonal (or unitary, for the complex case) representation $\rho \colon  G \rightarrow \gl(d,\R)$. Henceforth, we focus on the real case since the complex case (with representations in $ \gl({d,\C})$) is essentially similar. Having such $\rho$ means that solving synchronization over $G$ is equivalent to solving synchronization over the group of orthogonal transformations $\on(d)$, and we specifically focus on the case of $\so(d)$. 

The problem of synchronization of rotations is widely studied, since it naturally arises in many robotics (motion) applications and as a crucial step in solving the classical computer vision problem of Structure from Motion (SfM). A comprehensive survey about robust rotation optimization in SfM applications is done in \cite{tron2016survey}. Another extensive survey is done in \cite{carlone2015initialization} within the context of pose graph optimization (estimating the positions and orientations of a mobile robot). Since we ultimately plan to solve synchronization over some non-compact groups by mapping it into compact groups and solve it there, we shortly state several approaches for directly solving the synchronization of rotations. 

We begin with an example of a method that globally integrates the data of synchronization to obtain a solution. The full analysis of this method is given in \cite{boumal2014thesis, singer2011angular}.
\begin{example}[The spectral method] \label{exmple:EigenvectorsMethod}
The solution is constructed based upon the spectral structure of a specially designed (kernel) matrix $H$, defined as follows. Let $M$ be a $dn \times dn$ block matrix, consisting of the $d\times d$ blocks 
\begin{equation}   \label{eq:MeasurementMatrixM}
M_{ij} = \twopartdef { w_{ij}\rho \left( {g}_{ij} \right) } {(i,j) \in \mathcal{E}} {\mathbf{0}_{d\times d}} {(i,j) \notin \mathcal{E} . }
\end{equation}
Set $w_{ii}=1$ and $M_{ii} = \eye_d$, for all $1 \le i \le n$ and for consistency $M_{ji} =  M_{ij}^{T}$ since $\rho$ is orthogonal and so 
\[ (g_ig_j^{-1})^{-1} = (g_ig_j^{T})^{T} = g_jg_i^{T} = g_jg_i^{-1} .\] 
Namely, the orthogonality of our representation $\rho$ and the fact that $w_{ij}=w_{ji}$ make $M$ a self adjoint matrix. Denote by $\deg(i) = \sum_{j : (i,j)\in E} w_{ij} >0$ the total sum of edge weights from the $i^\text{th}$ vertex (the degree in $\mathcal{G}$) and define a diagonal $dn\times dn$ matrix $D$ where each $d\times d$ diagonal block $D_i$ is given by $D_i = \deg(i) \eye_d$. The kernel matrix is $H=D^{-1}M$.  

The fundamental observation of the spectral method is that in the noise-free case we get
\begin{equation}
\label{eq:NoiseFreeSynchronization}
M\mu = D\mu , \quad  \mu = \left[\begin{smallmatrix} \rho(g_1) \\ \vdots \\ \rho(g_n) \end{smallmatrix}\right]  .
\end{equation}
Namely, the columns of $\mu$ are $d$ top eigenvectors of $H$, corresponding to the eigenvalue $1$. In addition, based on the non-negativity of the weights $w_{ij}$, the matrix $H$ is similar to the Kronecker product of the normalized Laplacian matrix of the graph and the identity matrix $I_d$, using the conjugate matrix $D^{-1}A$, with $A$ being a block diagonal matrix with the block $\rho(g_i)$, $i=1,\ldots,n$. The last step in the spectral approach is a rounding procedure, since in the noisy case there is no guarantee that each $\mu_i$ is in the image of $\rho$. For a pseudocode see Algorithm~\ref{alg:eigenvector_method} on supplementary material in Section~\ref{app:Supplements}.
\end{example}

As seen in Example~\ref{exmple:EigenvectorsMethod} the structure of the data suggests a spectral solution. Yet a different way to integrate the data globally is by solving a semidefinite programming (SDP) problem, e.g., \cite{arie2012global, saunderson2014semidefinite, saunderson2015semidefinite, singer2011angular}. Global integration can be also exploited for the case of outliers in the data, where the cost function is tailored to be more robust for inconsistent data \cite{hartley2011l1, wang2013exact}. It is also possible to attain robustness using maximum likelihood approach \cite{boumal2013robust}. Another closely related approach is to use the Riemannian structure of the manifold of rotations to define iterative procedures, e.g., \cite{chatterjee2013efficient, hartley2013rotation, moakher2002means}. A different approach is related to the cyclic features of the data. Specifically, the consistency of the data states that for noiseless measurements $g_{ij}g_{j\ell} = g_{i\ell}$. This information can be used to either form a solution or to discard outliers, see e.g., \cite{sharp2004multiview, zach2010disambiguating}. Many other approaches have been designed for low dimensional rotations, for example with $d=3$ one can exploit the relation between rotations and quaternions, e.g., \cite{govindu2001combining}. These low dimensional cases are important as they appear in many applications. A further popular approach for practitioners is to mix between several methods to have a superior unified algorithm. These topics are well covered in the abovementioned surveys \cite{carlone2015initialization,tron2016survey}.

\subsection{An overview of synchronization over the special Euclidean group} \label{subsec:available_solutions_se_group}

In the main focus of this paper is the class of Cartan motion groups. This class of groups, which will be precisely defined in the sequel section, includes the important group of rigid motions, that is the special Euclidean group $G = \se(d)$ of isometries that preserve orientation. We naturally identify $\se(d)$ by the semidirect product $\so(d)\ltimes \R^d$ and multiplication is given by 
\begin{equation} \label{eqn:SE_action}
\left(\mu_1,b_1\right)\left(\mu_2,b_2\right) = \left(\mu_1 \mu_2, b_1+\mu_1 b_2 \right) , \quad \left(\mu_1,b_1\right),\left(\mu_2,b_2\right) \in \so(d)\ltimes \R^d .
\end{equation}
The above description is also reflected by the $d+1$ dimensional representation $\rho : \se(d) \rightarrow \gl(d+1,\R)$, 
\begin{equation} \label{eqn:representationSEd}
\rho \left( g \right) = \left[ \begin{smallmatrix} \mu & b \\ \mathbf{0}_{1\times d} & 1 \end{smallmatrix}\right] , \quad g = \left(\mu,b\right) \in \mbox{SO}(d)\ltimes \R^d . 
\end{equation}      
Henceforth, we will identify the group elements with their matrix representation \eqref{eqn:representationSEd}. 

Next we briefly survey some of the main approaches used today to solve the so called rigid motion synchronization. A fundamental issue is the more intricate Riemannian structure. This leads, for example, to extra effort in working with the exponential and logarithm maps. Another issue with synchronization over non-compact group, and in particular in rigid motion synchronization, is that there is no faithful, finite dimensional and orthogonal representation. This fact complicates most methods for rotation synchronization. For example, for the spectral method of Example~\ref{exmple:EigenvectorsMethod}, the affinity matrix $M$ defined using \eqref{eqn:representationSEd} is no longer self adjoint, which makes the spectral structure of the kernel matrix less accessible. Nevertheless, there is a way to address this difficulty, by considering a Laplacian operator and then to reveal its null space using the SVD decomposition. This spectral method is studied both in \cite{arrigoni2015spectral} and in \cite{bernard2015solution}, where it was shown to be an efficient and accurate method.

Optimization also plays a significant role in solving rigid motion synchronization. Indeed, the search space is non-compact and exponential in $n$, and thus it requires either relaxing or reducing the original problem into a more tractable one. A way to do so is by separating the optimization over the compact and non-compact parts of the group, e.g., \cite{jiang2013global}. One different motivation to use separation appears when considering a quadratic loss function, taken from the compact group case \cite{singer2011angular}
\begin{equation}
\label{eq:NonCompactOptimPrb1}
\begin{aligned}
& \underset{{\scriptstyle\{g_i\}_{i=1}^n} \subset G}{\text{minimize}}
& & \sum_{(i,j)\in \mathcal{E}} w_{ij} \left \| \rho(g_i g_j^{-1}) - \rho(g_{ij}) \right \|_F^2 , \\
\end{aligned}
\end{equation}
where $\norm{\cdot}_F$ stands for the Frobenius norm of a matrix. Using the representation \eqref{eqn:representationSEd} we can rewrite~(\ref{eq:NonCompactOptimPrb1}) as
\begin{equation} \label{eqn:NonCompactProbSimplified}
\begin{aligned}
& \underset{{\scriptstyle\{{\mu}_i, {b}_i\}_{i=1}^n}}{\text{minimize}}
& & \sum_{(i,j)\in \mathcal{E}} w_{ij} \left \| {\mu}_i {\mu}_j^T - {\mu}_{ij} \right \|_F^2 + \sum_{(i,j)\in \mathcal{E}} w_{ij} \left \| {b}_i - {\mu}_i {\mu}_{j}^T {b}_j -  {b}_{ij} \right \|_2^2 \\
& \text{subject to}
& & {\mu}_i^T{\mu}_i = \eye_d, \quad \det ({\mu}_i) = 1, \quad {b}_i\in \R^d,\quad i = 1,\ldots,n .
\end{aligned} 
\end{equation}
The general approach of separation suggests to first minimize the first sum, using for example the spectral method, and then solve the least squares problem of the second sum with the estimated $\{\mu_i\}_{i=1}^n$, as done in \cite{cucuringu2012sensor}. Since the representation of $\se(d)$ is not orthogonal, there is no canonical way of fixing the cost function in the rotations case. In fact, changing the cost function of \eqref{eq:NonCompactOptimPrb1} leads to several interesting solutions. One naive change is to consider 
\begin{equation} \label{eqn:SeparetedCostFunction}
\sum_{(i,j)\in \mathcal{E}} w_{ij} \left \| \rho(g_i^{-1} g_{ij} g_j) - \eye_{d+1} \right \|_F^2 , 
\end{equation}
which yields to the completely separated cost function 
\begin{equation}  \label{eqn:SeparetedCostFunctio2}
\sum_{(i,j)\in \mathcal{E}} w_{ij} \left \| {\mu}_i{\mu}_j^T  - {\mu}_{ij} \right \|_F^2 + \sum_{(i,j)\in \mathcal{E}} w_{ij} \left \| {b}_i - {\mu}_{ij}{b}_j -  {b}_{ij} \right \|_2^2  .
\end{equation}
In this case the separated cost function has an undesirable property -- it is not invariant to right multiplication by an arbitrary group element. The loss of invariance is clear directly from \eqref{eqn:SeparetedCostFunction} but also in \eqref{eqn:SeparetedCostFunctio2} as the translational part of such global alignment results in additional terms in the summand. However, this solution is potentially more suitable for problems where we need to introduce extra constraints for translation parts (the $b_i$-s), e.g. in the fresco reconstruction problem, where one can introduce extra constraints to prevent overlapping of the fresco fragments \cite{funkhouser2011learning}. Another variant is the least unsquared deviations (LUD) which uses 
\[ \sum_{(i,j)\in \mathcal{E}} w_{ij} \left \| \rho(g_i g_j^{-1}) - \rho(g_{ij}) \right \|_F .\]
This is done in \cite{wang2013exact} to face the problem of outliers. Yet another cost function is
\[  \sum_{(i,j)\in \mathcal{E}} w_{ij} \left \| \rho(g_i) - \rho(g_{ij}g_j) \right \|_F^2 =  \tr \left( \mu^T \mathcal{L} \mu \right) , \]  
where $\mu$ is the vectors of group elements (matrices) as in \eqref{eq:NoiseFreeSynchronization}, and $\mathcal{L} = (I-M)(I-M^T)$ where $M$ is the measurements matrix as defined in \eqref{eq:MeasurementMatrixM}. The matrix $\mathcal{L} $ can be also considered as an extension of the ``twisted Laplacian", as defined in~\cite{mantuano2007discretization}. This cost functions brings us back to the spectral method, as the operator $\mathcal{L}$ is  is related to the Laplacian operator in \cite{arrigoni2015spectral} and \cite{bernard2015solution}. Note that both papers \cite{arrigoni2015spectral, bernard2015solution} focus on the case $d=3$ and relax the constraint $\mu_i \in \rho\left(G  \right)$ by forcing the solution $\mu$ to coincide with $\left(0,0,0,1\right)$ in each fourth row and then project the upper left block of each $\mu_i$ on $\so(3)$. Another interesting relation, this time with the ``Connection Laplacian" is established in \cite{briales2017cartan}, where a Maximum Likelihood Estimation (MLE) is constructed for synchronization over $\so(d)$.

As in rotation synchronization, there are a few additional approaches, motivated by the SfM problem. Similar to the above spectral and optimization methods, the following solutions do not use separation, and integrate the data concurrent in both parts of the rigid motion. One method solves the optimization problem related to the SfM with rigid motions by using semidefinite relaxation (SDR) \cite{ozyesil2015stable}. Relaxing the constraints of $\se(3)$ by its convex hull was considered in \cite{rosen2015convex}. Note that some of the solutions related to the SfM problem are specificaly designed to operate under the assumption that $d=3$, and their generalizations to higher dimensions are not always trivial, if possible; one example of such case is in \cite{torsello2011multiview} where the authors use the unique relation between $\se(3)$ and quaternions, in order to use diffusion tools on the graph. The explicit structure of the exponential map and its inverse for $\se(3)$ is also utilized in \cite{govindu2004lie} to define an iterative averaging algorithm. The same explicit form of the exponential map helps in defining the cyclic constraints that lead to an optimized solution in \cite{peters2015sensor}. Iterative optimization that also exploits the graph structure is done in \cite{tron2009distributed} by a distributed algorithm for sensor network orientation. Riemannian gradient descent adapted for $\se(3)$ is also used for rigid motion synchronization under two different variants in \cite{tron2014statistical, tron2014distributed}.
 
\section{Compactification of Cartan motion groups}   \label{sec:contraction} 
 
We introduce the class of Cartan motion groups and then describe the notion of group contraction and compactification, including two specific examples of compactifications in the group level and in the matrix level.

\subsection{Cartan motion groups}

We start with a brief introduction to Cartan decomposition and its induced Cartan motion groups. This prefatory description is tailored for our ultimate goal of addressing the problem of synchronization for such groups via contraction maps. For more information see \cite[Chapter 1.2]{vilenkin2012representation}, \cite{dooley1985contractions}, and references therein.

Let $\Gb$ be a semisimple compact Lie group. It is worth mentioning that algebraically, most of the following can be carried for non-compact groups with finite center instead of $\Gb$, but this case is not that useful for us as will be clear soon. Let $\mathfrc{g}$ be the corresponding Lie algebra of $\Gb$, in the compact case we refer to $\mathfrc{g}$ as a real form of the complexification of the Lie algebra. Then, there is a Cartan decomposition $\mathfrc{g} = \mathfrc{t} \oplus \mathfrc{p}$ such that for an associated Cartan involution, $\mathfrc{t}$ and $\mathfrc{p}$ are the eigenspaces corresponding to eigenvalues of $+1$ and $-1$, respectively. As such, the linear space $\mathfrc{t}$ satisfies $\comm{\mathfrc{t}}{\mathfrc{t} } \subset \mathfrc{t}$ and thus is a Lie subalgebra, where in general $\mathfrc{p}$ is just a linear space. We denote by $K$ the Lie subgroup corresponding to the subalgebra $\mathfrc{t}$ and restrict the discussion to the real case, so we define $V=\mathfrc{p}$ to obtain the Cartan motion group, 
\[ \Gc = K \ltimes V . \] 
The action of the group is 
\begin{equation} \label{eqn:Cartan_action}
 (k_1,v_1)(k_2,v_2) = \left(k_1k_2,v_1+\ad{k_1}{v_2} \right) , 
\end{equation}
where $\ad{\cdot}{\cdot}$ refers to the adjoint representation on $\Gc$. Note that $\mathfrc{p}$ is closed under the adjoint, that is $\ad{k}{v} \in \mathfrc{p}$ for all $k\in K$ and $v \in \mathfrc{p}$. Before proceeding, we provide some examples for the above settings. 
\begin{example}[Euclidean group] \label{exmple:SE}
Let $\Gb = \so(d+1)$ be the group of all orthogonal matrices of order $d+1$ with positive determinant. This Lie group has the Lie algebra $\mathfrc{g}$ consists of all real square matrices of size $d+1$, which are skew symmetric, $X+X^T = 0$. To form the decomposition above, consider the matrices $E(i,j)$, having $1$ on the $(i,j)$ entry and zero otherwise. Thus, $\{ E(i,j)- E(j,i) \}_{1\le i <j \le d+1}$ is a basis of $\mathfrc{g}$. The span of $\{ E(i,j)- E(j,i) \}_{1\le i <j \le d}$ forms the Lie algebra $\mathfrc{t}$, which is isomorphic to the Lie algebra of $\so(d)$. Indeed, we can identify $\so(d)$ with $K$, the subgroup of $\Gb$ of all matrices having $e_{d+1}$ as their last column, where $e_i$ is the unit vector that has one on the $i^\text{th}$ component. The complement of $\mathfrc{t}$ is $\mathfrc{p}$, which is spanned by $\{ E(i,d+1)-E(d+1,i) \}_{1\le i \le  d}$. The linear space $\mathfrc{p}$ is easily identified with $\mathbb{R}^d$ via the map 
\begin{equation} \label{eqn:anti_sym_map_to_Euclidean}
E(i,d+1)-E(d+1,i) \mapsto e_i . 
\end{equation}
Thus, we end up with the Cartan motion group $\so(d) \ltimes \mathbb{R}^d$, that is $\se(d)$. More on this example can be found in \cite{dooley1983contractions} and in the next section.
\end{example}
\begin{example}
Consider $\Gb=\su(d)$, the group of all unitary matrices with determinant equal to one. Although its definition relies on complex matrices, it is common to consider it as a real Lie group. There are three main types of Cartan decompositions, we mention one that corresponds with the Cartan involution of complex conjugation. As such, it divides the Lie algebra $\{ X \mid X+X^\ast=0,\quad \tr(X)=0\} $ to a real part, which is the same Lie algebra of $\so(d)$, consisting of all skew symmetric matrices. Its orthogonal complement, denoted by $W = \so(d)^\perp$, can be described as the direct sum of a $d-1$ dimensional subspace of pure complex, diagonal matrices with zero trace, and a $\frac{d(d-1)}{2}$ dimensional subspace of pure complex matrices of the form $X_{i,j}=X_{j,i}$, with zero diagonal. Then, the Cartan Motion group in this case is $\so(d) \ltimes W$.
\end{example}
\begin{example}[matrix motion group] \label{exm:MMG}
Let $\Gb=\on(d+\ell)$ be the group of all real orthogonal matrices of order $d+\ell$. Denote by $\m(d,\ell)$ the space of all real matrices of order $d \times \ell$. Then, one Cartan decomposition of $\mathfrc{g}$ can be identified with the product of the Lie algebras of $\on(d)$ and $\on(\ell)$ (namely $\mathfrc{t}$) and the space $\mathfrc{p} = \m(d,\ell)$. This decomposition yields the so called matrix motion group, $\Gc = \left(\on(d) \times \on(\ell) \right) \ltimes\m(d,\ell)$. This particular example draws some attention as this Cartan motion group is associated with the quotation space $\Gb/K$ which in this case is the Grassmannian manifold, consisting of all $d$ dimensional linear subspaces of $\mathbb{R}^{d+\ell}$. For more details see \cite[Pages 107--128]{olafsson2008radon}. We return to this example in Subsection~\ref{subsec:MMG} where we describe synchronization over this group and some of its practical considerations.
\end{example}

\subsection{Group contraction map} \label{subsec:group_contraction}

Lie group contraction is a process of obtaining one Lie group from a sequence of other Lie groups. This concept was developed primary in the seminal paper~\cite{inonu1953contraction} to describe the relation between groups of two physical theories linked by a limiting process (for example, the inhomogeneous Lorentz group is a limiting case of the de Sitter groups, more details on the theory of contraction can be found in e.g., \cite{gilmore2012lie}). In this paper we focus on one family of contraction maps, established in~\cite{dooley1983contractions,dooley1985contractions}, for semisimple Lie groups. 

Following~\cite{dooley1983contractions} we concentrate on the family of contraction maps $\{ \Psi_{\lambda} \}_{\lambda>0}$, which are smooth maps $\Psi_{\lambda} \colon \Gc \mapsto \Gb$ and defined by
\begin{equation} \label{eqn:cartan_contraction}
\Psi_{\lambda}(k,v) = \exp(v/\lambda) k , \quad (k,v) \in \Gc
\end{equation}
where the exponential is the Lie group exponential map of $\Gb$ (note that indeed $v/\lambda \in \mathfrc{g} $). The group $\Gc$ is commonly referred to as a contraction of $\Gb$. 

The parameter $\lambda$ in~\eqref{eqn:cartan_contraction} has a major role as it reflects, in a sense, how much significant information from $\mathfrc{p} $ passes to $\Gb$. In the extreme case of $\lambda \rightarrow \infty$, the contraction map $\Psi_{\lambda}$ simply maps the group $K$ to itself as a subgroup of $\Gb$. In the opposite scale of $\lambda$ where it is closer to zero, the map becomes less practical for two main reasons. First, the vector $v/\lambda$ is likely to be very large and lies outside the injectivity radius of the exponential map of $\Gb$, namely the exponential map is no longer a diffeomorphism. This means, for example, that we cannot apply the (correct) inverse of $\Psi_{\lambda}$. For synchronization, the ability to apply the inverse is crucial. Second, the exponential map depends on the curvature of the Lie group manifold (given algebraically by a commutator term). Thus, small $\lambda$ values lead to larger vectors which are more exposed to the effects of curvature and result in large distortions of the mapped elements. Subsequently, and without loss of generality we consider $\lambda$ to be restricted in $[1,\infty)$. Note that the global Cartan decomposition of the group $\Gb$ is defined by \eqref{eqn:cartan_contraction} with $\lambda=1$. 

The first two algebraic properties which we require from the contractions are 
\begin{equation}  \label{eqn:app_homo}
  \begin{aligned}  
 \Psi_{\lambda}\left(g^{-1}\right) = \left(\Psi_{\lambda}(g)\right)^{-1} , \qquad g \in \Gc , \\
 \norm{ \Psi_{\lambda}(g_1g_2) - \Psi_{\lambda}(g_1)\Psi_{\lambda}(g_2) }_F = \mathcal{O}\left( \frac{1}{\lambda^2} \right) \ , \quad g_1,g_2 \in \Gc ,
  \end{aligned}
\end{equation}
where $\mathcal{O}(\cdot)$ means that $\mathcal{O}(x)/x$ is bounded for small enough $x$, and the group elements of $\Gb$ are identified with their orthogonal matrix representation. In addition, the coefficient of the term $\frac{1}{\lambda^2}$ of the second equation is independent of $\lambda$, but (as we will see next) can be a function of $g_1$ and $g_2$. We term the properties in \eqref{eqn:app_homo} ``approximated homomorphism". Note that while we allow some approximation on the homomorphism of the action of the group, we require exactness in the inverse. For our family of contraction maps, these requirements are satisfied.
\begin{prop} \label{prop:app_homo_contraction}
Let $v_1$ and $v_2$ be the translational parts of $g_1$ and $g_2$. In addition, assume $\lambda$ is large enough so $\norm{v_1}+\norm{v_2} \le \lambda r$, with $r = 0.59 $. Then, the contractions defined in~\eqref{eqn:cartan_contraction} are approximated homomorphisms.
\end{prop}
\begin{proof}
Let $g=(k,v)\in \Gc$, then by~\eqref{eqn:Cartan_action} we have $g^{-1} = (k^{-1},\ad{k^{-1}}{-v})$. Therefore, using the linearity of the adjoint and its commutativity with the exponent we have
\[ \begin{aligned}[t]
	\Psi_{\lambda}\left(g^{-1}\right)   &=  \exp(\ad{k^{-1}}{-v}/\lambda)k^{-1}    \\
            &=\ad{k^{-1}}{\exp(-v/\lambda)}k^{-1}  \\
	        &= k^{-1} \exp(-v/\lambda) = \left(\exp(v/\lambda))k\right)^{-1} = \left(\Psi_{\lambda}(g)\right)^{-1} . \\
      \end{aligned} . \]
For the second part, let $g_1=(k_1,v_1),g_2=(k_2,v_2) \in \Gc$. By~\eqref{eqn:Cartan_action}
\[  \norm{ \Psi_{\lambda}(g_1g_2) - \Psi_{\lambda}(g_1)\Psi_{\lambda}(g_2) }_F = 
\norm{\exp(v_1+\ad{k_1}{v_2}/\lambda)) k_1k_2 - \exp( v_1/\lambda)k_1 \exp(v_2/\lambda)k_2}_F .  \]
Since we identify each component with its orthogonal representation in $\Gb$, we can easily cancel $k_2$ from the right hand side. Moreover, by multiplying by $k_1^{-1}$ from the right and using again the adjoint commutativity with the exponential we get 
\[ \exp(v_1/\lambda)k_1 \exp(v_2/\lambda)k_1^{-1} =  \exp(v_1/\lambda) \ad{k_1}{\exp(v_2/\lambda)} = \exp(v_1/\lambda) \exp(\ad{k_1}{v_2/\lambda}).  \]
The Zassenhaus formula (see e.g., \cite{casas2012efficient}) indicates that
\[ \exp((v_1+\ad{k_1}{v_2})/\lambda)= \exp( v_1/\lambda) \exp( \ad{k_1}{v_2}/\lambda)\exp(-\frac{1}{2} \comm{v_1/\lambda}{\ad{k_1}{v_2}/\lambda})\left(1+ \mathcal{O}(1/\lambda^3) \right). \]
Note that the convergence of the above product of the Zassenhaus formula is guaranteed when $\norm{v_1}+\norm{\ad{k_1}{v_2}} \le \lambda r$, with $r\approx 0.59 $, see \cite{casas2012efficient}. This is always true for large enough $\lambda$, as assumed by the claim. Also, since $k_1$ has a representation as an orthogonal matrix, it holds that $\norm{\ad{k_1}{v_2}} = \norm{v_2}$. For simplicity denote $c=\comm{v_1/\lambda}{\ad{k_1}{v_2}/\lambda}$. Then, since the norm is invariant to the common exponential terms we are left with
\begin{equation} \label{eqn:appHomProf}
\norm{I-
\exp(-\frac{1}{2}c)\exp(\frac{1}{6} \left(2 \comm{\ad{k_1}{v_2}/\lambda}{c} -  \comm{v_1/\lambda}{c} \right)  )\left(1+ \mathcal{O}(1/\lambda^4) \right)   } . 
\end{equation}
Consider the Taylor expansion of the matrix exponential map $\exp(A) = I + A + \frac{A^2}{2} + \ldots$, and recall that $c  =  \lambda^{-2} \left(v_1\ad{k_1}{v_2}-\ad{k_1}{v_2}v_1\right) $, to conclude that the leading order in \eqref{eqn:appHomProf} is indeed $\lambda^{-2}$.
\end{proof}

\subsection{The Euclidean group case}    \label{subsec:EuclideanGroupCase}  

The case of $\Gc = \se(d)$ has already been shown in Example~\ref{exmple:SE} to be a special case of Cartan motion groups. Note that the general action \eqref{eqn:Cartan_action}, where $\ad{}{\cdot}$ is taken as similarity transformation in the matrix representation of $\Gb=\so(d+1)$, is identical to the explicit action of the group as given in \eqref{eqn:SE_action}, under the standard matrix representation in $\Gc$. 

The contraction maps in $\se(d)$ take the form
\begin{equation}
 \Psi_{\lambda}\left( \mu, b \right) = \exp(\frac{1}{\lambda}b)\mu, \quad  \lambda \ge 1, 
\end{equation}
where $\mu \in \so(d)$ and $\so(d)$ is identified as a subgroup of $\Gb$ given by $ \{Q \in \so(d+1) \mid Qe_{d+1}=e_{d+1}   \}$. In addition, the exponential map is in the group $\Gb$ and thus $b \in \mathbb{R}^d $ is understood in this context as a matrix. To be specific, recall~\eqref{eqn:anti_sym_map_to_Euclidean}, then we refer to the map $b \mapsto \left[ \begin{smallmatrix} \mathbf{0}_{d\times d}  & b \\-b^T & 0\end{smallmatrix} \right]$. Further examination of the latter matrix reveals that for $b\neq 0$ its rank is $2$, that is the kernel is of dimension $d-1$. In other words, the spectrum consists of $d-1$ zero eigenvalues. The remaining two eigenvalues are pure complex (and conjugated as the matrix is real skew symmetric) and equal to $ \pm \sqrt{-b^Tb} $ since
\begin{equation} \label{eqn:eigenvalueofB}
 \left[ \begin{smallmatrix} \mathbf{0}_{d\times d}  & b \\-b^T & 0\end{smallmatrix} \right] \left[ \begin{smallmatrix} b \\ \pm \sqrt{-b^Tb} \end{smallmatrix} \right] = \pm \sqrt{-b^Tb} \left[ \begin{smallmatrix} b \\ \pm \sqrt{-b^Tb} \end{smallmatrix} \right]. 
\end{equation}
The spectrum is important since the exponential function as a matrix function can be interpreted as a scalar function on the spectrum. Therefore, injectivity of the exponential is guaranteed when $\norm{b} = \sqrt{b^Tb}$ is corresponding to the principal branch of the logarithm function, that is for
\begin{equation} \label{eqn:lambda_cond}
\norm{b}/\lambda < \pi .
\end{equation}  
To have some intuition regarding the nature of our contraction map $\Psi_\lambda$ in $\se(d)$, it is interesting to consider the geometrical interpretation of the special cases $\se(1)$ and $\se(2)$. The following interpretation is accompanied by a visual illustration given in Figure~\ref{fig:VisualPsiOnSE2}. 

The group $\se(1)$ is isomorphic to the real line as the rotations are only the scalar $1$ and the translational part is of dimension $1$. Therefore, $\Psi_\lambda$ is just $\exp(\left[ \begin{smallmatrix} 0  & b \\-b & 0\end{smallmatrix} \right])$, $b\in \mathbb{R}$ which basically maps the scaled number $b/\lambda$ to a point on the unit sphere in $\mathbb{R}^2$ by an angular distance of $b/\lambda$ from the ``east pole" $(0,1)$. Note that \eqref{eqn:lambda_cond} ensures that we cannot approach the west pole and thus we remain in the injectivity domain on the sphere. This interpretation is illustrated in Figure~\ref{fig:VisualPsiOnSE2} (left panel). 

For $\se(2)$, $\Psi_\lambda$ maps any $g = \left( \mu, b \right)$ to an element of $\so(3)$. This element is equivalent to first applying the rotation $\mu$ around the $z$-axis (recall the subgroup $\so(2) \cong \{Q \in \so(3) \mid Qe_{3}=e_{3}   \}$) followed by rotation in the ``direction" of the matrix $b$ with a geodesic distance of $\norm{b}/\lambda$. This characterization, also known as the misorientation axis and angle, helps us interpret geometrically this exponential map as mapping a point from a tangent plane $\{ (x,y,1) \colon x,y\in \mathbb{R}  \}$ on the ``north pole" $(0,0,1)$ of the sphere $S^2$ into $S^2$. Note that while the usual representation of a rotation in $\so(3)$ requires a unit vector in $\mathbb{R}^4$ (unit quaternion), in this case since $\exp(b/\lambda)$ is determined by only $2$ coordinates we can identify any such matrix with a unit vector in $\mathbb{R}^3$. Therefore, we can think about the mapping $\Psi_\lambda(g)$ as sending the sufficiently small, scaled translational part $b/\lambda$, that is a point close to the origin in $\mathbb{R}^2$, to a corresponding point on the sphere $S^2$ close to its north pole. In this case, the illustration is given in Figure~\ref{fig:VisualPsiOnSE2} (right panel), which schematically demonstrates the mapping. Note that we do not described the full contraction of the Cartan motion group which should also include the product with the $K$ component, for the clearness of the figure.

\begin{figure}
\centering    
\captionsetup[subfloat]{labelformat=empty}
\subfloat[$\se(1) \cong \mathbb{R} \mapsto \so(2) \cong S^1$]{\label{subfig:GI1}    \includegraphics[width=0.33\textwidth]{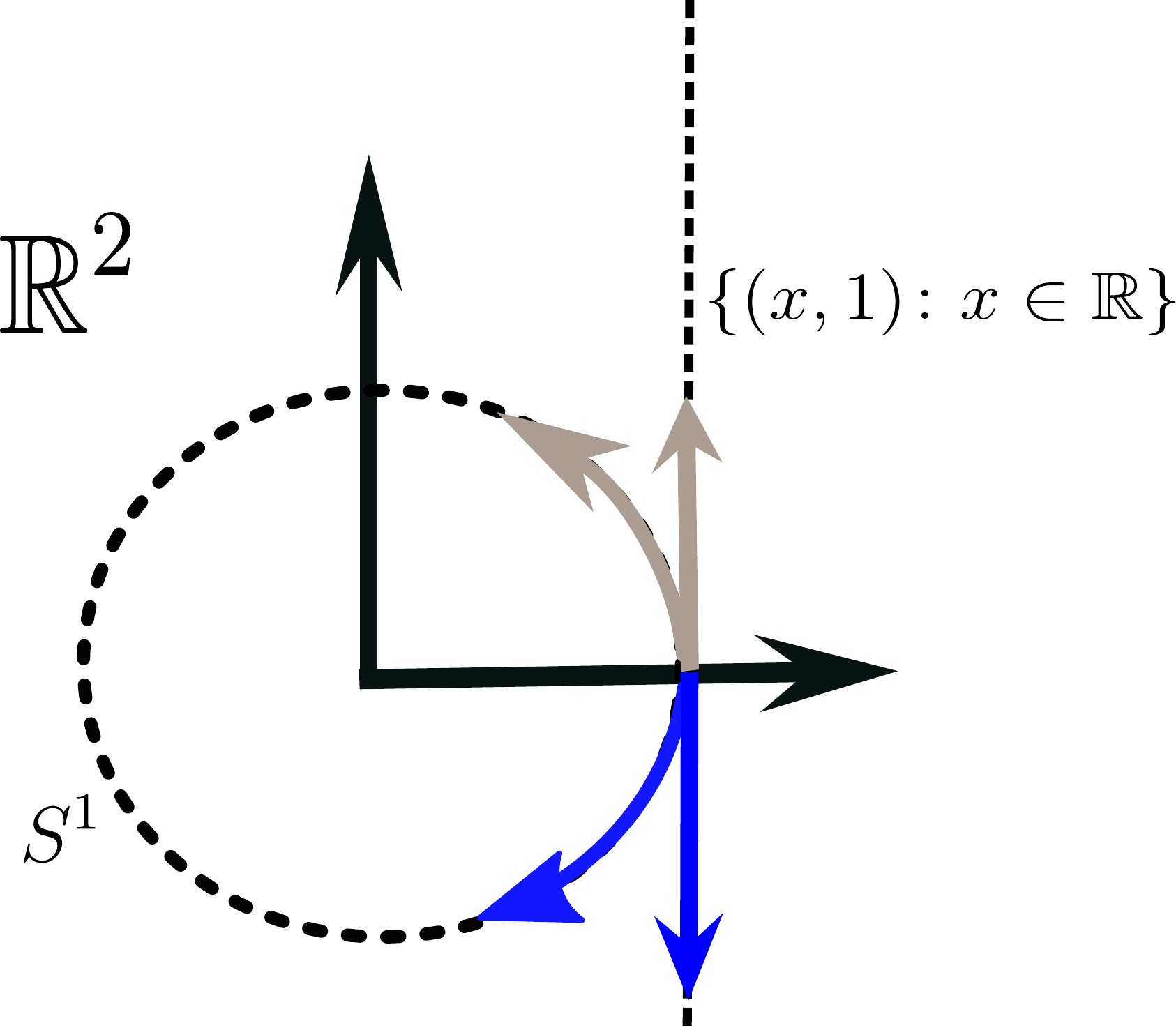}} \qquad \qquad
\subfloat[$\se(2) \mapsto \so(3) \cong \so(2) \times S^2$ ]{	\label{subfig:GI2}   \includegraphics[width=0.38\textwidth]{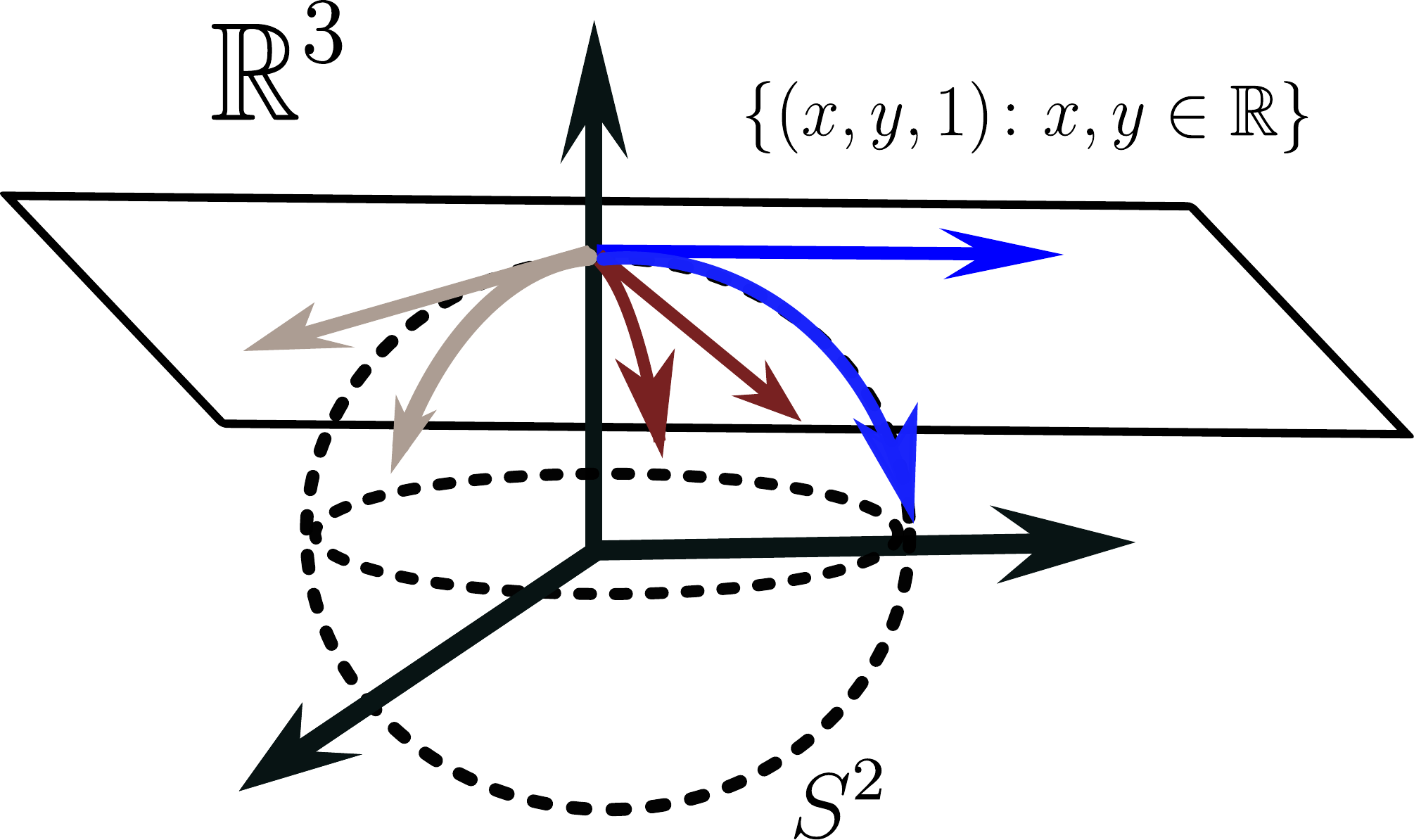}} 
\caption{A geometrical interpretation of  $\Psi_\lambda$ on $\se(1)$ and $\se(2)$ --- mapping scaled translational parts to the spheres.}
\label{fig:VisualPsiOnSE2}
\end{figure}

\subsection{A linear algebra variant -- matrix level projections} \label{subsec:matrix_projection}
 
In previous parts we described how one can exploit the group structure using a group decomposition (Cartan) for mapping $\se(d)$ to $\so(d+1)$. In this part, we aim to use the same mapping notion, but on the matrix level and by a matrix decomposition. Precisely, our second compactification method (valid for the special Euclidean group) is based upon the approximation of the matrices representing elements of $\se(d)$ by the orthogonal part of their polar decomposition (PD). 

Approximating elements from $\se(d)$ by elements from $\so(d+1)$ is not new. In the context of motion synthesis and the design of approximate bi-invariant metrics on $\se(d)$, the cases $d = 2$ and $d = 3$ were studied in~\cite{mccarthy1983planar} using Euler angles and geometrically intuitive ideas. The related following works~\cite{etzel1996metric,larochelle1995planar} use Taylor expansions. These methods are also related to the contraction maps discussed above, however since they are based on geometrical notions and Taylor series of special functions, it is difficult to generalize them further to other groups. The projections based on PD were introduced in~\cite{larochelle2004svd,larochelle2007distance}, and can be considered as alternative generalizations to the approximations in~\cite{mccarthy1983planar}, for the study of approximate invariant metrics in $\se(d)$. 

To be aligned with the contraction maps, we define a family of PD-based projections (we also refer to them broadly as contractions) $\{ \Phi_{\lambda} \}_{\lambda\ge 1}$, which are maps $\Phi_\lambda \colon \se(d) \rightarrow \so(d+1)$ that are defined as
\begin{equation} \label{eqn:matrix_projection}
\Phi_{\lambda}\left(  \left[ \begin{smallmatrix} \mu & b \\ \mathbf{0}_{1\times d} & 1\end{smallmatrix} \right]  \right) = U_{\lambda} ,
\end{equation}
where $U_{\lambda}$ is the orthogonal part of the PD decomposition $g_{\lambda} = U_{\lambda}H_{\lambda}$ ($H_{\lambda}$ is symmetric positive definite) of
\[g_{\lambda} = \left[ \begin{smallmatrix} \mu & b / \lambda \\ \mathbf{0}_{1\times d} & 1\end{smallmatrix} \right] .\] 
As in the contraction maps of groups, $\lambda$ is used as a parameter for the scale of the translational part. In addition, as analogue for the contraction maps, for large $\lambda$ the map $\Phi_{\lambda} \left( g \right)$ tends to $ \left[ \begin{smallmatrix} \mu & \mathbf{0}_{d\times 1} \\ \mathbf{0}_{1\times d} & 1\end{smallmatrix} \right]$ which means $\se(d)$ is mapped into $\so(d)$ (as a subgroup of $\so(d+1)$).

The term projection is used in the context of $\Phi_\lambda$ since in the space of square matrices of size $d+1$, under the metric induced from the Frobenius norm, we have that $\Phi_\lambda \left(\left[ \begin{smallmatrix} \mu & b / \lambda \\ \mathbf{0}_{1\times d} & 1\end{smallmatrix} \right] \right)$ is the closest element in $\so(d+1)$ to $g_\lambda$. This is true since the closest orthogonal matrix is given by $V_{\lambda} U_{\lambda}^T $ where  $g_\lambda = V_{\lambda}\Sigma_{\lambda} U_{\lambda}^T$ is the SVD decomposition, but this is exactly $U_\lambda$ from the PD decomposition. Next, we describe $\Phi_\lambda$ in more detail.

\begin{prop} \label{prop:detailed_phi}
Let $0\neq b \in \mathbb{R}^d$, and let $P$ be an orthogonal projection matrix for the subspace 
\begin{equation} \label{eqn:subspacePDmethod}
  \{ v\in \mathbb{R}^d \mid v^T b=0 \} . 
\end{equation}
Then, the maps $\Phi_\lambda$ of~\eqref{eqn:matrix_projection} are in $\so(d+1)$ and of the form
\[ \Phi_{\lambda}\left(  \left[ \begin{smallmatrix} \mu & b \\ \mathbf{0}_{1\times d} & 1\end{smallmatrix} \right]  \right) =  \left[\begin{matrix}  \left(I + (2\tau_{\lambda}-1)\hat{b}\hat{b}^T \right)\mu &\frac{ \tau_{\lambda}}{\lambda}b\\ - \frac{\tau_{\lambda}}{\lambda}b^T\mu & 2\tau_{\lambda} \end{matrix}\right] , \]
where $\hat{b} = b/\norm{b}$ and $\tau_{\lambda} = \left(4 + \|b\|_2^2/\lambda^2\right)^{-\frac{1}{2}}$.
\end{prop}
The proof is given in the supplementary, Section~\ref{app:proof_of_phi_form}.

The explicit form of the projections allows us to deduce some important properties. First, we do not have to fully compute the SVD of the representing matrix of $g_\lambda$ to calculate $\Phi_{\lambda}\left(g \right)$. Second, it shows how to invert $\Phi_{\lambda}$, as $I + (2\tau_{\lambda}-1)\hat{b}\hat{b}^T$ is invertible for any $\tau_{\lambda}>0$. Note that the case $\tau_{\lambda}=\frac{1}{2}$ corresponds to $b=0$ which verifies $\Phi_{\lambda}(g)=g$ in this case. The invertibility of $\Phi_{\lambda}$ is crucial for its use as a compactification method for synchronization, as we will see in the next section. Thus, we conclude
\begin{corollary} \label{prop:injective_maps}
The maps $\Phi_\lambda$ of~\eqref{eqn:matrix_projection} are injective.
\end{corollary}
\begin{proof}
The corollary holds since Proposition~\ref{prop:detailed_phi} implies that for any $g_1 = \left[\begin{smallmatrix} \mu_1 & b_1 \\ \mathbf{0}_{1\times d} & 1\end{smallmatrix}\right]$ and $g_2 = \left[\begin{smallmatrix} \mu_2 & b_2 \\ \mathbf{0}_{1\times d} & 1\end{smallmatrix}\right]$ in $\se(d)$, $\Phi_{\lambda}\left( g_1 \right) = \Phi_{\lambda}\left( g_2 \right)$ entails $b_1 = b_2$ and 
\[ \left(I + (2\tau_{\lambda}-1)\hat{b}\hat{b}^T \right)\mu_1 = \left(I + (2\tau_{\lambda}-1)\hat{b}\hat{b}^T\right)\mu_2 , \] 
implies $\mu_1=\mu_2$. Namely, $ g_1  = g_2$. 
\end{proof} 

As in the case of group contraction maps, the justification for using $\Phi_{\lambda}$ is as follows,
\begin{prop} \label{prop:matrix_proj_app_hom}
The maps $\Phi_\lambda$ of~\eqref{eqn:matrix_projection} are approximated homomorphisms, as defined in \eqref{eqn:app_homo}.
\end{prop}
The proof is given in supplementary, Section~\ref{app:proof_of_approximated_homomorphism}.

\section{Application of compactification to synchronization} \label{sec:ApplicationForSync}

We now focus on the application of the above algebraic tools of compactification to synchronization over groups. In light of current solutions for synchronization, described in Subsection \ref{subsec:available_solutions_se_group}, the use of contractions for synchronization offers several advantages. To begin with, it uses full integration of the available data (in contrary to methods that separate between the two components of the group). Secondly, it is valid for a wide class of non-compact Cartan motion groups, including the special case of $\se(d)$. At last, as reviewed in Subsection~\ref{subsec:synch_compact_group} many efficient methods were designed and suggested for solving rotation synchronization. By applying contraction, we facilitate the use of almost any such method in order to gain extra robustness (for example, dealing with outliers or noise) for the case of synchronization over Cartan motion groups as well.

\subsection{Synchronization: from Cartan motion to compact and back}

We describe the use of compactification for solving the synchronization problem over Cartan motion groups. The compactification methods are exploited to reduce this problem into the well-studied rotation synchronization. Using the notation of the previous section, $\Gb$ is a compact group and $\Gc$ is its contraction, a Cartan motion group. We solve the synchronization over $\Gc$ as follows. First, we apply the contraction map $ \Psi_\lambda$ to $ \{ g_{ij} \}$, forming a new set of measurements in $\Gb$. Then, we solve the corresponding synchronization problem, having the same data graph $\mathcal{G}$ under the new set of measurements. The constraint of the new synchronization problem is now $ \{ \Psi_\lambda\left(g_i \right) \}_{i=1}^n \subseteq \rho\left( \Gb  \right) $. Then, having a solution to the problem in $\Gb$, we map it back to $\Gc$ by using the inverse map $\Psi_\lambda^{-1}$. This procedure is summarized in Algorithm~\ref{alg:contraction_Cartan}.

\begin{algorithm}[ht]
\caption{Synchronization via group contraction on Cartan motion groups}
\label{alg:contraction_Cartan}
\begin{algorithmic}[1]
\REQUIRE  Ratio measurements $\{ g_{ij} \}_{(i,j) \in \mathcal{E}} \subset \Gc$, and corresponding weights $\{ w_{ij} \}_{(i,j) \in E} $.  \\ A synchronization solver $\operatorname{Solve}_\Gb$ over $\Gb$ (e.g., see Algorithm~\ref{alg:eigenvector_method}). \\
\ENSURE An approximation $\{\hat{g}_i\}_{i=1}^n$ to the unknown group elements.
\STATE Choose $\lambda=\lambda \left( \{ g_{ij} \}_{(i,j) \in \mathcal{E}} \right) \ge 1$.  \quad \COMMENT{See Section~\ref{sec:numerical_examples}}
\FOR{$(i,j) \in \mathcal{E}$}   
\STATE $ Q_{ij} \gets \Psi_\lambda\left(g_{ij}\right) $       
\COMMENT{Compactification step: Mapping $\Gc$ to $\Gb$ (for $\se(d)$ it can be done, e.g., by $\Phi_\lambda$)}                           
\ENDFOR 
\STATE $\{ Q_{i} \}_{i=1}^n \gets \operatorname{Solve}_\Gb\left( \{ Q_{ij}, w_{ij}\}_{(i,j) \in E} \right)$   \label{lin:compactSync}
\STATE Choose a global alignment $Q \in \Gb$  \quad \COMMENT{See Section~\ref{sec:numerical_examples}}
\FOR{$i=1,\ldots,n$}    
\STATE  $\hat{g}_i \gets \Psi_\lambda^{-1}\left(Q_{i} Q \right)$           \qquad           \COMMENT{Back mapping to $\Gc$}
\ENDFOR       
\RETURN $\{\hat{g}_i\}_{i=1}^n$.
\end{algorithmic}
\end{algorithm}

The back mapping from $\Gb$ to $\Gc$ after the compact synchronization step leads to the following observation; while compact synchronization offers a non-unique solution, which is determined up to a global alignment (right multiplication by a fixed group element), the choice of the certain solution which we actually map back to $\Gc$ effects the solution there. Namely, the solution in $\Gc$ is still non-unique, but for any two different solutions (equal up to global alignment) in $\Gb$ we get two different solutions in $\Gc$ which do not necessary differ only by global alignment. This is a result of the fact that we cannot guarantee the back mapping to be a homomorphism. Nevertheless, we can characterize when global alignment in $\Gb$ remains a global alignment in $\Gc$ after the back mapping.

\begin{prop} \label{prop:invarint_back_mapping}
Let $\{ Q_i \}_{i=1}^n$ be a solution to the compact synchronization step in Line~\ref{lin:compactSync} of Algorithm~\ref{alg:contraction_Cartan}. Denote the Cartan decomposition in $\Gb$ of each of its elements by $Q_i = \exp(v_i)k_i$. Then, the output of Algorithm~\ref{alg:contraction_Cartan} is invariant to any global alignment $Q$ with the Cartan decomposition $Q = \exp(v)k$ such that $\ad{k_i}{v}$ commutes with $v_i$, $i=1,\ldots,n$. 
\end{prop}
Note that $v=0$ always satisfies the conditions in the above proposition.
\begin{proof}
By definition,
\begin{equation} \label{eqn:prod_QQi}
Q_iQ = \exp(v_i)k_i \exp(v)k = \exp(v_i)\exp(\ad{k_i}{v})k_i k . 
\end{equation}
However, $\comm{\ad{k_i}{v}}{v_i} = 0$ and therefore $\exp(v_i)\exp(\ad{k_i}{v}) = \exp(v_i+\ad{k_i}{v})$ with both $v_i+\ad{k_i}{v} \in \mathfrc{p} $ and $k_i k \in \mathfrc{t}$. Namely,
\[ \begin{aligned}[t]
    \Psi_\lambda^{-1} \left( Q_iQ \right)    &=   \Psi_\lambda^{-1} \left( \exp(v_i+\ad{k_i}{v}) k_i k \right) \\
            &=  \left( k_i k, \lambda (v_i+\ad{k_i}{v}) \right) = \Psi_\lambda^{-1} \left( Q_i \right) \Psi_\lambda^{-1} \left( Q \right) .
      \end{aligned} \]
Since the inverse is invariant for approximated homomorphism maps, so is the case for the inverse map, and overall we conclude that 
\[  \Psi_\lambda^{-1} \left( Q_i Q \right) \Psi_\lambda^{-1} \left( (Q_j Q)^{-1} \right) = \Psi_\lambda^{-1} \left( Q_i \right) \Psi_\lambda^{-1} \left( Q_j \right)^{-1} , \]
In other words, any cost function of synchronization over $\Gc$ that is based on the distance between $g_{ij}$ and $\Psi_\lambda^{-1} \left( Q_i \right) \Psi_\lambda^{-1} \left( Q_j \right)^{-1} $ remains unchanged after the global alignment $Q_i \mapsto Q_iQ$.
\end{proof}

Proposition~\ref{prop:invarint_back_mapping} basically states that any global alignment that preserves the structure of the Cartan decomposition in $\Gb$, does not affect the back mapping to $\Gc$. In other words, this result opens the door for further potential improvement in synchronization, done by carefully choosing a global alignment in $\Gb$ that minimizes the error after mapping back to $\Gc$. We elaborate more on the practical considerations of this choice in Subsection~\ref{subsec:global_alignment}.

Many of current solutions for synchronization are based on minimizing a cost function that depends on the Euclidean distance between the ratios of estimated elements and the given measurements. One such example is given in \eqref{eq:NonCompactOptimPrb1} for the case of $\se(d)$. In terms of a general Cartan motion group $ \Gc =  K \ltimes V$ we can rewrite this problem as
\begin{equation} \label{eqn:Cartan_Optim}
\begin{aligned}
& \underset{{\scriptstyle\{g_i\}_{i=1}^n}}{\text{minimize}}
& & \sum_{(i,j)\in \mathcal{E}} w_{ij}  d_H \left( g_i g_j^{-1}, {g}_{ij} \right)^2 \\
& \text{subject to}
& & \{ g_i\}_{i=1}^n \subseteq \Gc  ,
\end{aligned}
\end{equation}
with the hybrid metric $d_H \left( g_1,g_2 \right) =\sqrt{ \norm{k_1-k_2}^2_K +\norm{v_1-v_2}^2_V   }$ for $g_1 = (k_1,v_1),g_2=(k_2,v_2) \in \Gc$. The norm on $K$, as a (compact) subgroup of the compact group $\Gb$, is defined as the Frobenius norm on $\rho$, its faithful, orthogonal matrix representation, and the norm on $V$ is the Euclidean norm. The theoretical justification to use a hybrid metric here is since one can show an orthogonality between the two component of the group, following from the Cartan decomposition in the algebra level, and the preservation of orthogonality by the exponential map as implied by Gauss's lemma.

For the special case $\Gc=\se(d)$ the above synchronization problem~\eqref{eqn:Cartan_Optim} gets the following form after applying the contraction $\Psi_\lambda$
\begin{equation} \label{eqn:sync_SE_w_contraction}
\begin{aligned}
& \underset{{\scriptstyle\{Q_i\}_{i=1}^n \subset \so(d+1)}}{\text{minimize}}
& & \sum_{(i,j)\in \mathcal{E}} w_{ij}  \norm{Q_iQ_j^T-\Psi_\lambda(g_{ij})}_F^2      \\
& \text{subject to}
& & Q_i =\Psi_\lambda(g_i) , \quad g_i \in \se(d)  ,\quad i = 1,\ldots,n .
\end{aligned}
\end{equation}
Following Algorithm~\ref{alg:contraction_Cartan}, we firstly approach \eqref{eqn:sync_SE_w_contraction} by solving synchronization with $\left\{ \Psi_{\lambda}\left( g_{ij} \right) \right\}_{i,j \in \mathcal{E}} \subset \so(d+1)$. For that, we can utilize one of the methods mentioned in Subsection~\ref{subsec:synch_compact_group}. For example, following the spectral method of Example~\ref{exmple:EigenvectorsMethod}, we construct the measurement matrix $M$, given by
\[ M_{ij} = \twopartdef { w_{ij}\Psi_{\lambda}\left(g_{ij}\right)} {(i,j) \in \mathcal{E},} {\mathbf{0}_{(d+1)\times (d+1)}} {(i,j) \notin \mathcal{E},} \]
and extract its leading eigenvectors to have our rotations estimation. We then calculate each inverse $\Psi_{\lambda}^{-1}\left( Q_i \right)$ to recover a solution, back in $\se(d)$. This step requires the calculation of the Cartan decomposition for every other element (see Example~\ref{exmple:SE} for this Cartan decomposition). We elaborate more on the calculations involved for obtaining the Cartan decomposition in supplementary, Section~\ref{apx:CartanCalculation}.

In the above example, we see a typical difference between an original synchronization problem in $\Gc$ and its reduced corresponding problem in $\Gb$. The source of difference in the cost functions of \eqref{eq:NonCompactOptimPrb1} and \eqref{eqn:Cartan_Optim} are the norms. Since these Euclidean distances play a key role in many instances of cost functions of synchronization, e.g, in LUD methods \cite{wang2013exact} or in least squares methods \cite{singer2011angular}, we prove that there exists a factor, quadratic in $\lambda$, that links them.
 
\begin{prop} \label{prop:distance_between_const_functions}
Let $\Psi_\lambda$ be a contraction of the form \eqref{eqn:cartan_contraction}. Assume that $g_1,g_2,g \in \Gc$ are such that $g_1g_2^{-1}$ and $g$ are close enough in the sense of Lemma~\ref{lemma:exp_dist_bound}, that is their $\mathfrc{p}$ components after scaling by $\lambda$ are inside the injectivity radius around the identity. Then, 
\[ \norm{\Psi_\lambda(g_1)\Psi_\lambda(g_2)^{-1}-\Psi_\lambda(g)}_F^2 \le  \left( d_H \left( g_1g_2^{-1},g\right)\right)^2 + \mathcal{O}(\lambda^{-2}) .\]
where the constant in $\mathcal{O}(\lambda^{-2})$ is a function of $\comm{g_1}{g_2}$ but independent of $\lambda$.
\end{prop}
Note that the condition above on sufficiently closeness of elements is always true for large enough $\lambda$. The proof of the proposition is given in supplementary, Section~\ref{app:proof_of_norms_dist_bound}.

\subsection{Additional analysis of synchronization via contractions} \label{subsec:sync_properties}

In the rest of this section, we aim to further study the effect of mapping the measurements using the contraction map. In this subsection, we focus on the case without the presence of noise, where in the next subsection we examine the influence of multiplicative noise on the mapped measurements. 

One feature of many synchronization methods, such as the spectral method, is the exactness of solution in the case of no noise (where a sufficiently amount of data is available). Recall for the noiseless case, the naive solution of constructing spanning tree is sufficient, so from a design point of view, we aim to apply synchronization for cases of corrupted (either by noise, outliers or missing data) sets of measurements. Nevertheless, it is interesting from the analysis point of view to understand the contraction based synchronization for this case. We start with a lemma.

\begin{lemma} \label{lemma:multiplicative_Psi_homo}
Let $g_1 = (k_1,v_1) $ and $g_2=(k_2,v_2)$ be element of a Cartan motion group $\Gc$. Then,
\[ \Psi_\lambda \left( g_1 g_2^{-1} \right)  = \Psi_\lambda \left( g_1 \right) Q (\Psi_\lambda \left( g_2 \right))^{T} , \]
where 
\[ Q = \prod_{m=2}^\infty \exp\left( Z_m(\ad{k_1^T}{v_1/\lambda},\ad{k_2^T}{v_2/ \lambda}) \right) , \]  
and $Z_m$ is a homogeneous Lie polynomial (nested commutators) of degree $m$.
\end{lemma}
\begin{proof}
By definition (recall that we identify $\Psi_\lambda(\cdot)$ with its orthogonal representation)
\begin{eqnarray*}
 Q &=& \Psi_\lambda \left( g_1 \right)^T \Psi_\lambda \left( g_1g_2^{-1} \right) \Psi_\lambda \left( g_2 \right) \\
 & =& k_1^T\exp(\lambda^{-1}(-v_1)) \exp(\lambda^{-1}(v_1-k_1 k_2^Tv_2)) k_1 k_2^T \exp(\lambda^{-1}(v_2))k_2 \\
  & =& \exp(\lambda^{-1}(\ad{k_1^T}{-v_1})) \exp(\lambda^{-1}(\ad{k_1^T}{v_1}-\ad{k_2^T}{v_2}))\exp(\lambda^{-1}(\ad{k_2^T}{v_2})) 
\end{eqnarray*}
The last equation is of the form $\exp(-A)\exp(A+B)\exp(B)$, so applying the Zassenhaus formula on the most right product yields $\exp(A)$ multiply by an infinite product of exponentials of nested commutators. Note that $\comm{A+B}{B} = \comm{A}{B}$.
\end{proof}

As a conclusion from Lemma~\ref{lemma:multiplicative_Psi_homo} we learn that only by applying the contraction map we suffer from some distortion of the data (by a multiplicative factor). 
\begin{corollary} \label{cor:clean_measure_analysis}
In the case of clean measurements $g_{ij} = g_ig_j^{-1} \in \Gc$, the contraction map $\Psi_\lambda$ yields the data $\Psi_\lambda \left( g_i \right)Q_{i,j}\Psi_\lambda \left( g_j \right)$. Namely, the elementwise mapping $\Psi_\lambda \left( g_i \right)\Psi_\lambda \left( g_j \right)$ is distorted by an orthogonal matrix $Q_{ij}$. The additional matrix $Q_{ij}$ lies in a distance bounded by $\mathcal{O}(\lambda^{-2}\comm{g_i}{g_j})$ from the identity matrix.
\end{corollary}
Bounding $Q_{ij}$ can be done as in the proof of Lemma~\ref{lemma:exp_dist_bound}, by either rewriting the leading term of the Taylor expansion or by considering the leading matrix in the product and its geodesic distance to the identity. 

The latter corollary gives further understanding about the nature of contraction maps for synchronization. It also provides an explanation for a typical numerical loss of a digit in synchronization error when acting on a full set of pure clean samples, see also Section~\ref{sec:numerical_examples}. While it is true that in the noiseless case one should use an unbiased method, such as the spanning tree method (see Section~\ref{sec:intro}), for noisy measurements it might be beneficial to estimate the group elements with a biased estimator that incurs a smaller variance error. Synchronization via contraction, which does not reconstruct the solution in case of clean measurements, does lead to improvements in many more realistic scenarios, as we will see later in Section~\ref{sec:numerical_examples}.

\subsection{Notes on noisy synchronization}

The analysis of synchronization in presence of noise, mostly in the case of the rotation group, has recieved a lot of attention in the past few years, e.g., \cite{boumal2014thesis, singer2011angular, tzeneva2011global, wang2013exact}. One common approach for analyzing global methods for synchronization is to study the spectrum of the measurement matrix, that is $M$ in \eqref{eq:MeasurementMatrixM}. The largest eigenvalues are associated with the synchronization solution, and are ideally separated from the rest of the spectrum which includes the noise. In other words, a gap between the largest eigenvalues and the rest of the spectrum means that a meaningful information is dominating the noise and a solution can be extracted. The eigenvalue separation fades away as noise increases and so does the accuracy of the solution. Other approaches for analysis include adapting tools from information theory, such as the analysis done in \cite{singer2011angular}, or the construction of lower bounds such as Cram{\'e}r-Rao \cite{boumal2014cramer}. Since the analysis of synchronization over compact groups has already been established, the key step in our analysis of synchronization via contraction lies in the compactification stage.

Let $g_{ij}$ be a noisy measurement, contaminated with a multiplicative noise matrix $N_{ij} \in \Gc$,
\begin{equation} \label{eqn:noisemodel}
g_{ij} = g_i N_{ij} g_j^{-1} .
\end{equation}  
Such a noise model is used in \cite{boumal2014cramer} and is often adapted by practitioners as well, for example over $\se(3)$ see \cite{barfoot2014associating}. In this multiplicative model we actually assume that the graph of the data $\mathcal{G}$ is contaminated with noise on its edges (this is not always the case in applications, e.g., \cite{singer2011viewing} but it facilitates the analysis comparing to the case of noisy vertices). 

A spectral analysis of rotation synchronization, for the case of noise on edges, is given in \cite[Chapter 5]{boumal2014thesis, singer2011angular}. The essence is as follows. Let $R$ be a block diagonal matrix, with its diagonal blocks $R_i$ are the unknown rotations. Then, each noisy measurement $g_{ij}$ can be viewed as the $ij$ block of $M= RWR^T$, where $W$ is a block matrix with the noise blocks $W_{ij}$. Since $M$ and $W$ are similar, the spectrum of $M$ is the same as $W$, which justifies studying $W$ for deriving conclusions on the spectrum of $M$. The goal is to deduce a condition that guarantees a separation of leading eigenvalues, as described above (meaning that the largest eigenvalues are not dominated by the noise). It is shown in \cite[Chapter 5]{boumal2014thesis} that if $W_{ij}$ is drawn from a distribution that has a density function which is both a class function and centered around the identity with $\E{W_{ij}} = \beta I$, $0\le \beta \le 1 $, a sufficient condition is  
\begin{equation} \label{eqn:analysis_cond}
\beta > \frac{1}{\sqrt{n}} .
\end{equation}
Next, we study some analogue requirements on the noise block $N_{ij}$ to satisfy similar conditions after applying the contraction mapping $\Psi_\lambda$.

Denote the components of each group element by $g_i = (k_i,v_i)$ and $N_{ij} = (\upsilon_{ij},a_{ij})$. Also, for simplicity we denote $\exp_\lambda(\cdot) = \exp(\cdot/\lambda)$. Then, a direct calculation of \eqref{eqn:noisemodel} by \eqref{eqn:Cartan_action} yields
\[ g_{ij} = \left( k_i \upsilon_{ij} k_j^T, v_i + \ad{k_i}{a_{ij}}-\ad{k_i \upsilon_{ij} k_j^T}{v_j} \right). \]
Some simplification using the commutativity of adjoint operator and the exponential map, and the definition of $\Psi_\lambda$ mean we can rewrite $r_{ij} =\Psi_\lambda(g_{ij}) \in \Gb$ as
\[ r_{ij} = k_i \left[ \exp_\lambda(\ad{k^T_i}{v_{i}}+a_{ij}-\ad{\upsilon_{ij} k_j^T}{v_j} ) \upsilon_{ij} \right] k_j^T . \]
One interpretation as a noisy data matrix for synchronization is to consider $\hat{W}_{ij} = \exp_\lambda(\ad{k^T_i}{v_{i}}+a_{ij}-\ad{\upsilon_{ij} k_j^T}{v_j} ) \upsilon_{ij} \in \Gb$ with the unknown $\{  k_i \}_{i=1,\ldots,n}$. However, further examination reveals that considering a slightly different matrix 
\begin{equation} \label{eqn:WnoisyCase}
 W_{ij} = \exp_\lambda(-\ad{k^T_i}{v_{i}})\exp_\lambda(\ad{k^T_i}{v_{i}}+a_{ij}-\ad{\upsilon_{ij} k_j^T}{v_j})\exp_\lambda(\ad{\upsilon_{ij} k_j^T}{v_j}) \upsilon_{ij}  
\end{equation}
leads to the more natural form $r_{ij} = \Psi_\lambda\left( g_i \right)  W_{ij}  \Psi_\lambda\left( g_j \right)^T$.

To simplify $W_{ij}$ of \eqref{eqn:WnoisyCase}, we apply the Baker-Campbell-Hausdorff (BCH) formula to have
\[  \exp_\lambda(\ad{k^T_i}{v_{i}}+a_{ij}-\ad{\upsilon_{ij} k_j^T}{v_j})\exp_\lambda(\ad{\upsilon_{ij} k_j^T}{v_j}) \upsilon_{ij} = 
\exp( (\ad{k^T_i}{v_{i}}+a_{ij})/\lambda + X ) , \]
where $X = \frac{1}{2}\comm{(\ad{k^T_i}{v_{i}}+a_{ij})/ \lambda }{(\ad{\upsilon_{ij} k_j^T}{v_j})/ \lambda } + \cdots $ is the series of nested commutators, with the leading term of order $\lambda^{-2} \left(  \comm{\ad{k^T_i}{v_{i}}}{\ad{\upsilon_{ij} k_j^T}{v_j}} + \comm{a_{ij}}{\ad{\upsilon_{ij} k_j^T}{v_j}}  \right) $. Then, a second application of the BCH formula implies that
\begin{equation} \label{eqn:SimplifiedNoiseMat}
  W_{ij} = \exp(a_{ij}/\lambda+Y)\upsilon_{ij} , 
\end{equation}
where $Y  =  \frac{1}{2}\comm{-\ad{k^T_i}{v_{i}}/\lambda}{(\ad{k^T_i}{v_{i}}+a_{ij})/\lambda + X} + \cdots $ is the corresponding series of nested commutators. Combine $X$ with the above gives the explicit form of the commutator as
\[  \lambda^{-2} \left(  \comm{a_{ij}}{\ad{k^T_i}{v_{i}}}   \right) + \lambda^{-3} \left(   \comm{\comm{\ad{k^T_i}{v_{i}}}{\ad{\upsilon_{ij} k_j^T}{v_j}} + \comm{a_{ij}}{\ad{\upsilon_{ij} k_j^T}{v_j}}}{\ad{k^T_i}{v_{i}}}      \right)   . \]
The next order nested commutator term in $Y$ corresponds to $\lambda^{-3}$ and $\lambda^{-4}$ and so on for higher order terms. One conclusion is that if the translational part of the noise is of zero expectation, then \eqref{eqn:SimplifiedNoiseMat} is in expectation a third order approximation to the Cartan decomposition of $W_{ij}$,
\begin{corollary} \label{cor:appCartanDecom}
Let $a_{ij}$ be a random variable independent of $k_i, k_j$ and $\upsilon_{ij}$, and satisfy $\E{a_{ij}}=0$. Then,
\[  \E{a_{ij}+Y} + \mathcal{O}\left( \lambda^{-4} \right) \in \mathfrc{p} . \]
\end{corollary}
\begin{proof}
The proof follows from two arguments regarding the series $Y$. First, for any independent factor $b \in \mathfrc{g}$, the commutator $\comm{a_{ij}}{b}$ is linear in $a_{ij}$ at each entry. So, $\E{a_{ij}} = 0$ implies $\E{\comm{a_{ij}}{b}} = 0$. Second, we know that $\comm{\mathfrc{p}}{\mathfrc{p}} \subset \mathfrc{t}$ and $\comm{\mathfrc{p}}{\mathfrc{t}} \subset \mathfrc{p}$, meaning that the remaining nonzero third order nested commutators like $\comm{\comm{\ad{k^T_i}{v_{i}}}{\ad{\upsilon_{ij} k_j^T}{v_j}}}{\ad{k^T_i}{v_{i}}}$ are in $\mathfrc{p}$. Namely, the leading order outside $\mathfrc{p}$ is of order $\lambda^{-4} $.
\end{proof}

The proof of Corollary~\ref{cor:appCartanDecom} shows that if $a_{ij}$ has a zero expectation, mapping the noise matrix $\Psi_\lambda \left(N_{ij} \right) = \exp \left(  a_{ij}/\lambda \right)   \upsilon_{ij} $, which is already a matrix in its Cartan form, is not too far away from the Cartan form, in expectation, of $W_{ij}$ of \eqref{eqn:SimplifiedNoiseMat}. For this case, which is a first order perturbation in $\lambda$ to $W_{ij}$, we can adapt conditions like \eqref{eqn:analysis_cond} from rotation synchronization to synchronization over the Euclidean group. In the following, for a cleaner presentation, we conclude the noise analysis for the case of $\se(d)$. The overall arguments can be carried with some minor modifications to a general Cartan motion group. The model is of a full graph, that is all measurements are available. We use $\circ$ for Hadamard (entrywise) matrix product.

\begin{prop} \label{prop:noise_analysis}
Denote by $W_{ij} = \Psi_\lambda \left(N_{ij} \right) = \exp \left(  a_{ij}/\lambda \right) \upsilon_{ij} $ the $ij^\text{th}$ block of the noise matrix $W$ and consider the rotation synchronization problem defined by the measurements matrix $RWR^T$. Assume all $a_{ij}$, and $\upsilon_{ij}$ are independent random variables over $V=\mathbb{R}^d$ and $K=\so(d)$ with probability density functions (PDF) $f_V$ and $f_K$, respectively. Furthermore, we assume each PDF is smooth and non-negative such that $\E{\upsilon_{ij}} = \beta I$, $0 \le \beta \le 1$ and $\E{a_{ij}}=0$ with even $f_V$. Then, the following condition guarantees a spectral gap between the eigenvalues due to noise and the leading eigenvalue that corresponds to the the synchronization solution,
\[  \min\{  (\gamma-\alpha_1)\beta, \gamma-\alpha_2) \}>\frac{1}{\sqrt{n}} . \]
Here $\alpha_1>0$ is $\alpha_1I = \int_{\norm{v}\le\pi}  \left( I \circ vv^T \right) h(v)dv $, and $\alpha_2=\int_{ \norm{v}\le\pi} \norm{v}^2 h(v)dv >0 $, and $\gamma = \int_{ \norm{v}\le\pi} h(v)dv$ with respect to the density function $h(x) = \frac{1-\cos(\norm{x} )}{\norm{x}^2} f_V(x) \norm{(I-xx^t)^{-\frac{1}{2}}}$.
\end{prop}

The proof is given in supplementary, Section~\ref{app:proof_noise_analysis}.

The new condition of Proposition~\ref{prop:noise_analysis} shows a tradeoff between the amount of noise in the two parts of the data over $\se(d)$ versus the sample size. In addition, it encodes the effect of geometry of the induced homogeneous space $\Gb/K$ as it includes two additional factors that depend on the new, modified density function $h$. The condition shows that robustness to noise increases as the number of group elements is growing, as it allows more variance around the identity. The positivity of $\alpha_1$ and $\alpha_2$, implies the expectation values on (the diagonal parts of) $\E{W_{ij}}$ are bounded by $1$. Indeed, as $f_V$ tends to be concentrated around $0$, which means lower variance on the translational parts, $\alpha_1$ and $\alpha_2$ approach to zero and have a lesser effect on the condition. The expectation term $\beta$ affects the upper $d \times d$ block of each $\E{W_{ij}}$, that is associated with the subgroup $K$. As $f_K$ gets more concentrated around the identity, $\beta$ becomes closer to one. High variance on the element parts of $K$ (in a case of $f_K$ that is concentrated around the identity) implies smaller $\beta$ values and more data is needed to guarantee valid information.

As a final note, one can carry the same proof technique as we use in Proposition~\ref{prop:noise_analysis} to prove a similar result for more general Cartan motion groups. Nevertheless, such a result might have a less explicit form, as both the Jacobian of local coordinates in $\Gb/K$ as well as the exponential map there can be less accessible or lead to more complicated expressions.

\section{Numerical evaluation and its considarations}  \label{sec:numerical_examples}  
In this section, we discuss the numerical evaluation of the method of contraction for synchronization over two instances of Cartan motion groups: the group of rigid-motions and matrix motion group. The discussion goes hand in hand with numerical examples using synthetic and real world data.

\subsection{Preliminaries}
We firstly state how we address three preliminary issues: measuring the error, choosing the parameter $\lambda$ of the contraction, and optimizing the global rotations alignment. All algorithms and examples in this section were implemented in Matlab and are available online in \url{https://github.com/nirsharon/Synchronization-over-Cartan-Motion-Groups}. The numerical examples were executed on a Springdale Linux desktop equipped with 3.2 GHz i7-Intel Core\textsuperscript{TM} processor with 16 GB of memory.

\subsubsection{Measuring errors}
In our examples we choose to evaluate the performance of the algorithms in terms of mean squared error (MSE). However, since the solution of synchronization is non-unique, we explicitly show how we measure this error. Consider $X = \left\{\hat{g}_i\right\}_{i=1}^n$ to be the solution of one algorithm, that is the estimators of $\left\{g_i\right\}_{i=1}^n$. Then, the measure of error for this solution is given by
\begin{equation}
\label{eq:PerfMeasure}
\mse{X} = \frac{1}{n}\min_{ g \in \Gc } \sum_{i=1}^n    d_H \left(  \hat{g}_i g,g_i\right)^2 . 
\end{equation}
Using the above notation on the two components $(k,v)$ of each element, and by the definition of the action \eqref{eqn:Cartan_action}, the sum in \eqref{eq:PerfMeasure} becomes
\[
 \sum_{i=1}^n \left\| \hat{k}_i k - k_i \right\|_K^2 + \left\| \hat{v}_i - \ad{\hat{k}_i}{v} -v_i \right\|_V^2 . \]
Since we optimize over the single element $g=(k,v)$, we can clearly see the target function is separable. In particular, for $\se(d)$ and with the representation of elements as $g = \left[ \begin{smallmatrix} \mu & b \\ \mathbf{0}_{1\times d} & 1\end{smallmatrix} \right] $ we have
\[ \begin{aligned}[t]
	\sum_{i=1}^n \left\| \hat{g}_i g - g_i \right\|_F^2    &=  \sum_{i=1}^n \left\| \hat{\mu}_i\mu - \mu_i \right\|_F^2 + \left\| b - \hat{\mu}_i^T\left(b_i - \hat{b}_i\right) \right\|_2^2\\
  &= \left(n b^Tb - 2b^T\sum_{i=1}^n \hat{\mu}_i^T\left(b_i - \hat{b}_i\right) \right) - 2\tr \left( \mu^T \left( \sum_{i=1}^n \hat{\mu}_i^T\mu_i\right) \right) + C 
      \end{aligned} , \]
where $C = \sum_{i=1}^n \left\| b_i - \hat{b}_i \right\|_2^2 + 4n$ is independent of the minimized element $g$. Therefore, the minimizer is given by $b^* = \frac{1}{n} \sum_{i} \hat{\mu}^T\left(b_i - \hat{b}_i\right)$ and $\mu^* = V\left[\begin{smallmatrix} 1 & 0 \\ 0 & \det \left( VU^T \right) \end{smallmatrix} \right] U^T$, where $\sum_i \hat{\mu}_i^T\mu_i = U\Sigma V^T$ is the SVD of $\sum_i \hat{\mu}_i^T\mu_i$ with $\Sigma_{1,1} \geq \Sigma_{2,2}$, see e.g.,~\cite{hanson1981analysis}. 

The important conclusion is then, measuring the error by~\eqref{eq:PerfMeasure} is feasible to calculate.

\subsubsection{Choosing \texorpdfstring{$\lambda$}{the parameter}  }

The basic tradeoff of varying $\lambda$ is as follows. For large parameter values ($\lambda \gg 1$), the translational part has smaller effect on the data after applying $\Psi_\lambda$ (In fact, by looking at the explicit form of the exponential map, as given in supplementary, Section~\ref{apx:CartanCalculation}, the entries of the exponential roughly perturbs the identity quadratically with respect to the norm of the translational parts). On the other hand, in such cases the mapped measurements $\Psi_\lambda(g_{ij})$ are more consistent as data for synchronization in terms of satisfying the triplet relation $\Psi_\lambda(g_{ij})\Psi_\lambda(g_{j\ell}) \approx \Psi_\lambda(g_{i\ell})$. In contrary, small $\lambda$ values increase the information we gain from the translational part of the data but involves solving non-consistent synchronization problem which often results in large errors. One additional issue we would like to address in choosing the parameter of contraction is to fulfill both \eqref{eqn:lambda_cond} (keeping the data inside the injectivity domain of the exponential map) and the condition for the approximated homomorphism as given in Proposition~\ref{prop:app_homo_contraction}.

We tested numerically the dependency between varying $\lambda$ values and the synchronization error, resulted from our synchronization via contraction. As expected, this dependency is related to the different properties of the data (noise level, sparsity of the graph of data, etc.) as well as to the specific method one uses for the synchronization of rotations. To demonstrate the latter in the case of $\se(d)$, we compare the optimal $\lambda$ value, given as a function of the noise level, between two different rotation synchronization algorithms. One algorithm is the spectral method (see Algorithm~\ref{alg:eigenvector_method}) and the other is synchronization based on maximum-likelihood estimation (MLE) \cite{boumal2013robust}. The ground truth consists of $100$ elements, collected randomly on $\se(3)$. We formed the data by considering a full data graph and contaminated the samples with a multiplicative noise, as in \eqref{eqn:noisemodel} (we elaborate more on that type of noise in the next subsection). The results are given in Figure~\ref{subfig:lambda1}; while a difference exists between the two cases, the overall trend is similar -- the parameter $\lambda$ increases as the noise decreases and vice versa. Namely, the optimal $\lambda$ tends to be found closer to $1$ as the noise increases. As an example of how the error changes in the vicinity of the optimal $\lambda$ value, we present in Figure~\ref{subfig:lambda2} the error as a function of $\lambda$, for the case of EIG, with noise level of about $7 dB$ of SNR. 

\begin{figure}
\centering    
\subfloat[Optimal $\lambda$ values for two different methods for rotations synchronization]{	  \label{subfig:lambda1}  \includegraphics[width=0.3\textwidth]{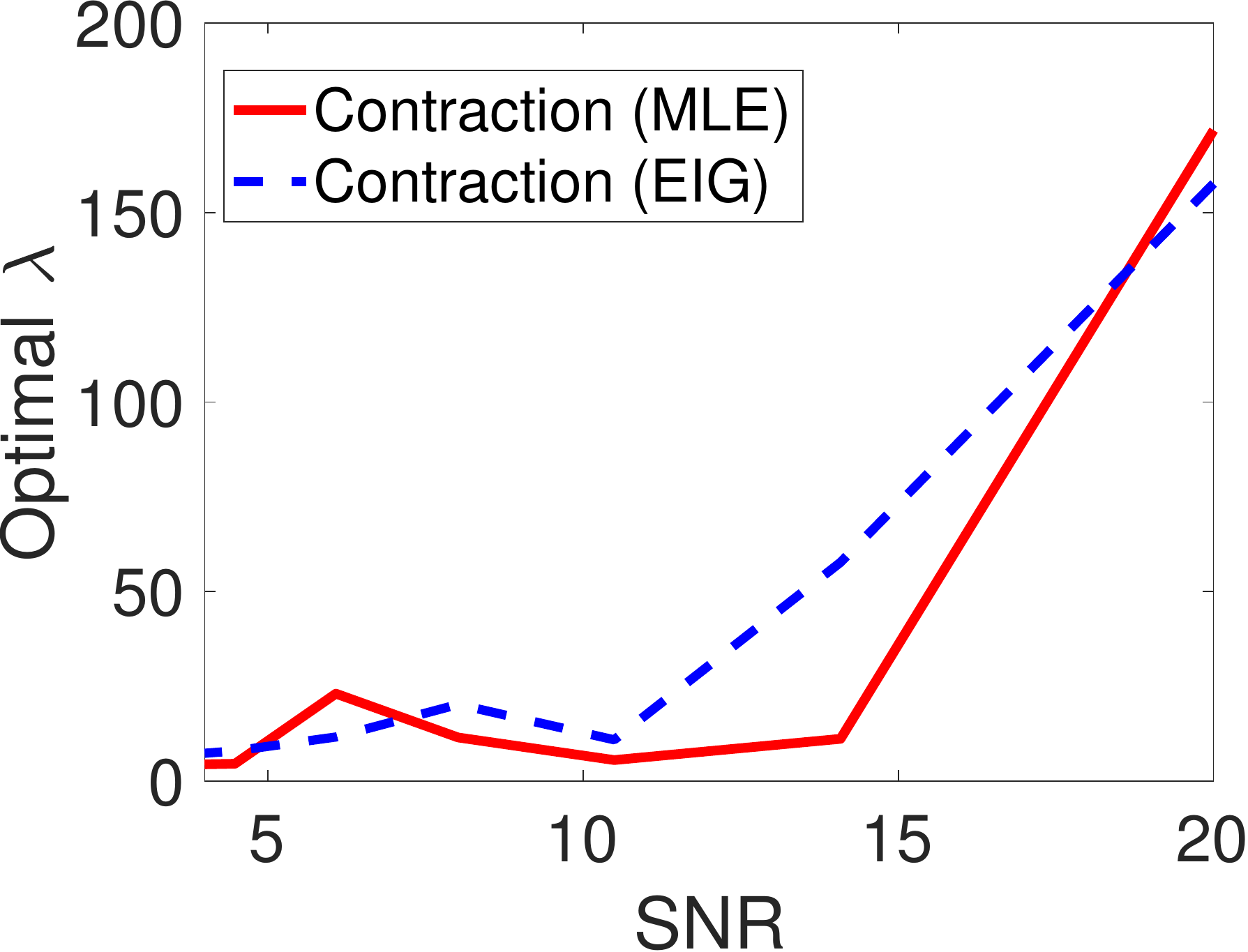}} \qquad
\subfloat[Synchronization error in the vicinity of an optimal $\lambda$ with SNR $\approx 7 dB$ ]{	\label{subfig:lambda2}    \includegraphics[width=0.3\textwidth]{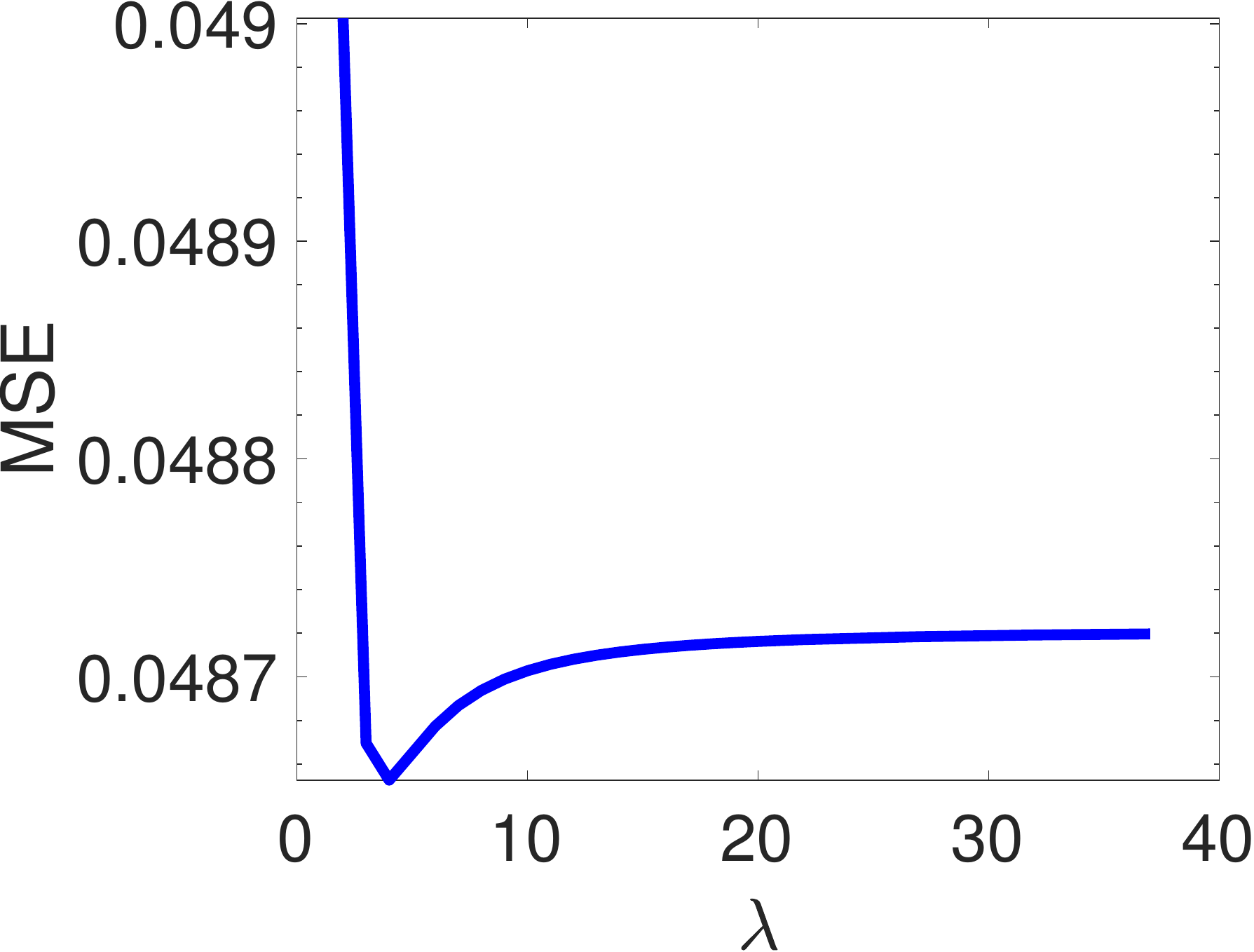}} 
\caption{The $\lambda$ parameter and its effect on synchronization error.}
\label{fig:MSEofLambda}
\end{figure}

Following the above experiments, we suggest the following approach for choosing the $\lambda$ parameter. First, we observe that as the level of noise decreases, such is the effect of different $\lambda$ on the error, that is the sensitivity of our choice decreases too. On the other hand, the search interval of $\lambda$ decreases as noise increase, which facilitates the search for (close to) optimal $\lambda$. This interval is determined from above according to the rotation synchronization in use, and ideally is inversely proportional to (an estimation of) the SNR \eqref{eqn:SNR}. Such an estimation is extracted from statistics on the spectrum of the translations (e.g., by considering the median) and can be further sharpened by having some assumptions on the data, see e.g., \cite[Chapter 11]{jeruchim2006simulation}. From below, we find the maximal norm of the translations of the data and set the lower bound to be $\frac{2}{0.59}$ times this number. This factor would guarantee the two conditions on $\lambda$ from Proposition~\ref{prop:app_homo_contraction} and \eqref{eqn:lambda_cond}. We set the number of synchronization evaluations we are willing to apply, and choose $\lambda$ that minimizes a predefined cost function, such as \eqref{eqn:Cartan_Optim}. In our examples, we follow such a procedure which is also available in the online code for the case of $\se(d)$.

\subsubsection{Optimizing the global alignment in the inner rotations synchronization} \label{subsec:global_alignment}

Proposition~\ref{prop:invarint_back_mapping} characterizes the conditions on $Q\in \Gb$ such that
\[  \Psi_\lambda^{-1} \left( Q_i Q \right) \Psi_\lambda^{-1} \left( (Q_j Q)^{-1} \right) = \Psi_\lambda^{-1} \left( Q_i \right) \Psi_\lambda^{-1} \left( Q_j \right)^{-1}  .\]
This means that a global alignment $Q$ on $\Gb$ also serves as a global alignment in $\Gc$, after the back mapping $\Psi_\lambda^{-1} $. The conditions on $Q$ means that it does not break the Cartan decomposition of $Q_i$. Namely, if $Q_i = \exp(v_i)k_i$ and $Q = \exp(v)k$ are the global Cartan decompositions, their product, under the conditions of Proposition~\ref{prop:invarint_back_mapping}, has the form $\exp(v_i+\ad{k_i}{v}) k_i k$, which is already in its Cartan form since $k_i k \in K$ and $v_i+\ad{k_i}{v} \in \mathfrc{p}$. Therefore, since multiplying from the right by any element from $K$ is invariant for the back mapping, it is clear that for an arbitrary global alignment $Q = \exp(v)k \in \Gb$, the problem is reduced to finding an optimal $v \in \mathfrc{p}$ that minimizes some error after applying $\Psi_\lambda^{-1}$. In other words, we look for a global alignment of the form  $Q = \exp(v)$.

In practice, a fast Cartan decomposition as described in Section~\ref{apx:CartanCalculation} of the supplementary, opens the door for searching a good global alignment. Specifically, using \eqref{eqn:prod_QQi} we have
\[ \Psi_\lambda^{-1} \left( Q_i Q \right) = \Psi_\lambda^{-1} \left( Q_i \exp(v) \right) = \Psi_\lambda^{-1} \left( \exp(v_i)\exp(\ad{k_i}{v}) k_i\ \right)  . \]
Denote the Cartan decomposition of the exponential product by
\[ \exp(v_i)\exp(\ad{k_i}{v})  = \exp(\widetilde{v}_i) \widetilde{k}_i , \quad \widetilde{v}_i \in \mathfrc{p} , \quad \widetilde{k}_i \in K . \]
This decomposition can be efficiently calculated. Similarly, we have 
\[ (Q_j Q)^{-1} = k_j^T \widetilde{k}_j^T \exp(-\widetilde{v}_j) =  \exp(\ad{k_j^T \widetilde{k}_j^T}{-\widetilde{v}_j}) k_j^T \widetilde{k}_j^T .\]
Then, the product of the back mapping is also efficiently computable, and is a function of $v$,
\[ \begin{aligned}[t]
p_{ij}(v)  =  \Psi_\lambda^{-1} \left( Q_i Q \right) \left( \Psi_\lambda^{-1} \left( Q_i Q \right) \right)^{-1} & = \Psi_\lambda^{-1} \left( Q_i Q \right)  \Psi_\lambda^{-1} \left( (Q_i Q)^{-1} \right) \\
  &= \left( \widetilde{k}_i k_i,  \lambda \widetilde{v}_i \right)   \left( k_j^T \widetilde{k}_j^T ,  \lambda \ad{k_j^T \widetilde{k}_j^T}{-\widetilde{v}_j}\right) \\
    &= \left( \widetilde{k}_i k_i k_j^T \widetilde{k}_j^T , \lambda ( \widetilde{v}_i  - \ad{\widetilde{k}_i k_i k_j^T \widetilde{k}_j^T}{\widetilde{v}_j} ) \right) \in \Gc .  
\end{aligned}  \]
Having the above product, we want to optimize
\[ \min_{v \in \mathfrc{p}} \sum_{(i,j) \in \mathcal{E}} w_{ij} d\left(p_{ij}(v), g_{ij} \right) , \] 
where $d(\cdot,\cdot)$ is an appropriate metric on $\Gc$ and $g_{ij}$ are measurements of synchronization.

To keep the implementation of synchronization efficient, we limit this procedure to a fixed, small number of iterations. For cases of large amount of measurement, a further time improvement can be done by adopting a random sampling strategy (``bootstrapping") to approximate the cost function by averaging the results of several partial sum calculations and avoid a large number of computations of Cartan decomposition needed for each iteration. Implementation of such a process is embedded in the code online.

\subsection{Synthetic Data over Special Euclidean Group}  \label{subsec:synthData}

We study the performances of four main methods for synchronization over $\se(d)$; first is our method of synchronization via contraction. Second is the spectral method as done in \cite{arrigoni2015spectral, bernard2015solution}, however, we further accommodate it by allowing a varying scale on the translations as we do in the contraction method. The third is a separation-based method implemented similarly to the approach in \cite{cucuringu2012sensor} (solving first the rotational parts of synchronization and then use the result for estimating the translational parts). To make the comparison more realistic, we strengthen this separation-based algorithm by allowing it to solve the rotational part in the most suitable way, exactly as we do for our method after applying the mapping to rotations. We elaborate more on the choice of rotations synchronization when describing the specific examples. The fourth method is an intrinsic solver of least squares as defined in~\eqref{eqn:Cartan_Optim}. We initiate the least squares algorithm (as it is a local solver) with the solution of our synchronization method to make the scenario more practical. A demonstration of different choices of initial guess and a numerical connection between the least squares cost function and the synchronization error are given in the complementary online code. One common feature to all of those methods is their availability for data with any given dimensionality $d$. For more details on the abovementioned approaches see Subsection~\ref{subsec:available_solutions_se_group}.

In the sequel, we generate synthetic random data as our ground truth elements, to be estimated by the synchronization algorithms. Every element is generated by projecting to $\so(d) $ a $d \times d$ matrix with entries chosen uniformly over $[0,1]$, and a translation part with entries uniformly picked from $[0,2]$. Having the ground truth data, we generate the relative relations $g_ig_j^{-1}$ to serve as data for the synchronization problem. The three obstacles that we artificially create are missing data, noise, and outliers in data. In the different scenarios, we mix these conditions in several ways to imitate real-world conditions.

In the first set of examples, we examine the effect of missing data, noise, and their combination on the performance of the synchronization algorithms. A
common choice for the noise is to assume it is Gaussian. However, there is more than one way to define such a distribution over a given Lie group. One approach is to define the PDF using the Euclidean multivariate Gaussian PDF on the Lie algebra (in terms of manifolds, on the tangent plane of the identity). This is done by clipping the tail of the Gaussian function or by folding the tails, also known as ``wrapped Gaussian", for more details see e.g., \cite[Chapter 2]{chirikjian2011stochastic}. This distribution, which is closely related to the von Mises distribution on the circle, facilitates the implementation of normal noise and therefore we adopt it in our experiments. To evaluate the noise we use Signal-to-Noise Ratio (SNR), which is a standard measure in signal processing. Given a data sample, the SNR is defined as the ratio between deviations in data and deviations in noise. Also, it is usually measured in logarithmic scale that is decibels (dB) units. Since we use Gaussian noise over the Lie algebra, we measure the noise level there too. To be specific, recall our noisy measurements $g_{ij}$ are of the form\eqref{eqn:noisemodel}, we use
\begin{equation} \label{eqn:SNR}
 \text{SNR}\left( \{ g_{ij} \}_{(i,j)\in \mathcal{E}} \right) = \frac{20}{|\mathcal{E}|} \sum_{(i,j)\in \mathcal{E}} \log_{10} \left( \norm{\log(g_ig_j^{-1})}/\norm{\log(N_{ij})}  \right) , 
\end{equation}
where $|\mathcal{E}|$ stands for the number of edges in $\mathcal{E}$. Measuring norms in the Lie algebra is also justified as it stands for the geodesic distance from the identity.

The other difficulty that we artificially create in our database is formed by making some of the measurements non-available. Potentially, for a set of $n$ unknown group elements, one has $\binom{n}{2}$ different meaningful measurements $g_{ij}$. However, it is often the case in applications that not all measurements are available. In the first set of examples, we consider cases with only a fraction of $p \binom{n}{2}$, $p \in [0.05,0.3]$ available measurements. 

Figure~\ref{fig:MissingDataWnoise} presents the first abovementioned set of example of synchronization over a set of measurements with missing data and noise. The ground truth are elements over the group $\se(d)$ of rigid motions over $\mathbb{R}^5$) with $d=3$ and $d=5$, in two sets of lengths $n=100$ and $n=200$. For the rotation synchronization in both the contraction and separation algorithms, we used the spectral method (see Algorithm~\ref{alg:eigenvector_method}). To demonstrate the flexibility in the contraction method, we also show in the first two examples the gain from choosing the rotation synchronization for the contraction method by combining it with the MLE solver. IN addition, the contraction based method uses automatic $\lambda$ picking, that typically outputs values in $50-120$ for those examples. Note that we also use $\lambda$ to scale the translational parts for the spectral method as well, to factor out the scaling from the performance comparison. Figure~\ref{subfig:miss1} depicts the synchronization error~\eqref{eq:PerfMeasure} as a function of the fraction of available data, contaminated with wrapped Gaussian noise of about $\approx 12 dB$. In this first example, we use the data set of $n=100$ group elements. We see that the performance of the least squares and separation methods is almost the same, and worst than the spectral and contraction methods. The contraction based methods show slightly superior error rates, especially as the amount of available data drops. Next, we zoom in on a fixed rate of $10\%$ of available data and increasing the noise level (recall that lower SNR means higher noise level). The error rates of this scenario are shown in Figure~\ref{subfig:miss2}, where we can clearly see the difference between the methods of separation and least squares (which give higher error rates) and the two other approaches.  Indeed, as the scenario becomes more challenging, i.e. more noisy data, the contraction methods perform better. We repeated the last example, and increase even more the noise level with a data set of $n=200$ group elements in the higher dimension group $\se(5)$. Figure~\ref{subfig:miss4} provides the error rate for this case, where more information is obtained with the same percentage of available data compared to the case of $n=100$ (recall that the amount of measurements grows quadratically with $n$). Once again, upon further pushing the SNR to about $\approx 8 dB$, the advantage of using contraction-based method is pointed out.
\begin{figure}
\centering    
\subfloat[Varying level of missing data with a fixed level of added noise]{	\label{subfig:miss1}    \includegraphics[width=0.3\textwidth]{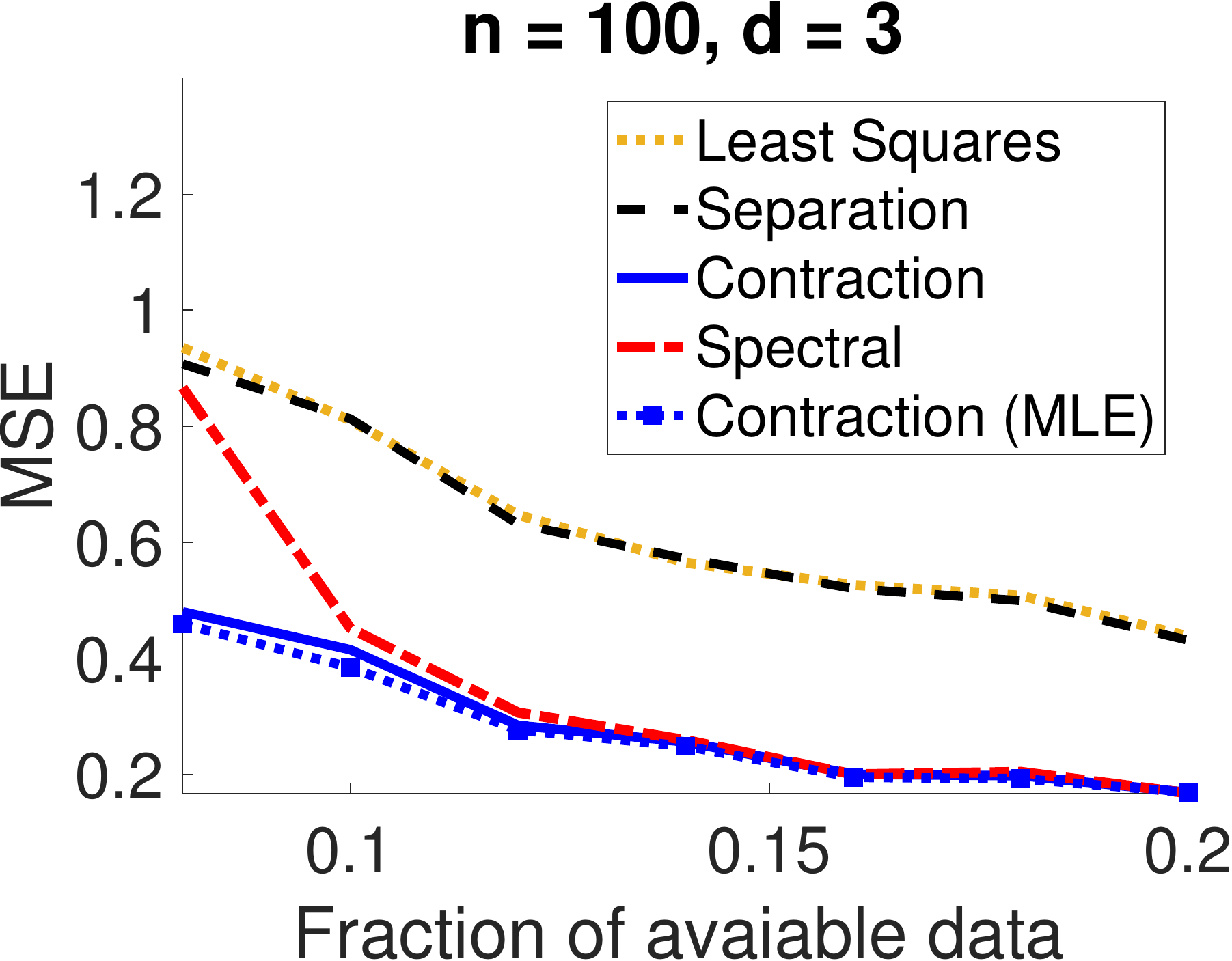}} \quad
\subfloat[A fixed level of missing data ($10\%$) and varying noise]{	\label{subfig:miss2}    \includegraphics[width=0.3\textwidth]{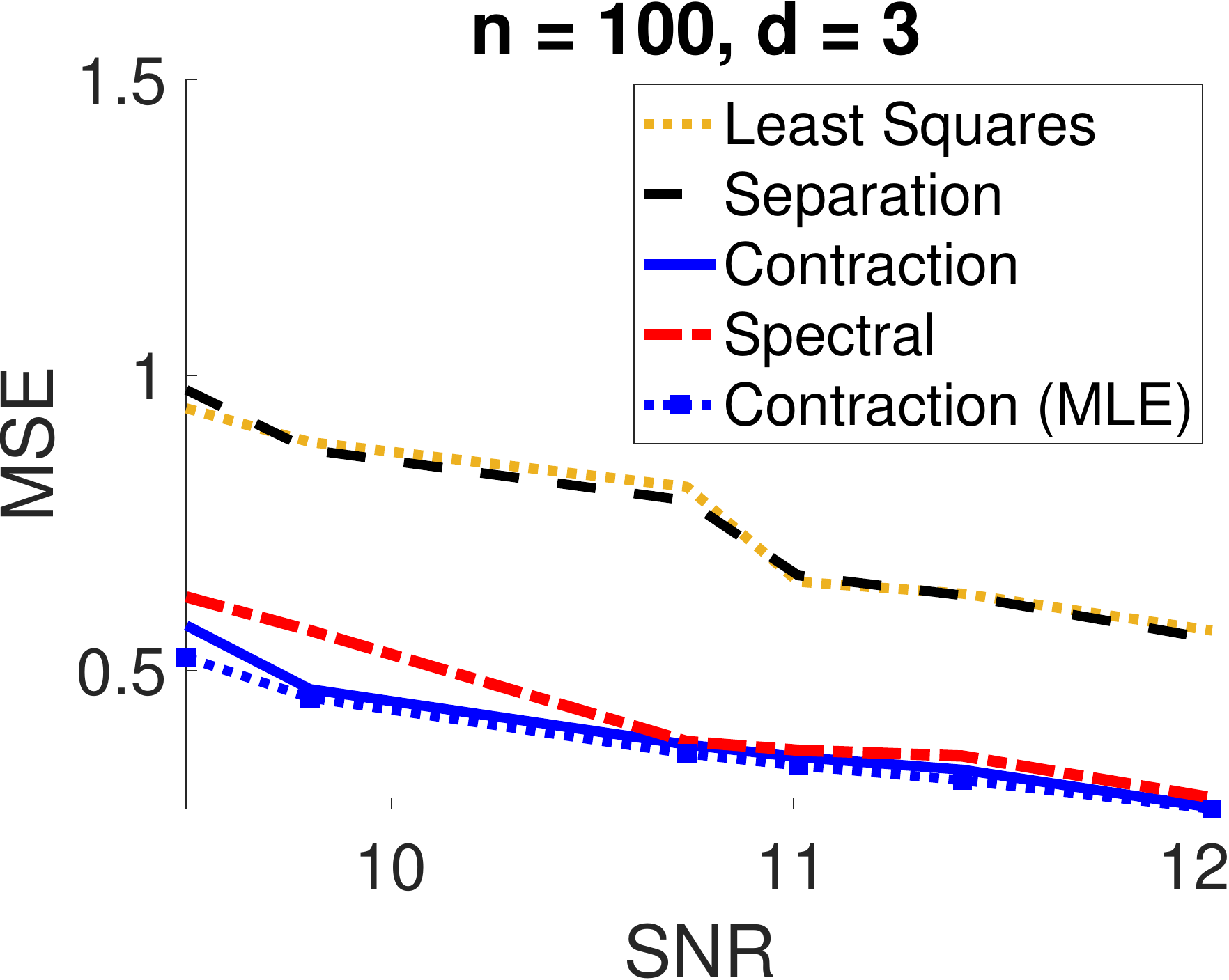}} \quad
\subfloat[A fixed level of missing data ($10\%$) and varying noise in higher dimension]{	 \label{subfig:miss4}   \includegraphics[width=0.3\textwidth]{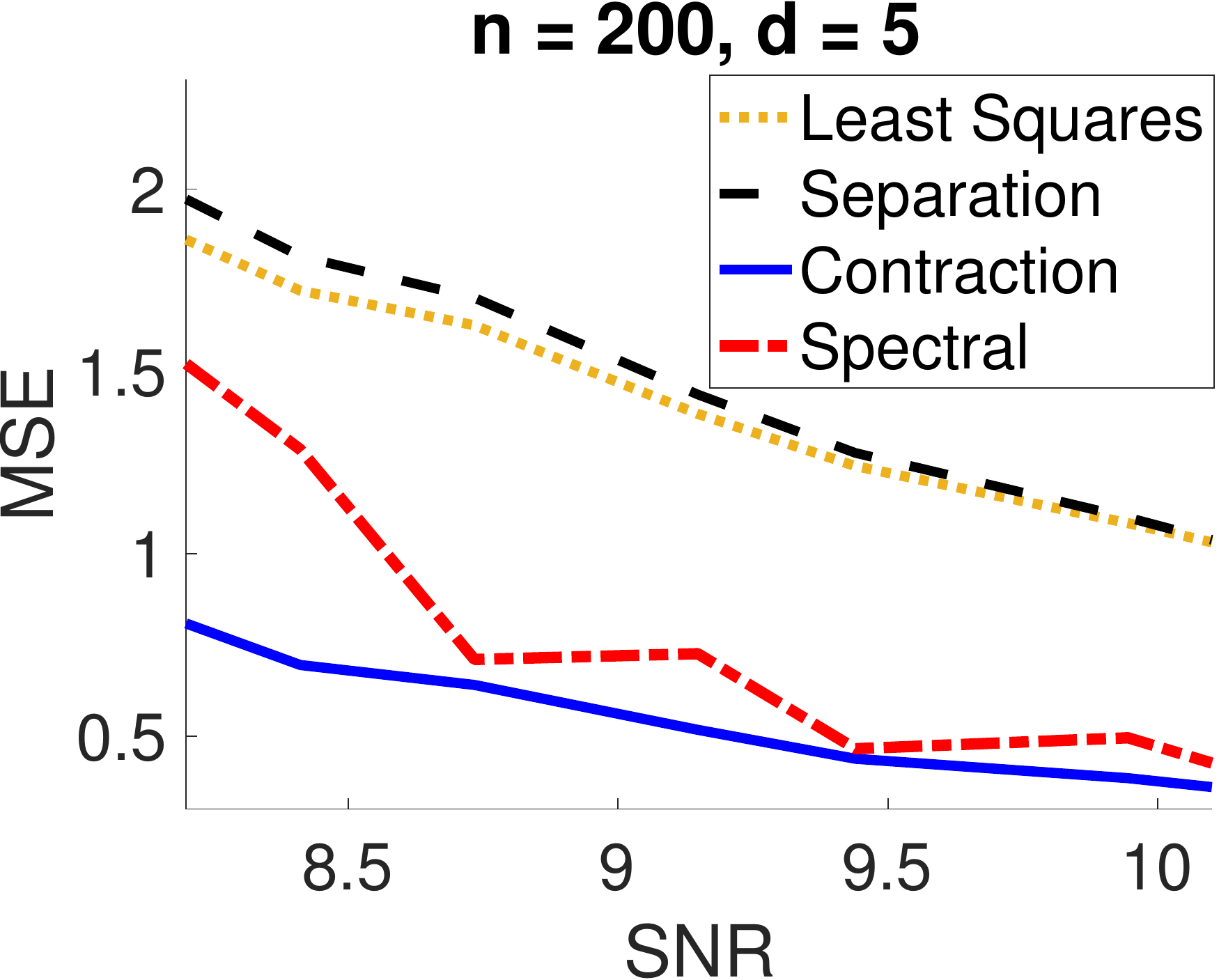}}   
\caption{Comparison in cases of missing data and with and without the presence of noise. If not stated otherwise, the rotation synchronization is done using spectral method (EIG).}
\label{fig:MissingDataWnoise}
\end{figure}

The second set of examples deals with the case of outliers in data. Each outlier is a randomly chosen $\se(d)$ element, that replaces a certain measurement $g_{ij}$. We generate the outlier to be distinguished from the other measurements. Specifically, the rotational part is a projection of a square matrix of order $d$ with standard normal distributed entries to $\so(d)$, and a $d$ size vector of uniform entries over $[0,1]$ for the translational part. These examples are done for $d=2$, that is rigid motions in $\mathbb{R}^2$. In order to efficiently address cases of outliers, we use for rotations synchronization the LUD algorithm \cite{wang2013exact}, minimizing $\ell_1$ cost function. This is true for the contraction method as well as for the separation-based method. In addition, for separation method we also applied $\ell_1$ minimization for the translation part too, when applicable, that is for $n=100$. The contraction method uses $\lambda = 220$ for this set of examples. The results are given in Figure~\ref{fig:OutliersSD}, where we examine the error rates as a function of outliers, where the $x$-axis in each plot is the non-outliers rate, that is the fraction of available genuine data. The range of outliers is $10\%-60\%$, meaning $90\%$ down to $40\%$ of good measurements. 

In the first examples, we use $n=100$ group elements and a model of only outliers. The results are presented in Figure~\ref{subfig:outl1}, where we observe significant advantage to the methods of contraction and separation. Nevertheless, for values of around $40\%$ of outliers and more, the contraction method is superior. With the same $n$, we change the model and add to the measurements a fixed level of Gaussian noise, resulting in about $\approx 12 dB$ of SNR. In this case, all methods gain more error, where the contraction method still shows very accurate estimations till approaching $50\%$ of outliers, as seen in Figure~\ref{subfig:outl2}. We repeated those examples, now with $n=300$ group elements. Since $\ell_1$ minimization for the translation part of the separation method became computationally intractable (the constraints are in a matrix of size $n^2 \times n^2$), we use a standard least squares solver there. In the presence of noise, this modification turns to have little to no effect on the overall performance of the separation. The results in the ``only outliers" scenario are given in Figure~\ref{subfig:outl3}, where the effect of using LUD is significant and provides excellent results for the contraction method. Then, we add noise to the measurements, of about $\approx 9 dB$. The result, given in Figure~\ref{subfig:outl4} show better performance of contraction for any tested value of outliers rate and demonstrate the advantage of using the contraction for this mixed model of noise.

\begin{figure}
\centering    
\subfloat[Varying level of outliers in data]{	 \label{subfig:outl1}   \includegraphics[width=0.3\textwidth]{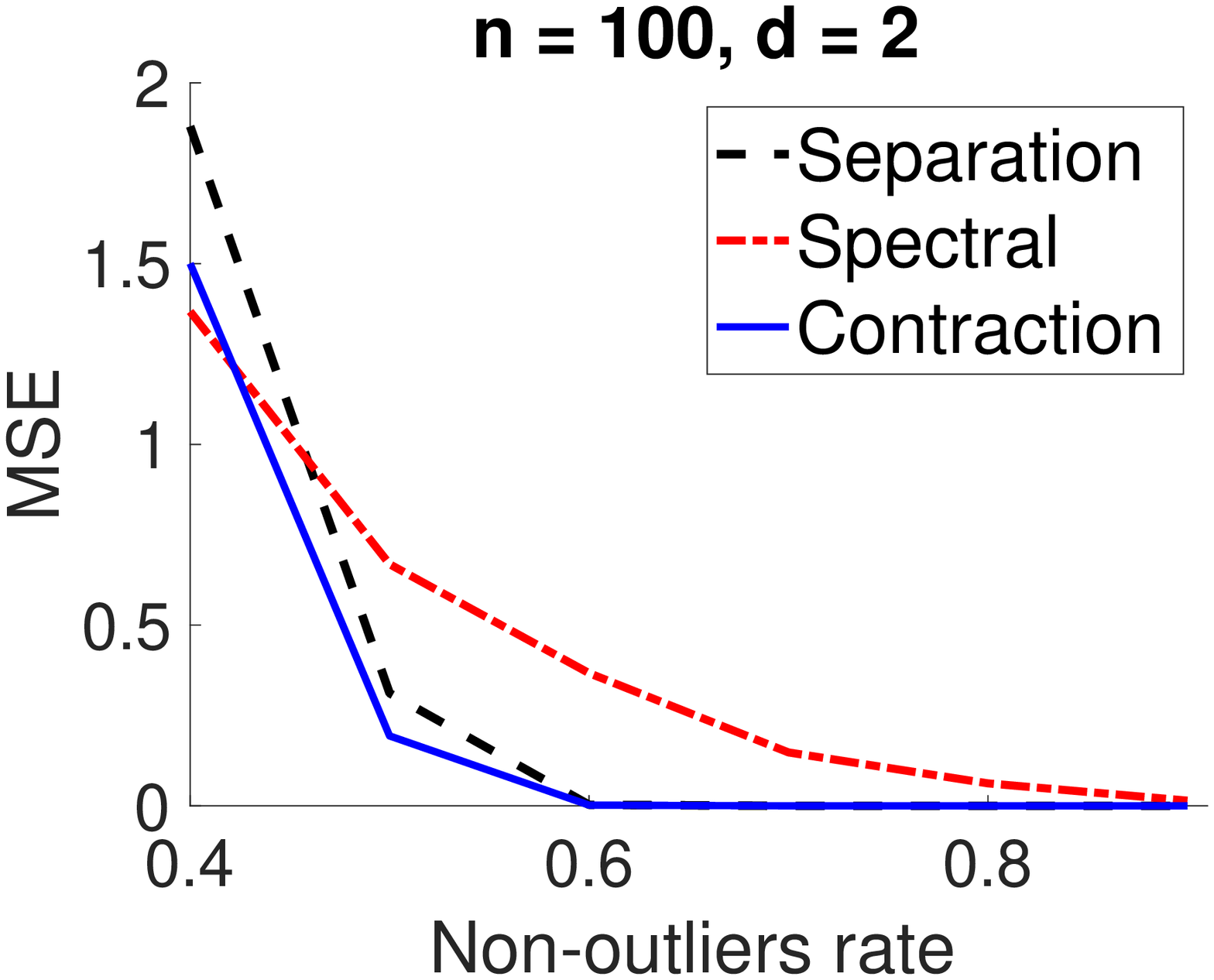}} \qquad
\subfloat[Varying level of outliers in data with a fixed level added noise]{	\label{subfig:outl2}    \includegraphics[width=0.3\textwidth]{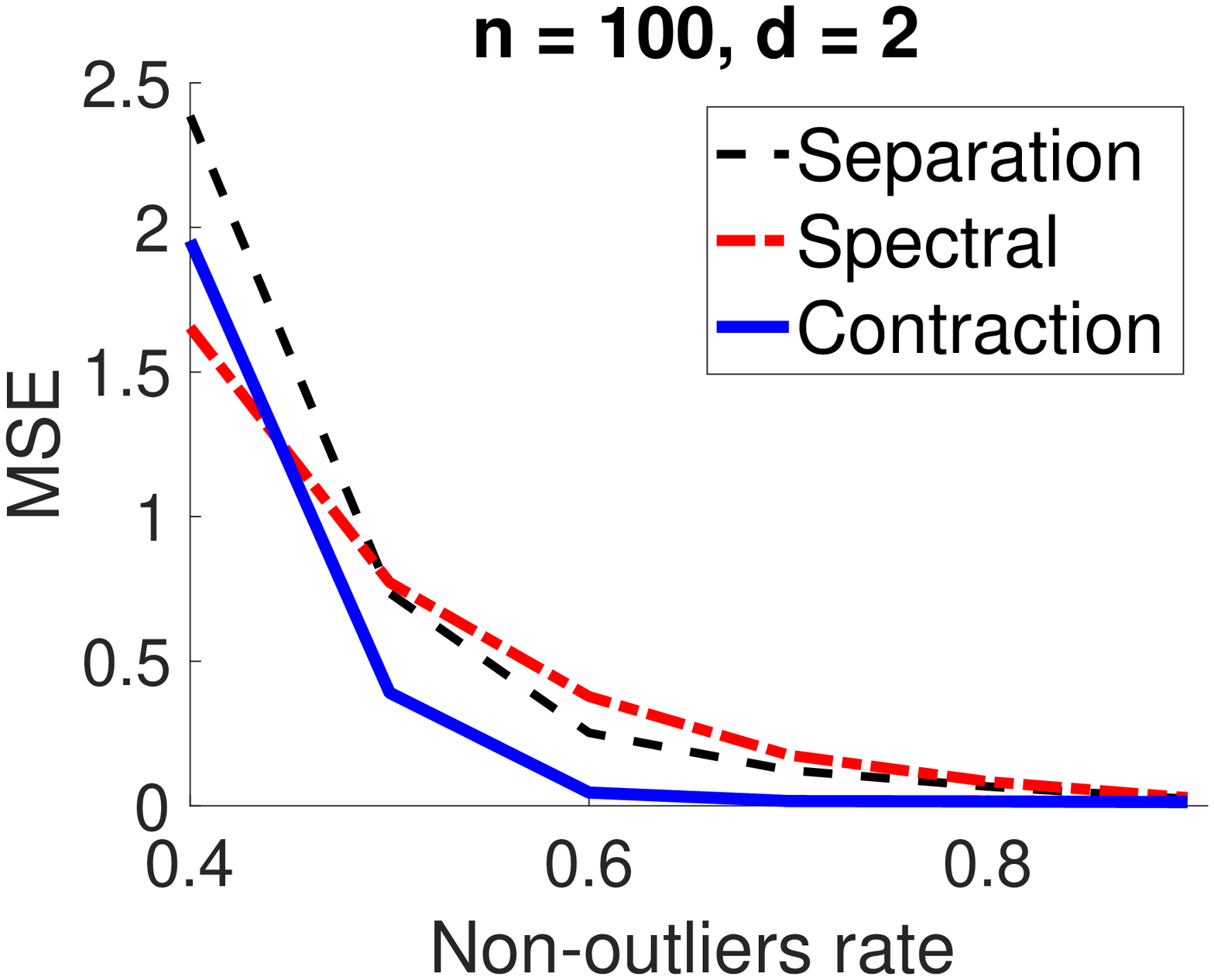}}  \\
\subfloat[Varying level of outliers in data]{	  \label{subfig:outl3}  \includegraphics[width=0.3\textwidth]{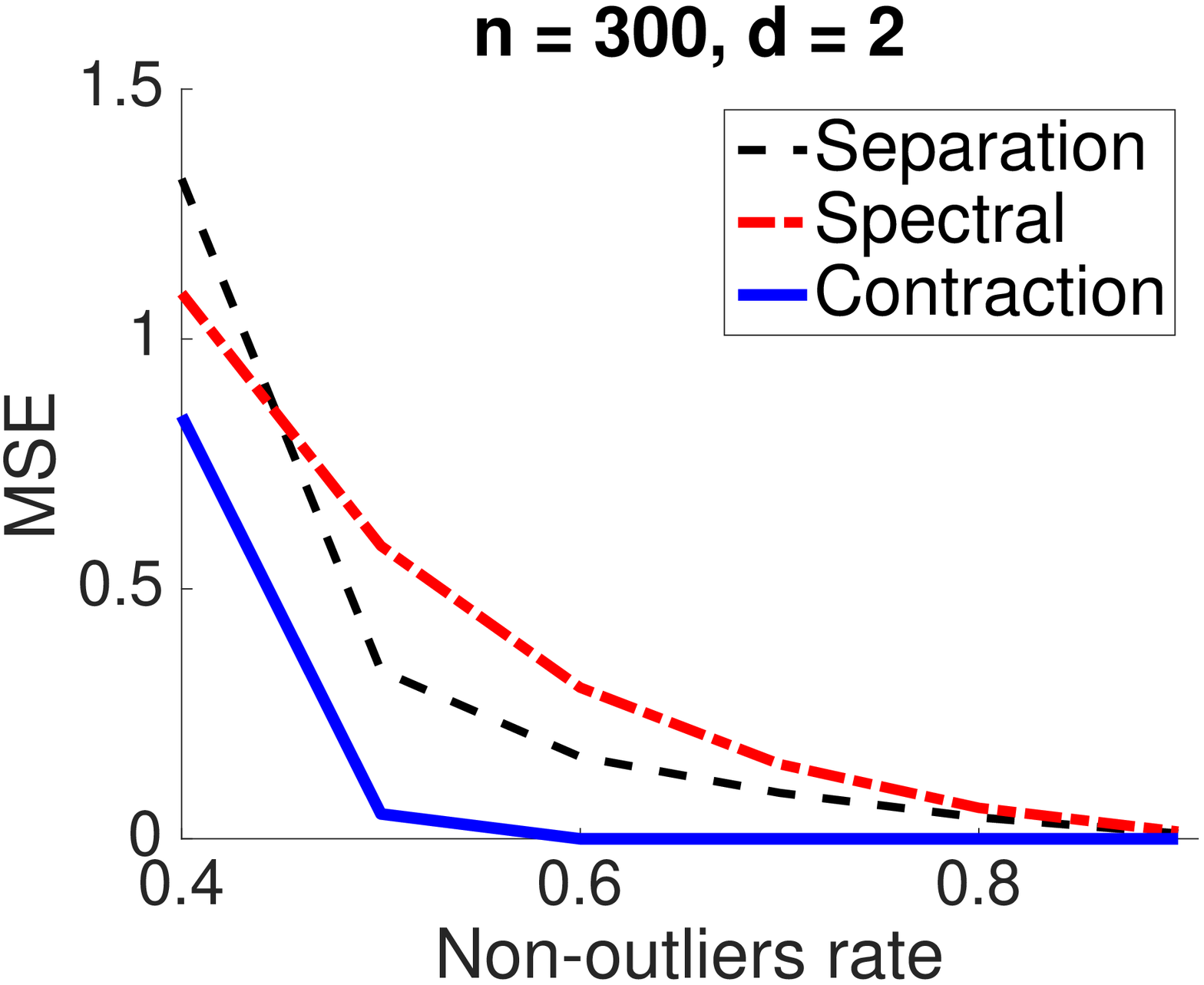}}  \qquad
\subfloat[Varying level of outliers in data with a fixed level added noise]{	\label{subfig:outl4}    \includegraphics[width=0.3\textwidth]{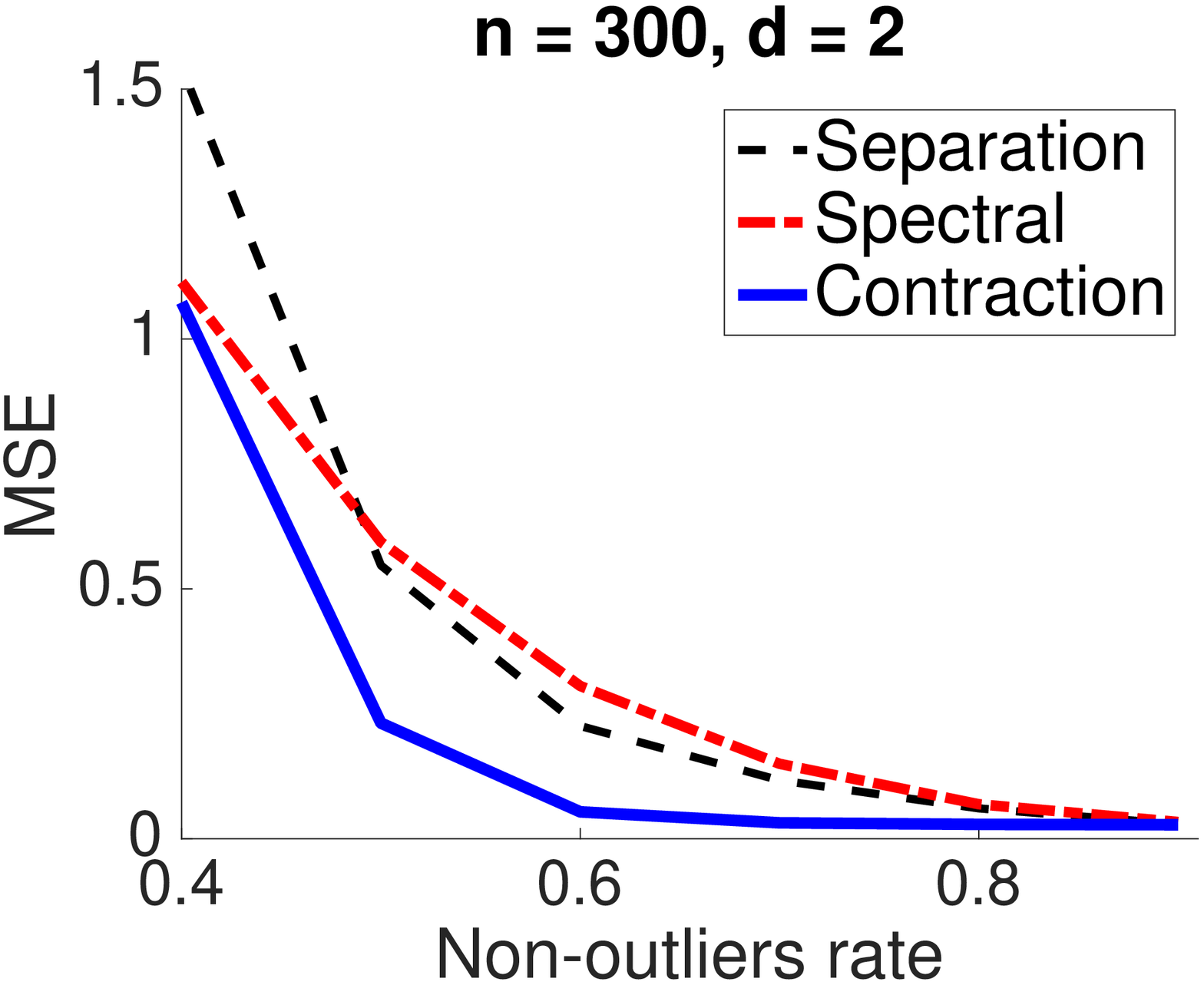}}   
\caption{Comparison in case of outliers. }
\label{fig:OutliersSD}
\end{figure}


\subsection{Synchronization over Matrix Motion Group} \label{subsec:MMG}

The group of matrix motion, 
\[ \mmg(d, \ell) = \left( \on(d) \times \on(\ell) \right) \ltimes \m(d,\ell), \] 
is briefly mentioned in Example~\ref{exm:MMG}, where we present how the Cartan decomposition of $\on(d + \ell)$ leads to the construction of this Cartan motion group. In this group, the ``motion part" is a matrix of size $d \times \ell$, comparing to the translation vector of length $d$ in $\se(d)$. The ``compact part" of $\mmg(d, \ell)$ is also more general than in $\se(d)$ as it consists of products of orthogonal matrices, which in addition does not necessarily have positive determinants. In this point, the difference becomes even more significant, as the adjoint representation defines the action of the group \eqref{eqn:Cartan_action} to be
\[ \left(\mu_1, \eta_1, B_1 \right)\left(\mu_2, \eta_2, B_2 \right) = \left(\mu_1\mu_2, \eta_1\eta_2, \mu_1B_2 \eta_1^T + B_1 \right) , \]
which makes it harder to write the group elements in terms of matrices. Note that the adjoint $(\mu,\eta) \mapsto \mu B \eta^{T}$ of $\on(d) \times \on(\ell)$ on $\m(d,\ell)$ is followed directly by $\left[ \begin{smallmatrix} \mu & 0 \\ 0 & \eta \end{smallmatrix} \right] \left[ \begin{smallmatrix} 0 & B \\ B^T & 0 \end{smallmatrix} \right]\left[ \begin{smallmatrix} \mu & 0 \\ 0 & \eta \end{smallmatrix} \right]^{-1} $. The inverse is then 
\[ \left(\mu, \eta, B \right)^{-1} = \left(\mu^T, \eta^T, -\mu^T B \eta \right)  .\]

The contraction map $\Psi_\lambda \colon \mmg(d, \ell) \mapsto \on(d+\ell)$  takes the form of
 \[ \Psi_\lambda \left(\mu, \eta, B \right) = \exp \left( \left[ \begin{smallmatrix} 0 & B_\lambda \\ -B_\lambda^T & 0 \end{smallmatrix} \right]  \right) \left[ \begin{smallmatrix} \mu & 0 \\ 0 & \eta \end{smallmatrix} \right] , \quad B_\lambda = \frac{1}{\lambda} B .  \]
The back mapping $\Psi_\lambda^{-1}$ is locally unique when $ \exp \left( \left[ \begin{smallmatrix} 0 & B_\lambda \\ -B_\lambda^T & 0 \end{smallmatrix} \right]  \right)$ lies near the identity, that is for $B_\lambda$ small enough. In practice, we calculate $\Psi_\lambda^{-1}$ using a gradient based optimization, as described in Section~\ref{apx:CartanCalculation}, and based on this uniqueness for large enough $\lambda$.

To measure the synchronization error we also use \eqref{eq:PerfMeasure}, where the global alignment $g \in \mmg(d, \ell)$ minimizes the separated loss function
\[ \sum_{i=1}^n \norm{ \hat{g}_i g - g_i }_F^2    =  \sum_{i=1}^n \norm{ \hat{\mu}_i\mu - \mu_i }_F^2 + \norm{ \hat{\eta}_i\eta - \eta_i }_F^2 + \norm{ \hat{\mu}_i B  \hat{\eta}_i^T + \hat{B}_i - B_i }_F^2 . \]
The last term which involves $B$, is solved in a least squares fashion, where the first two summands are calculated using the SVD decomposition, similar to the case of \eqref{eq:PerfMeasure}.

As a benchmark for our method, we implement a separation-based method for $\mmg(d, \ell) $. In this method, we first solve \eqref{eq:NonCompactOptimPrb1} separately for the two orthogonal parts, by applying rotation synchronization. Then, for the ``translational" parts $\{B_i\}_{i=1}^n$, we use a least square solution defined as follows. Denote each unknown group element by $g_i = \left(\mu_i, \eta_i , B_i \right)$ and by $g_{ij} = \left(\mu_{ij}, \eta_{ij} , B_{ij} \right)$  an element of the data. Then, the ``translational" parts are the minimizers of
\[  \sum_{i,j} w_{ij} \norm{B_{ij} \eta_j +\mu_i \mu_j^T B_j + B_i}_F^2 . \]
Other methods, such as the spectral method, heavily depend on the matrix representation one chooses for the group, and are thus omitted at this point from the numerical demonstration. For simplicity, both separation and contraction methods use for rotation synchronization the spectral algorithm. We compare the performances of the two methods over three scenarios: added noise, missing data with a fixed noise level, and noisy data further contaminated with outliers. The results are depicted in Figure~\ref{fig:MMGresults}. As the ground truth, we randomly generated $n=60$ elements in $\mmg(4,3)$, and make synchronization data according to the different scenarios. The contraction map uses $\lambda=50$ for all scenarios. In Figure~\ref{subfig:noisy}, we use full graph data, where each measurement is calculated with a multiplicative noise. For each test, we use noise with different variance and so the x-axis represents a growing amount of noise as the variance grows. The difference between the separation and contraction methods becomes significant as the noise increases. Figure~\ref{subfig:missing} displays the results for synchronization data consists of only partial data. A similar behavior of the two methods is observed, with a fixed advantage to the results of the contraction method. In figure~\ref{subfig:outliers}, we add outliers instead of the missing data of the previous scenario (outliers were drawn randomly and separately on each component of the group element). The advantage of the contraction is yet again shown.

\begin{figure}
\centering    
\subfloat[Different noise levels]{	  \label{subfig:noisy}  \includegraphics[width=0.3\textwidth]{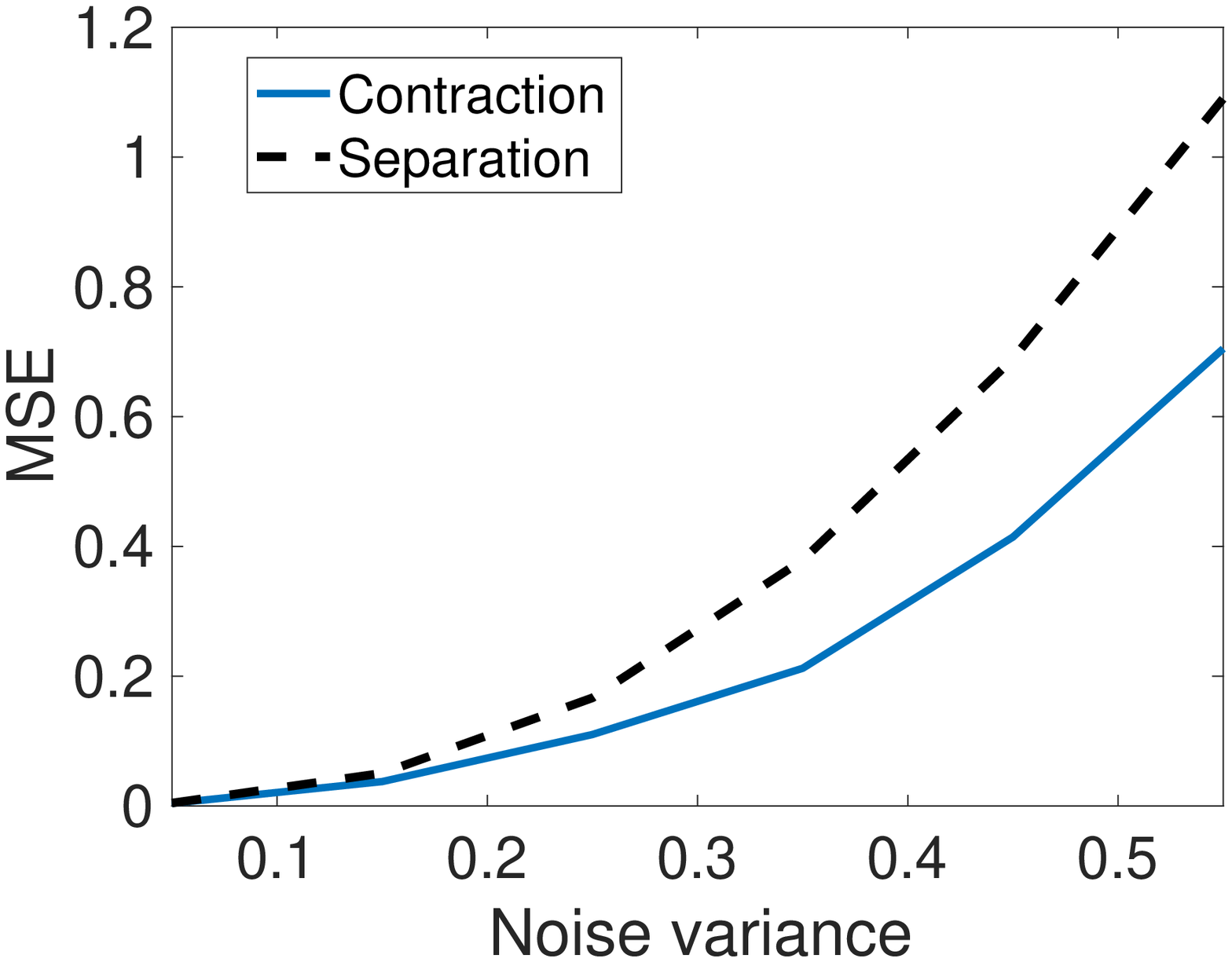}} \quad
\subfloat[Missing data with noise]{	\label{subfig:missing}    \includegraphics[width=0.3\textwidth]{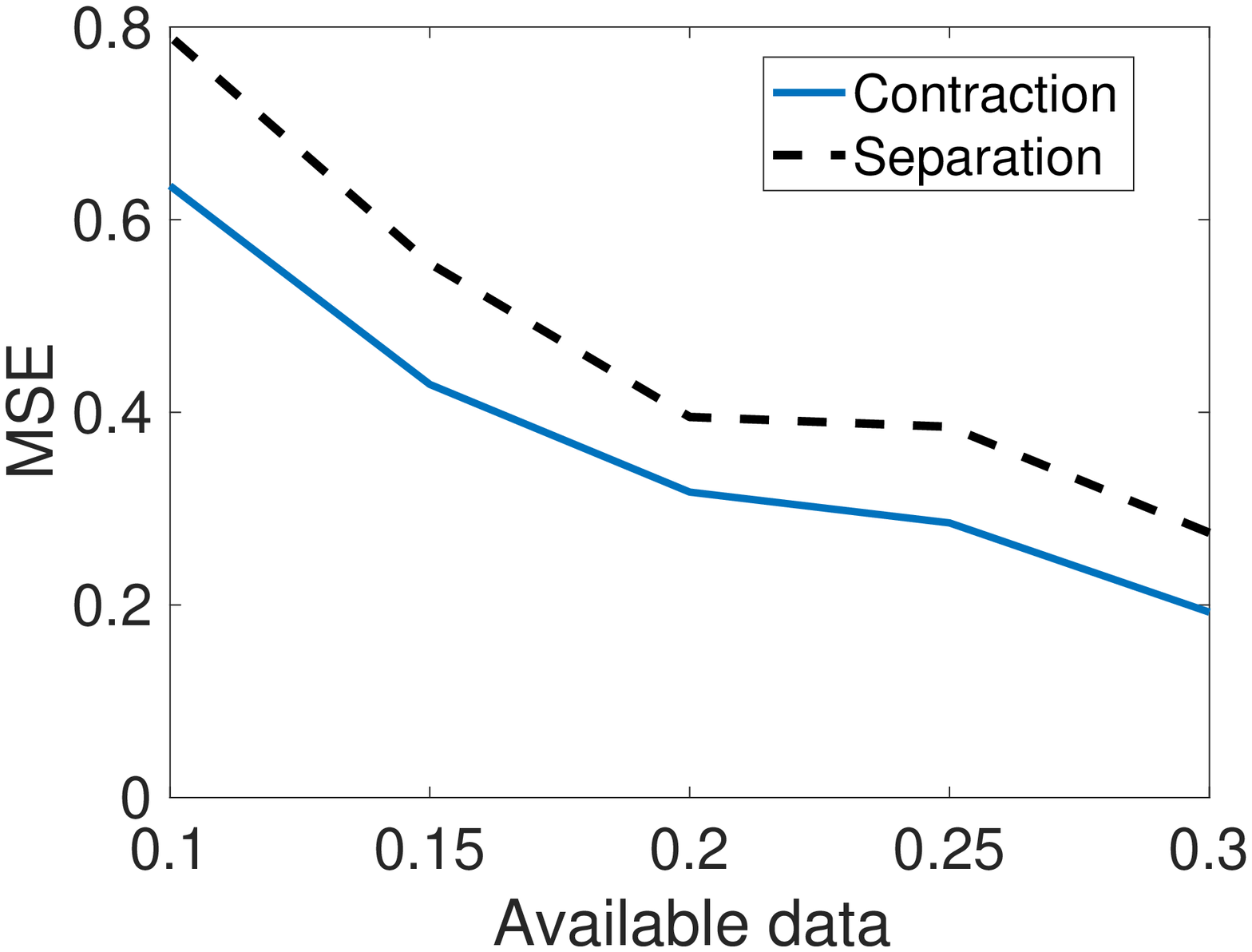}}  \quad
\subfloat[Outliers with noise]{	\label{subfig:outliers}    \includegraphics[width=0.3\textwidth]{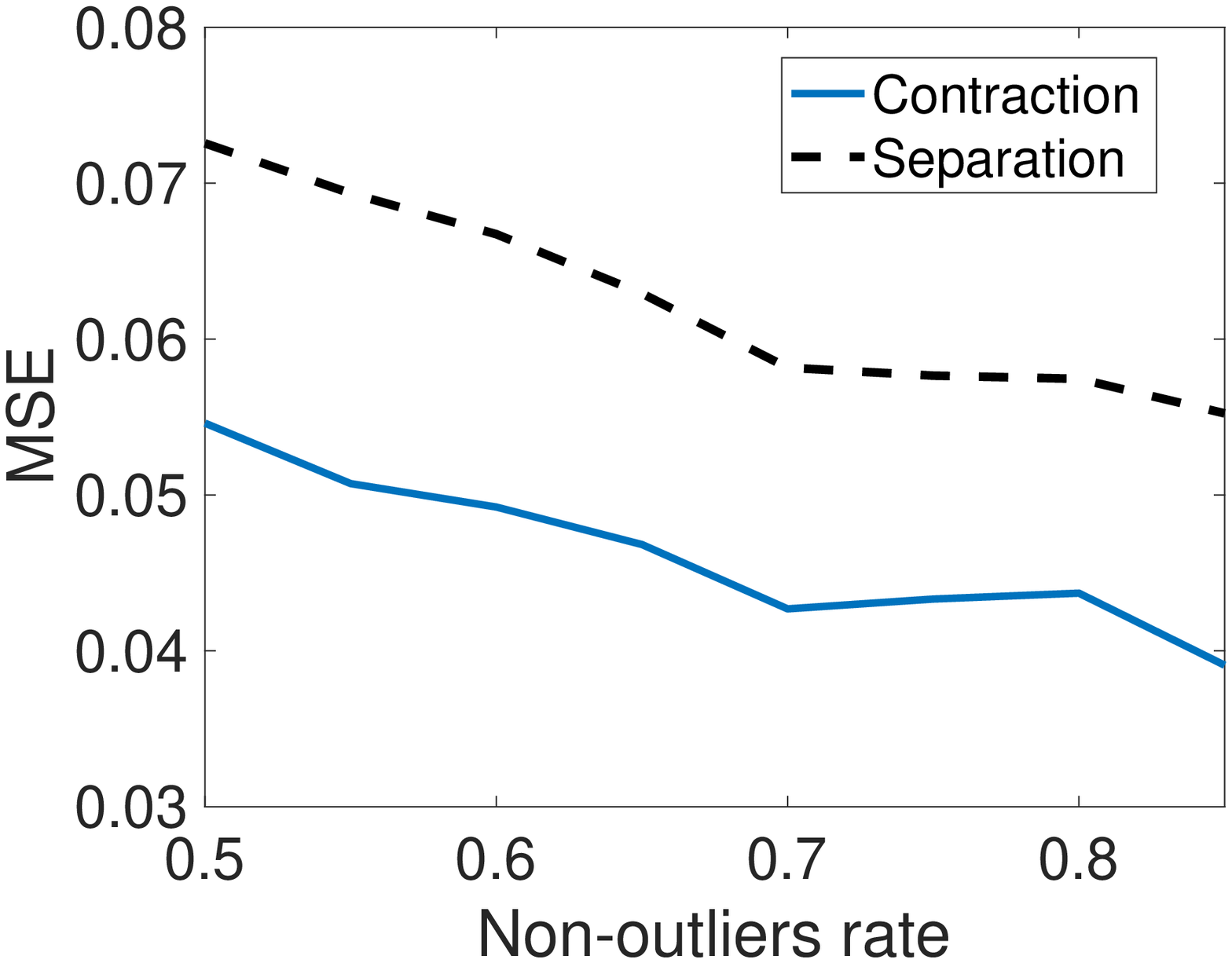}} 
\caption{Three comparisons of two synchronization methods over $\mmg(4, 3)$ with $n=60$, separation-based versus contraction method. The comparisons, from left to right: varying noise level (x-axis corresponds to noise variance), a fixed level of noise with different percentage of available data, and a varying amount of non-outliers contaminated with a fixed level of noise.}
\label{fig:MMGresults}
\end{figure}

\subsection{Examples of real data --- 3D point cloud registration}

We further demonstrate the performance of our synchronization method for real data. The problem we address is multiple point-set registration in $\mathbb{R}^3$. In this problem we are given point clouds that represent raw scans of a three dimensional object. We aim to recover the transformations required to bring each range scan into a single coordinate system, in order to form an estimation to the original object. In our examples we have the original model object so we can calculate the ground truth transformations and compare them to the results we get from the different algorithms being tested. 

The recovery process of the 3D object begins with estimating the relative motions between pairs of scans, computed by the Iterative Closest Point Algorithm (ICP), see e.g., \cite{rusinkiewicz2001efficient}. In particular, for any two sets of point clouds we apply the ICP algorithm for several initial guesses, as done e.g., in \cite{govindu2014averaging, torsello2011multiview}. The ICP algorithm has two outputs, a rigid motion (an element from $\se(3)$) that transforms one point cloud to the other, and an estimated error of the mutual registration after applying the transform. We use the latter as an indicator for determining whether to include a given measurement or relative motion or not. This should be done carefully in order to eliminate both measurements of large deviations as well as outliers measurements. Since most of raw scans have small part of intersection, if any, just adding measurements would not improve the registration results but the opposite.

In general, there are two main approaches for solving the multiple point-set registration. First is to optimize a cost function that depends on bringing together corresponding points, e.g., \cite{pulli1999multiview}. The other approach, also known as frame space approach, is to solve the relations (rotations and translations) between mutual scans, e.g., \cite{sharp2004multiview}. We follow the second approach in order to highlight the synchronization algorithms and their performance for addressing the 3D registration problem. We compare the performance of six algorithms, including two variants of our contraction algorithm. Apart from the synchronization via contraction, the participanting algorithms in the examples are: the spectral method and separation-based method, as done in the previous subsection. Another method is the diffusion method in \cite{torsello2011multiview} that is based upon the relation of $\se(3)$ and dual quaternions, we termed this method QD in the tables \footnote{The code is courtesy of the authors of \cite{torsello2011multiview}, see \url{https://vision.in.tum.de/members/rodola/code}}. We also use the linear variant of contraction, done in the matrix level and based on polar decomposition, see Subsection~\ref{subsec:matrix_projection}. Both contraction methods as well as the separation methods use for rotations synchronization the maximum likelihood estimator (MLE) from \cite{boumal2013robust}. At last we also include synchronization via contraction with the spectral method (Algorithm~\ref{alg:eigenvector_method}).

For our first example we use the frog model taken from the shape Repository of the Visualisation Virtual Services \url{http://visionair.ge.imati.cnr.it/ontologies/shapes}, the model ID is ``269-frog-range images". This model consists of $24$ scans, each is a point cloud consists of between $24,000$ to $36,000$ points. The full model of the frog is presented in Figures~\ref{subfig:FrogView1}--\ref{subfig:FrogView2}, where the entire set of scans, each under different color, is depicted in Figure~\ref{subfig:FrogSamples}. We repeat the comparison for two different sets of measurements; the first one has $80$ pairwise measurements (from a total of potential $276$, about $29\%$) and the second has $163$ pairwise measurements. These measurements were chosen to minimize the errors according to the estimated error of each measurements, as generated by the ICP algorithm. This also means that adding more measurements increases the noise level. The results are given both numerically in Table~\ref{tab:resultsFrog} and visually in Figures~\ref{fig:FrogModelresults}--\ref{fig:FrogModelHighlights}. In Table~\ref{tab:resultsFrog} we notice that the three leading methods are the two contraction methods (both the group level contraction and the matrix level contraction) and the separation method. All three methods use MLE rotations synchronization, with $\lambda = 400$ for the contraction map and $\lambda=200$ for the projection map. The method based on group contraction shows a slightly better numerical error, for both sets of measurements, however the visual differences in shape reconstruction is minor. The reconstructions (with the exception of the one of QD which was too erroneous) are given in Figures~\ref{subfig:FrogCont} --\ref{subfig:FrogConteig}, where the colors correspond to different scans, as shown in Figure~\ref{subfig:FrogSamples}. The error of the synchronization via contraction, that is based on EIG with $\lambda=850$, shows larger errors but still leads to a good visual reconstruction. However, both the spectral method and the diffusion method (dual quaternions) yield higher errors which were also very significant visually and do not provide good reconstructions, see Figure~\ref{fig:FrogModelHighlights}. Last interesting result is that the additional measurements of the second set of data did not lead to better results, probably due to the extra noise it adds to the data. This is true for all methods except for the method based on group contraction combined with the spectral method for rotation synchronization which shows improvement in error.

\begin{figure}
\centering    
\subfloat[Original object first view]{	 \label{subfig:FrogView1}   \includegraphics[width=0.25\textwidth]{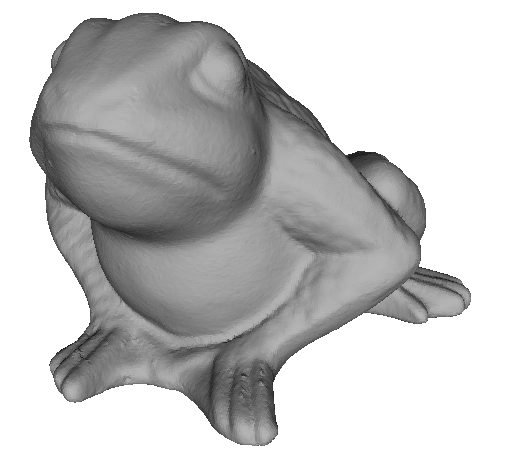}} \quad
\subfloat[Original object second view]{\label{subfig:FrogView2}   \includegraphics[width=0.25\textwidth]{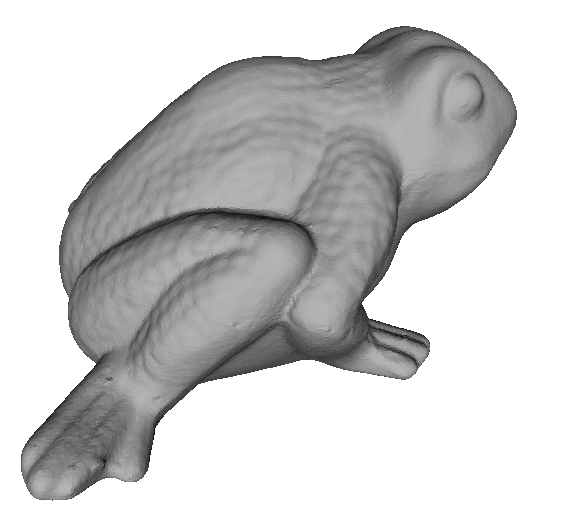}} \quad
\subfloat[Colored raw scans]{				 \label{subfig:FrogSamples}    \includegraphics[width=0.35\textwidth]{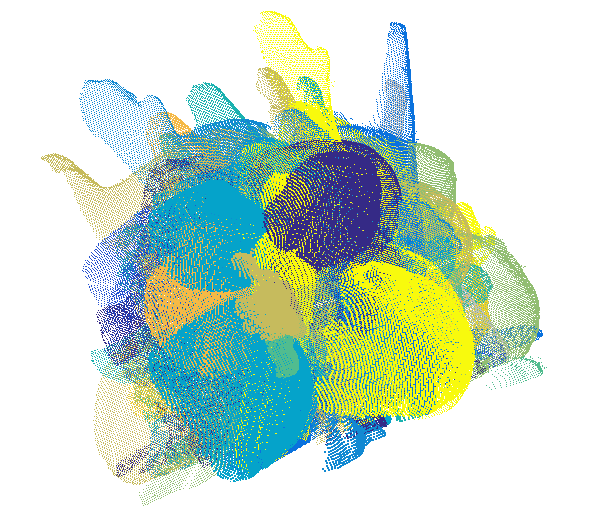}}   \\
\caption{The ``Frog" model: full model views and the raw scans colored.}
\label{fig:FrogModel}
\end{figure}
 
 \begin{figure}
\centering    
\captionsetup[subfloat]{labelformat=empty}
\subfloat[Contraction (MLE)]{   \label{subfig:FrogCont}   \includegraphics[width=0.18\textwidth]{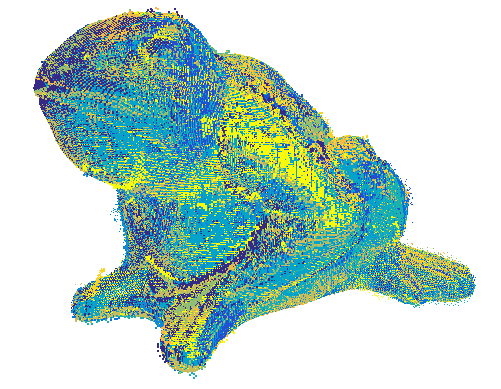}} 
\subfloat[Separation (MLE)]{		\label{subfig:FrogSept}   \includegraphics[width=0.18\textwidth]{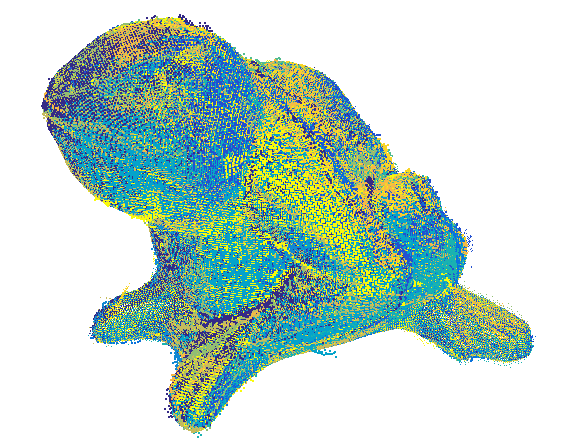}} 
\subfloat[PD (MLE)  			 ]{	    \label{subfig:FrogPD}    \includegraphics[width=0.18\textwidth]{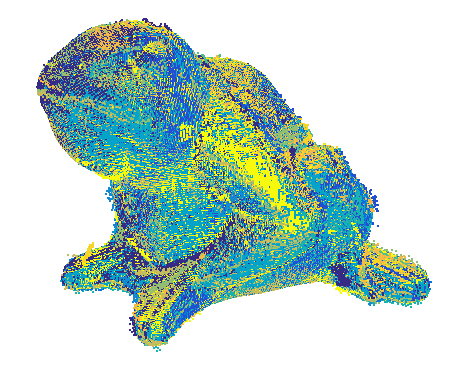}}   
\subfloat[Spectral  		    ]{	    \label{subfig:FrogSpec}   \includegraphics[width=0.18\textwidth]{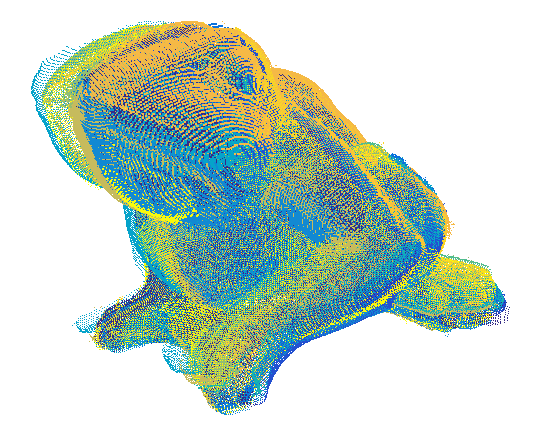}} 
\subfloat[Contraction (EIG)]{	    \label{subfig:FrogConteig}   \includegraphics[width=0.18\textwidth]{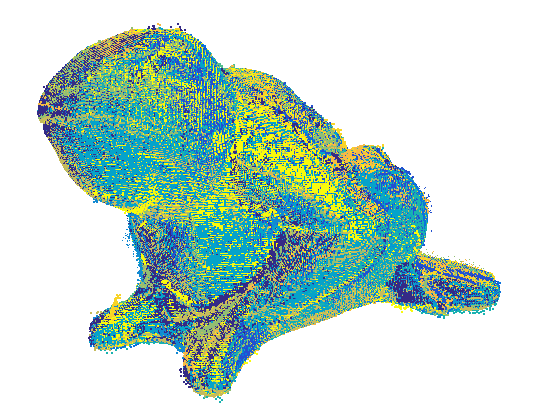}} 
\caption{The ``Frog" model reconstruction result, with the first set of measurements consists of $29\%$ least registration error chosen measurements.}
\label{fig:FrogModelresults}
\end{figure}

\begin{figure}
\centering    
\includegraphics[width=0.9\textwidth]{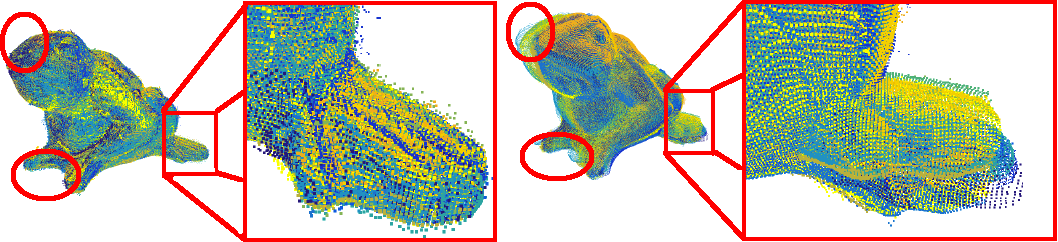}
\caption{Highlight of some of the differences between a good reconstruction obtained by the contraction based method to a bad reconstruction when error $>0.01$.}
\label{fig:FrogModelHighlights}
\end{figure}

 \begin{table}\centering
 \scalebox{0.9}{
  \begin{tabular}{c|cccccc}
    Algorithm   & Contraction (MLE)  			& Separation (MLE)    & PD (MLE) 				  & Spectral 	& Contraction (EIG)  & QD			\\ \hline\hline
    Error on set 1  & \textbf{.001757}   & \textbf{.00176}     & \textbf{.001763}  & $>0.01$   & .003891 		& $>0.01$ \\
    Error on set 2  & \textbf{.001758}   & \textbf{.001759}   & .001816  				  & $>0.01$   & .002807 		& $>0.01$
  \end{tabular}  }
  \caption{Numerical results over the "Frog" model. Set 1 includes $29\%$ of available measurements, set 2 consists of $59\%$ of available measurement. The best result (up to fourth) are highlighted in bold.}
  \label{tab:resultsFrog}
\end{table}

The second example is registration over the ``Skeleton Hand" model of the Georgia Tech large geometric models archive, \url{http://www.cc.gatech.edu/projects/large_models/hand.html}. In this example we have nine scans, each of between $33,000$ to $55,000$ points, where the complete model holds $327,323$ points. We use the same technique of deriving pairwise measurements, as done in previous example. Due to the small amount of scans, we use one optimal set of measurements which consists of $9$ out of the $36$ possible pairwise measurements. The raw scans in this example are somehow more similar than the ones in the frog model, which leads to less variance in error. However, a very similar hierarchy of performance was observed in this example as well, as seen in Table~\ref{tab:resultsHand}. The synchronization via contraction, based on MLE rotations synchronization and with $\lambda=100$, presents the least erroneous result. Contrary to the previous example, the spectral method shows very good results as well. The other two compactification methods use $\lambda=1000$ for contraction with the spectral rotation synchronization (EIG) and $\lambda=100$ for the projection method (PD) with the MLE rotations synchronization. The reconstruction by synchronization via contraction is given visually in Figure~\ref{fig:SkelatonHandModel}.

\begin{figure}
\centering    
\subfloat[Original object]{	 \label{subfig:handFull1}   \includegraphics[width=0.23\textwidth]{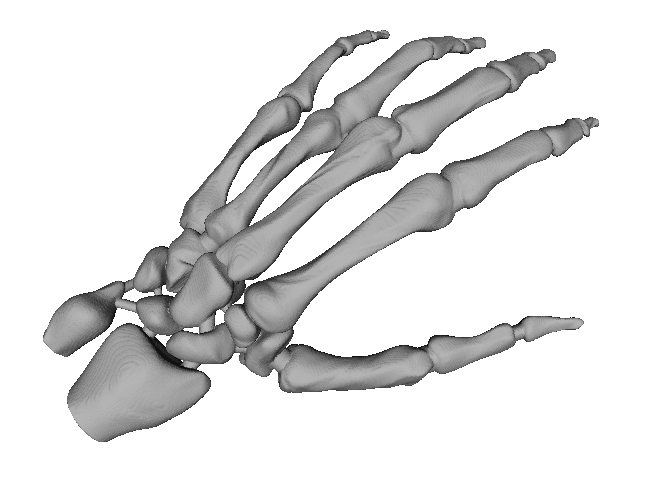}} \qquad
\subfloat[Raw scans]{	\label{subfig:handSamples}    \includegraphics[width=0.33\textwidth]{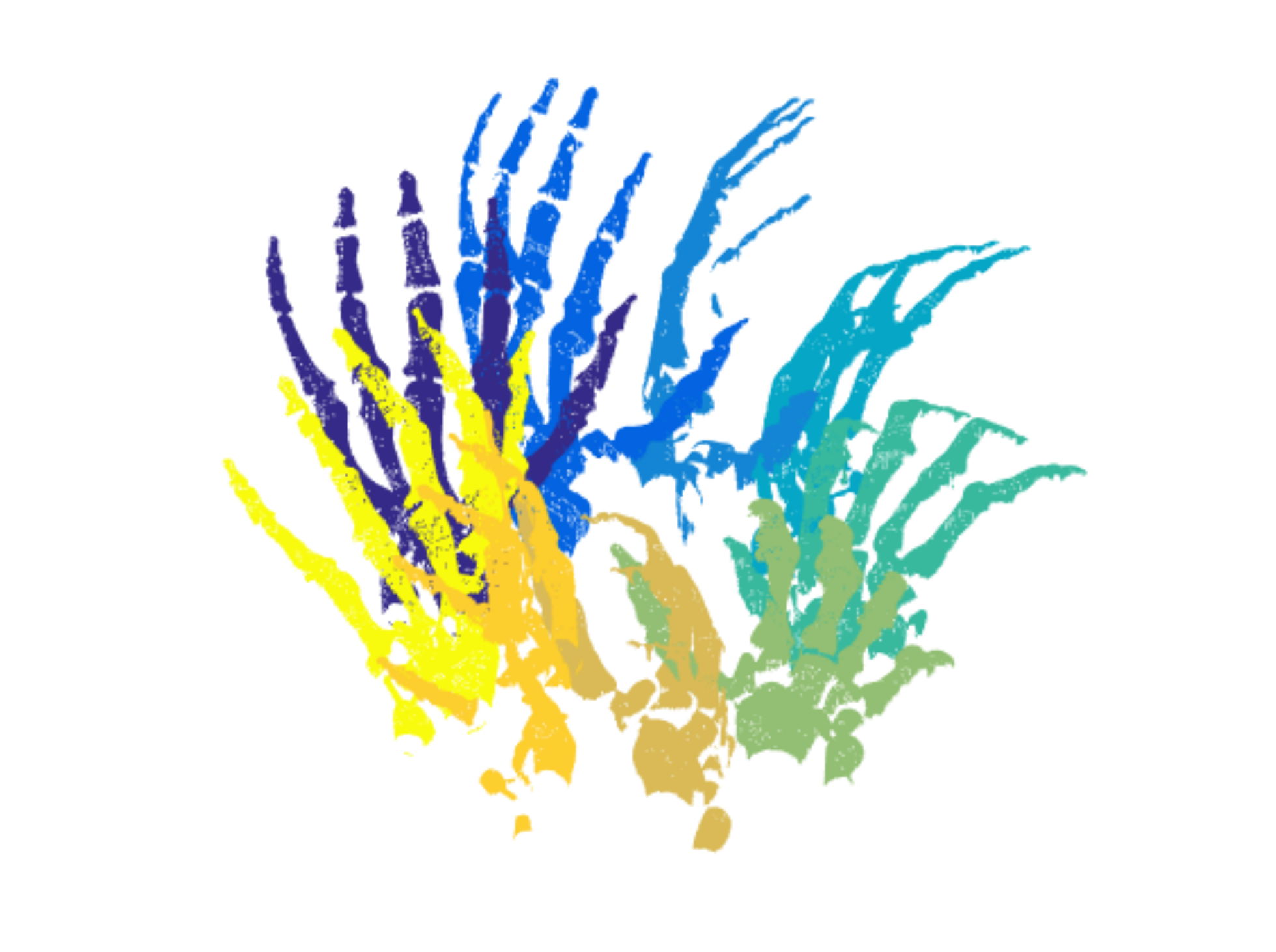}}   
\subfloat[Reconstruction]{	  \label{subfig:handRecons}  \includegraphics[width=0.34\textwidth]{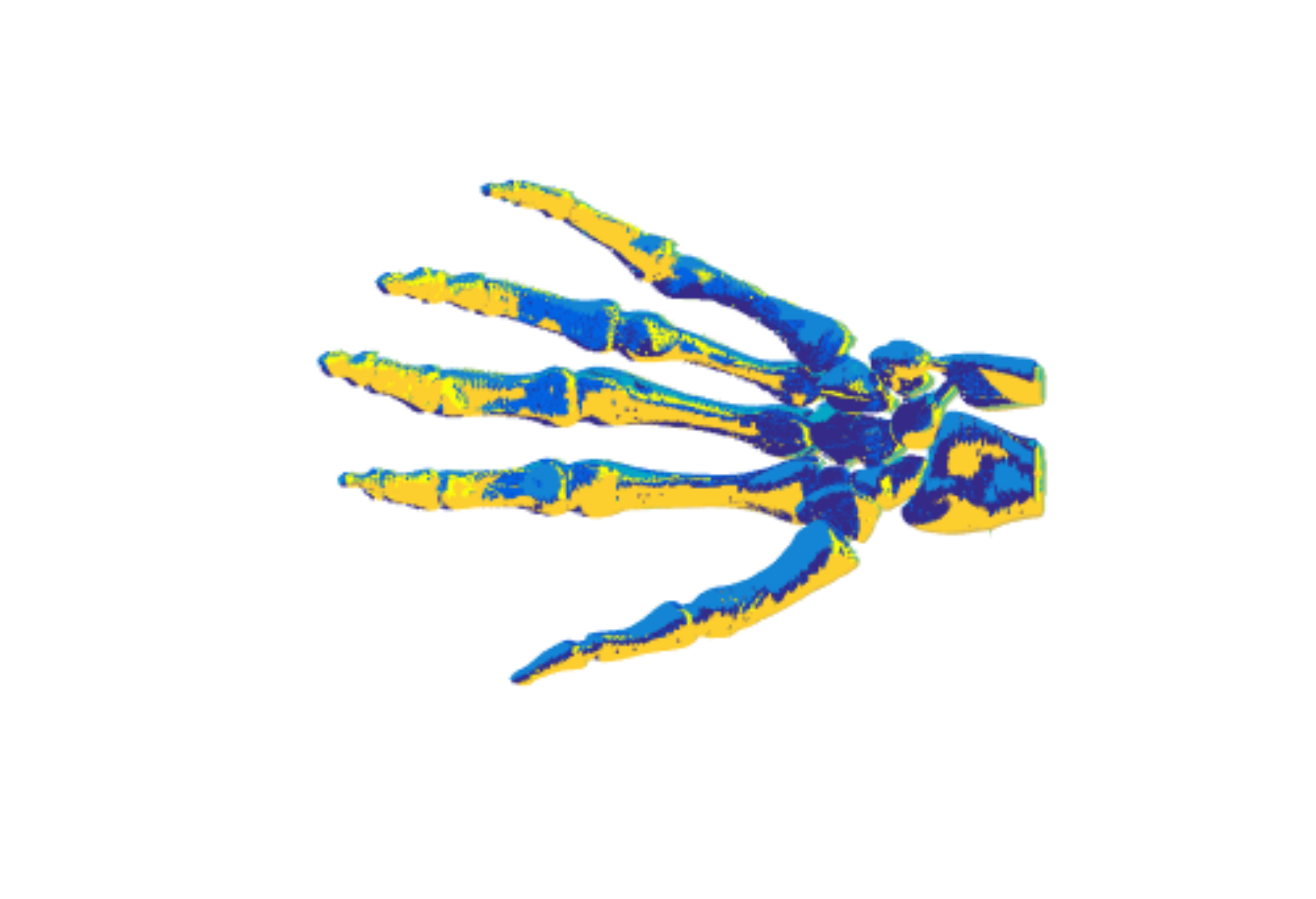}}  
\caption{The ``Skeleton Hand" model, its raw scans and their registration by synchronization via contraction.}
\label{fig:SkelatonHandModel}
\end{figure}

\begin{table}\centering
 \scalebox{0.90}{
  \begin{tabular}{c|cccccc}
    Algorithm & Contraction (MLE)  & Separation (MLE) & PD (MLE) & Spectral & Contraction (EIG) & DQ \\ \hline\hline
    Error & \textbf{.00077} & .00078 & .00078 & .00079 & .00079 & .00135
  \end{tabular} }
  \caption{Empirical results over the "Hand" model.}
  \label{tab:resultsHand}
\end{table}


\section{Conclusion}
This paper focuses on applying a compactification process, that is, a mapping from a non-compact domain into a compact one, for solving synchronization problem. We illustrate this idea using contraction, a tool for compactification drawn from group theory. The contraction mappings enable to transform measurements from a Cartan motion group to an associated compact group. We then efficiently solve the reduced synchronization problem over the compact group, which ultimately provides a solution for synchronization in its original domain. We analyze the properties of synchronization via contraction and show its numerical efficiency in several different data scenarios.

As an extension to this work, it will be interesting to further consider contractions (or any alternative compactification methods) to solve approximation problems on non-compact groups other than synchronization, such as non-unique games \cite{bandeira2015non}, and to form a broader base of knowledge regarding the power of these methods. Furthermore, since compactification methods exist for general Lie groups, these methods might find use for addressing different classes of problems. As another perspective, we hope in future work to extend our analysis to more refined models of noise and data; for example, by considering other statistical measures in $\se(d)$ or by changing the assumptions on the graph structure of the data.


\section*{Acknowledgments}
The authors were partially supported by Award Number R01GM090200 from the NIGMS, FA9550-12-1-0317 and FA9550-13-1-0076 from AFOSR, Simons Foundation Investigator Award and Simons Collaborations on Algorithms and Geometry, and the Moore Foundation Data-Driven Discovery Investigator Award.
The authors thank Amit Bermano for many useful discussions regarding 3D registration. 

\bibliographystyle{plain}
\bibliography{SNC_bib_V2}

\appendix

\section{Complementary proofs} \label{app:proofs_of_mat_projection}

\subsection{Proof of Proposition~\ref{prop:detailed_phi}} \label{app:proof_of_phi_form}

\begin{proof}
Let $T_{\lambda} = \left[ \begin{smallmatrix} \eye_{d} & b / \lambda \\ \mathbf{0}_{1\times d} & 1\end{smallmatrix} \right]$, so that $g_{\lambda} = T_{\lambda}\left[ \begin{smallmatrix} \mu & \mathbf{0}_{d\times 1} \\ \mathbf{0}_{1\times d} & 1\end{smallmatrix} \right]$. We rewrite the SVD of $g_{\lambda}$ in terms of the SVD of $T_{\lambda} = U_{\lambda}\Sigma_{\lambda}V_{\lambda}^T$, i.e., $g_{\lambda} = U_{\lambda}\Sigma_{\lambda}V_{\lambda}^T\left[ \begin{smallmatrix} \mu & \mathbf{0}_{d\times 1} \\ \mathbf{0}_{1\times d} & 1\end{smallmatrix} \right]$. Here, the columns of $U_{\lambda}$ and $V_{\lambda}$ are the eigenvectors of 
\begin{equation} \label{eqn:explicitTT}
T_{\lambda}T_{\lambda}^T = \eye_{d+1} + \left[ \begin{smallmatrix} bb^T/\lambda^2 & b/\lambda \\ b^T/\lambda & 0 \end{smallmatrix}\right] \quad
 \text{ and } \quad T_{\lambda}^TT_{\lambda} = \eye_{d+1} + \left[ \begin{smallmatrix} \mathbf{0}_{d\times d} & b/\lambda \\ b^T/\lambda & b^Tb/\lambda^2 \end{smallmatrix}\right], 
\end{equation}
respectively. 

We observe that the subspace of vectors orthogonal to $b$ is of dimension $d-1$, and is the nullspace of both 
\begin{equation} \label{eqn:nullspaceof}
\left[ \begin{smallmatrix} bb^T/\lambda^2 & b/\lambda \\ b^T/\lambda & 0 \end{smallmatrix}\right] \quad \text{  and  } \quad \left[ \begin{smallmatrix} \mathbf{0}_{d\times d} & b/\lambda \\ b^T/\lambda & b^Tb/\lambda^2 \end{smallmatrix}\right]. 
\end{equation}
Therefore, any set of vectors $\left\{ \left[ \begin{smallmatrix} \nu_i \\ 0 \end{smallmatrix} \right] \right\}_{i=1}^{d-1}$ that forms an orthogonal basis for this subspace is also a set of orthogonal eigenvectors, with the eigenvalue $1$, for $T_{\lambda}T_{\lambda}^T$ and $T_{\lambda}^TT_{\lambda}$. 

To find the two remaining columns of $U_{\lambda}$ and $V_{\lambda}$, we make the ansatz that $\left[\begin{smallmatrix} b/\lambda \\ \alpha \end{smallmatrix}\right]$ and $\left[\begin{smallmatrix} b/\lambda \\ \beta \end{smallmatrix}\right]$ are the two remaining eigenvectors of $T_{\lambda}T_{\lambda}^T$ and $T_{\lambda}^TT_{\lambda}$, respectively. We also remember that the spectrum of the matrices are identical. Indeed, using~\eqref{eqn:explicitTT} we get
\[T_{\lambda}T_{\lambda}^T \left[ \begin{smallmatrix} b/\lambda\\ \alpha \end{smallmatrix}\right] = \left[ \begin{smallmatrix} \left(1 + \alpha + \|b\|_2^2/\lambda^2 \right)b/\lambda\\ \alpha + \|b\|_2^2/\lambda^2 \end{smallmatrix}\right] =  \left(1+\alpha + \|b\|_2^2/\lambda^2 \right) \left[ \begin{smallmatrix} b/\lambda\\  \alpha \end{smallmatrix}\right] , \]
where the last equality holds iff $\alpha^2 + \left(\|b\|_2^2/\lambda^2\right)(\alpha-1) = 0$, that is with the eigenvalues
\[ \alpha_{1,2}  = \frac{-\frac{\|b\|_2^2}{\lambda^2} \pm \sqrt{\frac{\|b\|_2^4}{\lambda^4} + 4\frac{\|b\|_2^2}{\lambda^2}}}{2} . \]
Similarly,
\[ T_{\lambda}^TT_{\lambda} \left[ \begin{smallmatrix} b/\lambda\\ \beta \end{smallmatrix}\right] = \left[ \begin{smallmatrix} \left( \beta+1\right) b/\lambda \\ \left(\|b\|_2^2/\lambda^2+1\right)\beta + \|b\|_2^2/\lambda^2 \end{smallmatrix}\right] = \left(\beta+1\right) \left[ \begin{smallmatrix} b/\lambda\\ \beta \end{smallmatrix}\right] ,\]
with the eigenvalues
\[  \beta_{1,2}  = \frac{\frac{\|b\|_2^2}{\lambda^2} \pm \sqrt{\frac{\|b\|_2^4}{\lambda^4} + 4\frac{\|b\|_2^2}{\lambda^2}}}{2} = -\alpha_{2,1} .\]
Denote the normalized remaining eigenvectors of $U_\lambda$ as
\[ u_{1} = \frac{1}{\sqrt{\|b\|^2/\lambda^2 + \alpha_+^2}} \left[\begin{smallmatrix} b/\lambda \\ \alpha_+ \end{smallmatrix} \right] , \quad u_2 = \frac{1}{\sqrt{\|b\|^2/\lambda^2 + \alpha_-^2}} \left[\begin{smallmatrix} b/\lambda \\ \alpha_- \end{smallmatrix} \right] , \]
and those of $V_\lambda$ as
\[ v_1 = \frac{1}{\sqrt{\|b\|^2/\lambda^2 + \beta_+^2}} \left[\begin{smallmatrix} b/\lambda \\ \beta_+ \end{smallmatrix} \right], \quad v_2= \frac{1}{\sqrt{\|b\|^2/\lambda^2 + \beta_-^2}} \left[\begin{smallmatrix} b/\lambda \\ \beta_- \end{smallmatrix} \right] .\]
As a result, with an orthogonal basis $\left\{ \left[ \begin{smallmatrix} \nu_i \\ 0 \end{smallmatrix} \right] \right\}_{i=1}^{d-1}$ for the null subspace of \eqref{eqn:nullspaceof}, we get
\[ \begin{aligned}[t]
    \Phi_{\lambda}\left(g\right) &= U_{\lambda}V_{\lambda}^T\left[\begin{smallmatrix} \mu & \mathbf{0}_{d\times 1} \\ \mathbf{0}_{1\times d} & 1\end{smallmatrix}\right] \\
																 &= \left(\left[\begin{smallmatrix} \sum_{i = 1}^{d-1} \nu_i \nu_i^T  & \mathbf{0}_{d\times 1}\\ \mathbf{0}_{1\times d} & 0\end{smallmatrix}\right] +  v_1u_1^T +  v_2u_2^T \right)\left[\begin{smallmatrix} \mu & \mathbf{0}_{d\times 1} \\ \mathbf{0}_{1\times d} & 1\end{smallmatrix}\right] \\
																 & = \left[\begin{matrix}  \left( 2\tau_{\lambda}\hat{b}\hat{b}^T + P \right)\mu &\frac{ \tau_{\lambda}}{\lambda}b\\ - \frac{\tau_{\lambda}}{\lambda}b^T\mu & 2\tau_{\lambda} \end{matrix}\right] .
      \end{aligned} \]
Note that a direct calculation shows that the columns of $\Phi_{\lambda}\left(g\right)$ are indeed orthonormal. The explicit form of $\Phi_{\lambda}$ also imples positive determinant of $\Phi_{\lambda}\left(g\right)$ since for blocks $A,B,C,D$, where $D$ is inverible, it holds that \cite{silvester2000determinants}
\[ \det(\left[\begin{matrix}  A \enspace B \\ C \enspace D \end{matrix}\right])= \det(A-BD^{-1}C)\det(D) . \]
In addition, $\lambda,\tau_{\lambda}>0$, $\mu \in \so(d)$, and the spectrum of $P$, as a projection operator, is non-negative. Finally, writing $P = I - \hat{b}\hat{b}^T$ and the claim follows.
\end{proof}

\subsection{Proof of Proposition~\ref{prop:matrix_proj_app_hom}}  \label{app:proof_of_approximated_homomorphism}

\begin{proof}
Starting with the inverse claim of the approximated homomorphism, we have that for $g_{\lambda} = V_{\lambda}\Sigma_{\lambda}W_{\lambda}^T$ we also  get $g_{\lambda}^{-1} = W_{\lambda}\Sigma_{\lambda}^{-1}V_{\lambda}^T$. Therefore,
\[ \Phi_{\lambda}\left(g^{-1}\right) = W_{\lambda}V_{\lambda}^T = \left(V_{\lambda}W_{\lambda}^T\right)^{-1} = \left(\Phi_{\lambda}(g)\right)^{-1} .\]
For the second part of the approximated homomorphism, observe that $\tau_{\lambda}$ of Proposition~\ref{prop:detailed_phi} satisfies (by Taylor)
\begin{equation} \label{eqn:tau_Taylor_expn}
\tau_{\lambda} = \frac{1}{2} + \mathcal{O}\left(\frac{1}{\lambda^2} \right) . 
\end{equation}
Recall that $\{  \nu_i \}_{i=1}^{d-1}$  is an orthogonal basis to the orthogonal complement of $\{b\}$, which implies that $\hat{b}\hat{b}^T + \sum_{i=1}^{d-1}\nu_i\nu_i^T = \eye$. Therefore, 
\[ \left\|2\tau_{\lambda}\hat{b}\hat{b}^T + \sum_{i=1}^{d-1}\nu_i\nu_i^T - \eye \right\|_F = 
\mathcal{O}\left(\frac{1}{\lambda^2}\right) \left\| \hat{b}\hat{b}^T  \right\|_F =
\mathcal{O}\left(\frac{1}{\lambda^2}\right). \] 
In other words, we can denote for simplicity
\begin{equation} \label{eqn:simpleNotation}
\Phi_{\lambda}\left( g \right) = \left[ \begin{matrix} X\mu & \frac{\tau_\lambda}{\lambda}b \\ -\frac{\tau_\lambda}{\lambda}b^T\mu & 2 \tau_\lambda \end{matrix} \right] , \quad X = \eye + S, \quad \norm{S}_F = \mathcal{O}\left(\frac{1}{\lambda^2}\right) .
\end{equation}
We show the leading order of $\norm{\Phi_{\lambda}(g_1g_2) - \Phi_{\lambda}(g_1)\Phi_{\lambda}(g_2)}^2_F$ by dividing the matrix to the four main blocks as
\[  \norm{\left[\begin{matrix}  A_{d \times d} \enspace B_{d \times 1} \\ C_{1 \times d} \enspace D_{1 \times 1} \end{matrix}\right]}_F^2 =  \norm{A}_F^2+\norm{B}_2^2 + \norm{C}_2^2 + D^2 . \]
We use \eqref{eqn:tau_Taylor_expn} and the notation in \eqref{eqn:simpleNotation} with subscripts. The upper block yields
\begin{align*}
\norm{X_{12}\mu_1\mu_2 - X_1 \mu_1 X_2 \mu_2 + \frac{1/4 + \mathcal{O}\left(\frac{1}{\lambda^2}\right)}{\lambda^2}b_1b_2^T \mu_1 }_F & =   \\
\norm{ \left( \eye + S_{12} \right) \mu_1- \left( \eye + S_{1} \right) \mu_1 \left( \eye + S_{2} \right)  + \mathcal{O}\left(\frac{1}{\lambda^2}\right) b_1b_2^T\mu_1\mu_2^T }_F & = \\
\norm{S_{12}\mu_1 - S_{1}\mu_1 +  \mu_1 S_{2} + S_1\mu_1 S_{2} + \mathcal{O}\left(\frac{1}{\lambda^2}\right)b_1b_2\mu_1\mu_2 }_F & = \mathcal{O}\left(\frac{1}{\lambda^2}\right) .
\end{align*}
The last equality is since the different summands are all matrices of leading order $\mathcal{O}\left(\frac{1}{\lambda^2}\right)$. The second block (last column upper part) is
\[ \frac{1/2 + \mathcal{O}\left(\frac{1}{\lambda^2}\right)}{\lambda} \left( \mu_1b_2+b_1\right) - \frac{1/2 + \mathcal{O}\left(\frac{1}{\lambda^2}\right)}{\lambda}X_1\mu_1 b_2+ 2\frac{1/4 + \mathcal{O}\left(\frac{1}{\lambda^2}\right)}{\lambda}b_1 . \]
Thus, the leading order is $\norm{\frac{1}{2\lambda}\left(\eye - X_1\right) \mu_1 b_2 }_2 = \mathcal{O}\left(\frac{1}{\lambda^3}\right)$. Similarly, the third block is
\[ -\frac{1/2 + \mathcal{O}\left(\frac{1}{\lambda^2}\right)}{\lambda} \left( b_1^T \mu_1 \mu_2 \right) + \frac{1/2 + \mathcal{O}\left(\frac{1}{\lambda^2}\right)}{\lambda} b_1^T \mu_1 X_2\mu_2 + 2\frac{1/4 + \mathcal{O}\left(\frac{1}{\lambda^2}\right)}{\lambda}b_2^T \mu _2 . \]
The matrix $\mu_2$ is canceled out and  the leading order is again $\norm{\frac{1}{2\lambda} b_1^T\mu_1 \left(\eye - X_2\right) }_2 = \mathcal{O}\left(\frac{1}{\lambda^3}\right)$. Finally, for the last scalar block we have
\[ 2\left(1/2 + \mathcal{O}\left(\frac{1}{\lambda^2}\right) \right) + \frac{1/2 + \mathcal{O}\left(\frac{1}{\lambda^2}\right)}{\lambda^2} b_1^T \mu_1 b_2 - 4\left(1/4 + \mathcal{O}\left(\frac{1}{\lambda^2}\right) \right) = \mathcal{O}\left(\frac{1}{\lambda^2}\right) . \]
\end{proof}

\subsection{Proof of Proposition~\ref{prop:distance_between_const_functions}}  \label{app:proof_of_norms_dist_bound}

To prove Proposition~\ref{prop:distance_between_const_functions} we use two auxiliary Lemmas.
\begin{lemma} \label{lemma:orthogonal_mat_dist}
Suppose $Q_1,Q_2$ and $P_1,P_2$ are orthogonal matrices, such that 
\[ \norm{Q_1-Q_2}_F\le\varepsilon_1, \quad \text{ and }, \quad \norm{P_1-P_2}_F\le\varepsilon_2. \]
Then,
\[ \norm{Q_1P_1-Q_2P_2}_F \le \varepsilon_1+\varepsilon_2 .\]
\end{lemma}
\begin{proof}
We use the orthogonality invariance of the Frobenius norm and the triangle inequality to get
\[ \begin{aligned}[t]
    \norm{Q_1P_1-Q_2P_2}_F        &=  \norm{I-P_1^T Q_1^T Q_2P_2}_F                                \\
            &=\norm{I+P_1^T \left( I-Q_1^T Q_2\right) P_2 -  P_1^T P_2 }_F  \\
            &\le \norm{I- P_1^T P_2} + \norm{P_1^T \left( I-Q_1^T Q_2\right) P_2  }_F  \\
            & = \norm{P_1-P_2}_F + \norm{Q_1-Q_2}_F \le \varepsilon_1 + \varepsilon_2 .
      \end{aligned} \]
\end{proof}

The next lemma bounds the distance between two exponential maps with respect to the original distance in the algebra. The two distances are linked by a term that involves the commutator (which is also related to the curvature of the Lie group manifold).
\begin{lemma} \label{lemma:exp_dist_bound}  
Let $p_1,p_2 \in \mathfrc{g} $, and let $\lambda$ be sufficiently large such that 
\[ \exp(p_1/\lambda), \quad \exp(p_2/ \lambda), \text{  and  }\exp(p_1/ \lambda)\left(\exp(p_2/\lambda)\right)^{-1} \] 
are within the injectivity radius around the identity. Then,
\[ \norm{\exp(p_1/\lambda)-\exp(p_2/\lambda)} \le \norm{p_1-p_2} + C\left( \lambda^{-2} \right) ,     \]
where the constant $C$ depends on the commutator $\comm{p_1}{p_2}$ (and thus can be bound by $ \max{\norm{p_1},\norm{p_2}}$) but independent on $\lambda$.
\end{lemma}
\begin{proof}
Since $\Gb$ is compact, the Lie group exponential map is surjective and there is $q\in \mathfrc{g} $ such that $\exp(q) = \exp(p_1/\lambda)\left(\exp(p_2/\lambda)\right)^{-1}$. Now, the compactness also implies that $\Gb$ admits a bi-invariant metric, see e.g., \cite[Chapter 2]{alexandrino2015lie}, and thus the induced Riemannian metric $d_R(\cdot,\cdot)$ satisfies
\[    d_R(\left( \exp(p_1/\lambda),\exp(p_2/\lambda) \right)  =   d_R \left( \exp(q) ,I \right) =   \norm{q}        . \]
However, from Baker-Campbell-Hausdorff (BCH) formula we get
\[ q = \frac{1}{\lambda}(p_1-p_2) + \mathcal{O}\left( \comm{\frac{p_1}{\lambda}}{\frac{p_2}{\lambda}} \right) . \]
The result of the lemma follows since the intrinsic (geodesic) Riemannian distance bounds the Euclidean distance in the embedded space, and since $\frac{1}{\lambda} \le 1$, that is
 \[ \norm{\exp(p_1/\lambda)-\exp(p_2/\lambda)} \le  d_R \left( \exp(p_1/\lambda),\exp(p_2/\lambda) \right)  \le \norm{p_1-p_2} + \frac{1}{\lambda^2}\mathcal{O}\left(\norm{\comm{p_1}{p_2}}\right) . \]
\end{proof}

We now turn to prove the proposition.
\begin{proof}[Proof of Proposition~\ref{prop:distance_between_const_functions}]
By the approximated homomorphism \eqref{eqn:app_homo} we have 
\[ \norm{\Psi_\lambda(g_1)\Psi_\lambda(g_2)^{-1}-\Psi_\lambda(g)}_F = \norm{\Psi_\lambda(g_1)\Psi_\lambda(g_2^{-1})-\Psi_\lambda(g)}_F .\]
Then, using the second part of \eqref{eqn:app_homo} and triangle inequality we get
\[ \begin{aligned}[t]
    \norm{\Psi_\lambda(g_1)\Psi_\lambda(g_2^{-1})-\Psi_\lambda(g)}_F  &=  \norm{\Psi_\lambda(g_1)\Psi_\lambda(g_2^{-1})-\Psi_\lambda(g_1g_2^{-1})+\Psi_\lambda(g_1g_2^{-1})-\Psi_\lambda(g)}_F  \\
            &\le \mathcal{O}(\lambda^{-2})+ \norm{\Psi_\lambda(g_1g_2^{-1})-\Psi_\lambda(g)}_F                .
      \end{aligned} \]
Denote the decompositions $g_1g_2^{-1} = (k,v),g= (\bar{k},\bar{v}) \in \Gc$ and by \eqref{eqn:cartan_contraction} 
\[ \norm{\Psi_\lambda(g_1g_2^{-1})-\Psi_\lambda(g)}_F =\norm{ \exp(\lambda v) \cdot k-\exp(\lambda \bar{v}) \cdot \bar{k}}_F . \]
Now, recall that $k,\exp(\lambda v),\bar{k},\exp(\lambda \bar{v}) \in \Gb$ and thus they are represented by orthogonal matrices. Apply Lemma~\ref{lemma:orthogonal_mat_dist},
\[ \norm{\Psi_\lambda(g_1g_2^{-1})-\Psi_\lambda(g)}_F \le \norm{k -\bar{k}}_F + \norm{\exp(\lambda v)- \exp(\lambda \bar{v})}_F .\]
Assume the elements are close enough so we can apply Lemma~\ref{lemma:exp_dist_bound} to finally obtain
\[ \norm{\Psi_\lambda(g_1)\Psi_\lambda(g_2)^{-1}-\Psi_\lambda(g)}_F \le \norm{k -\bar{k}}_F + \norm{v- \bar{v}}_F + \mathcal{O}(\lambda^{-2}). \] 
\end{proof}

\subsection{Proof of Proposition~\ref{prop:noise_analysis}}  \label{app:proof_noise_analysis}

\begin{proof}  
We follow the analysis of the phase transition point in the spectral method, as done in \cite[Chapter 5.3]{boumal2014thesis} for the case of a full set of measurements. This proof uses the estimation of the maximal eigenvalue of random block matrix in \cite{girko1996matrix}. In essence, $M=RWR^T$ and $W$ share the same full rank spectrum. Therefore, it is easier to approach the largest eigenvalues of the data matrix $M$ by analyzing $W$. In particular, setting $W_{ii} = 0$ and since $M$ is a symmetric matrix, we also have that $Y= W - \E{W}$ is a symmetric block matrix, with zero mean and $Y_{ii}=0$. These conditions together with the fact that $W$ has a fixed Frobenius norm ($W_{ij}$ are orthogonal matrices) give the general structure needed to apply the results in \cite{girko1996matrix} for estimating the largest eigenvalue in $Y$. There are two specific conditions to show. First is the independence between the blocks which we have from the independence assumptions on $a_{ij}$ and $\upsilon_{ij}$. Second is the expectation over $W_{ij}$. To calculate this expectation, we use $f_V$ and $f_K$ and define a joint PDF $f$ based upon the global Cartan decomposition $\Gb \cong P\cdot K$, 
\[ f(g) = f_V(\log(p))f_K(k), \quad g = p\cdot k \in \Gb .\]
Note that $K$ acts transitively on $P$ so $f$ is well-defined. The coset $X=\Gb/K$ is a symmetric homogeneous space equipped with a Riemannian metric so integration over $X$ is well-understood, and we can use it to separate the integration on $\Gb$. Also, $f$ is separable by definition and smooth. We use the Haar measure and the fact that $\Gb$ is compact to have
\[  \E{W_{ij}} = \int_{\Gb} g f(g) d\mu(g) = \int_{X} p f_V(\log(p)) d\mu(p) \int_{K} k f_K(k) d\mu(k) .\]
The most right integral is equal by assumption to $\left[ \begin{smallmatrix} \E{\upsilon_{ij}} & \mathbf{0}_{d\times 1} \\ \mathbf{0}_{1\times d} & 1\end{smallmatrix} \right] $. For the second integral, we change coordinates to have
\[  \int_{X} p f_V(\log(p)) d\mu(p) = \int_{\norm{v}\le \pi} \exp(M_v) f_V(v) \norm{J(v)}^{-1}  d\mu(v) .\]
Here $v$ in $f_V(v) $ is understood as a vector and inside the exponential we have it as a matrix using $M_x = \left[ \begin{smallmatrix} 0 & -x^T \\ x & 0\end{smallmatrix} \right] $. The injectivity radius which we use for this case is followed by \eqref{eqn:lambda_cond}. For our case of the sphere, since it is a rank one coset, we have an explicit form for the Jacobian, $(I-vv^t)^\frac{1}{2}$, see e.g., \cite[Chapter 9, p. 391]{gilmore2012lie}. To simplify the exponential we use the result obtained in \ref{apx:CartanCalculation} to deduce that only three terms are involved in $\exp(v)$, that are $I$, $v$ and $v^2$. Namely, we end up with
\[  \int_{\norm{v}\le \pi} \left( I + \sin(\norm{v})M_{v/ \norm{v}} +(1-\cos(\norm{v}))M_{v/ \norm{v}}^2  \right) f_V(v) \norm{(I-vv^t)^{-\frac{1}{2}}}  d\mu(v)  , \]
The weight function $g(v) = f_V(v)\norm{(I-vv^t)^{-\frac{1}{2}}}$ is even with $v$, which leaves us with only two terms, the identity and the diagonal part of $v^2$ (the parts that include the quadratic terms) as all other terms lead to odd Integrands over a symmetric domain around the identity. Therefore, the integral is
\[ \left( \int_{\norm{v}\le \pi} g(v)  d\mu(v) \right) I      
- \int_{\norm{v}\le \pi}   \frac{(1-\cos(\norm{v}))}{\norm{v}^2}  \left[ \begin{smallmatrix} I \circ vv^T & 0  \\ 0 & \norm{v}^2 \end{smallmatrix} \right]  g(v)  d\mu(v) 
= \gamma I - \left[ \begin{smallmatrix} \alpha_1  I & 0  \\ 0 & \alpha_2 \end{smallmatrix} \right] . \]
To conclude, the expectation of each noise block is a diagonal matrix of the form
 \[  \E{W_{ij}} = \left[ \begin{smallmatrix} (\gamma-\alpha_1)\beta I_d & \mathbf{0}_{d\times 1} \\ \mathbf{0}_{1\times d} & \gamma-\alpha_2\end{smallmatrix} \right] ,\]
where $\gamma$, $\alpha_1$ and $\alpha_2$ as described in the statement of the proposition. Note that although $\E{W_{ij}}$ not a scalar matrix (the last entry on the diagonal is not necessary equals to the others), the conclusions about the largest eigenvalues $Y$ from \cite[Chapter 5.3]{boumal2014thesis} are still valid.
\end{proof}

\section{Supplements} \label{app:Supplements}

\subsection{The spectral method -- pseudocode algorithm}

The spectral method is given as a pseudocode in Algorithm~\ref{alg:eigenvector_method}. For more details on this algorithm see~\cite{singer2011angular}. A Matlab implementation for the spectral method, for both $\so(d)$ and $\on(d)$, is available online in \url{https://github.com/nirsharon/Synchronization-over-Cartan-Motion-Groups}.

\begin{algorithm}[ht]
\caption{Spectral method for data on a semisimple compact group $G$}
\label{alg:eigenvector_method}
\begin{algorithmic}[1]
\REQUIRE  Ratio measurements $\{ g_{ij} \}_{(i,j) \in E} $. \\ Measurement confidence weights $\{ w_{ij} \}_{(i,j) \in \mathcal{E}} $.  \\ Non-trivial orthogonal faithful representation $\rho$ of order $d$ over $G$. \\ 
\ENSURE An approximated solution for $\{g_i\}_{i=1}^n$.

\STATE Initialize $M$ and $D$ as $dn \times dn$ zeros matrices.
\FOR{$(i,j) \in E$}    
\STATE  Define block $M_{ij} \gets w_{ij} \rho\left({g}_{ij}\right) $.  
\FOR{$\{ \ell \mid (i,\ell) \in E \}$}    
\STATE  Define block $D_{ii} \gets D_{ii}+w_{i \ell} \eye_d$.  
\ENDFOR    
\ENDFOR 
\STATE  $H \gets D^{-1}M$.  
\STATE  $ \mu \gets\operatorname{EIG}(H,d)$. \COMMENT{Extract top $d$ eigenvectors}  
\FOR{$i=1,\ldots,n$}    
\STATE $b_i \gets \operatorname{BLOCK}(\mu,d,i) $. \COMMENT{The $i^\text{th}$block} 
\STATE  $ \mu_i \gets \rd(b_i) $.  \COMMENT{Rounding procedure} 
\ENDFOR    
\RETURN $\{  \rho^{-1}(\mu_i) \}$, $i=1,\ldots,n$.
\end{algorithmic}
\end{algorithm}

\subsection{Calculating the Cartan decomposition} \label{apx:CartanCalculation}

To calculate $\Psi_\lambda^{-1}$ we have to reveal the Cartan decomposition of an element from $\Gb$. In the algebra level, the Cartan decomposition is straightforward since it induced from a direct sum of linear spaces. However, in the group level this calculation is less trivial. To make the discussion concrete, consider $\Gb=\so(d+1)$ and denote by $Q$ the element in $\Gb$ to be decomposed. Then, the Cartan decomposition consists of two matrices $P,R \in \so(d+1)$ such that 
\begin{equation}  \label{eqn:CartanRotation}
 Q = PR, \quad P = \exp(p), \quad p \in \mathfrc{p} , \quad R = \left[ \begin{smallmatrix} \mu & 0 \\ \mathbf{0}_{1\times d} & 1\end{smallmatrix} \right] , \quad \mu\in \so(d)  . 
\end{equation}
Here $ \mathfrc{p}$ is the subspace given from the Cartan decomposition in the algebra level. Let $q = \log(Q)$ be the correspondence Lie algebra element of $Q$, and denote its Cartan decomposition by 
\[ q = q_t+q_p, \quad q_t \in \mathfrc{t}, \quad q_p \in \mathfrc{p} . \] 
The difficulty in decompose at group level is that in general, $q_t \neq r = \log(R)$ and similarly $q_p \neq p$, since by BCH 
\begin{equation}  \label{eqn:BCH_4_cartan}
 q_t+q_p = p + r + \frac{1}{2}[p,r] + \frac{1}{12} \left( [p,[p,r]]+[r,[r,p]]\right) + \ldots .  
\end{equation}
From the other hand, it is not enough to work only in the group level and decompose $Q$ as $ \widehat{P} \left[ \begin{smallmatrix} \widehat{\mu} & 0 \\ \mathbf{0}_{1\times d} & 1\end{smallmatrix} \right]$ (this can be done, for example, by applying Givens rotations to $Q$) does not guarantee that $\log{\widehat{P}} \in \mathfrc{p} $. 

For the special case of measurement in $\se(d)$ and the contraction map there, the most efficient approach for attaining the Cartan decomposition of a given rotation is based upon a generalization of the Rodrigues formula. In particular, we can exploit the special structure of 
\[ p = \left[ \begin{smallmatrix} 0 & -b^T \\ b & 0\end{smallmatrix} \right] \in \mathfrc{p} , \]
where $b \in \mathbb{R}^d$, to have the Rodrigues formula for any $d$ (and not just for $P=\exp(p)$ of order $3$). By \eqref{eqn:eigenvalueofB} we learn that the only nonzero eigenvalues of $p$ are $\pm i \theta$ with $\theta = \norm{b}$. Therefore, inspired from the proof of the Rodrigues formula in \cite{gallier2003computing}, the key observation is that the normalized $p_\theta = \frac{1}{\theta}p$ satisfies $p_\theta^3 = -p_\theta$. This leads to the more general relations $p_\theta^{2m+j} = (-1)^m p_\theta^j$ for $j=1,2$ and $m \in \mathbb{N}$, which imply the Rodrigues formula
\begin{eqnarray*}
 \exp(p) &=& \sum_{n=0}^\infty \frac{p^n}{n!} = I + \frac{\theta  p_\theta}{1!} + \frac{\theta^2  p_\theta^2}{2!} + \ldots \\
 &=& I + \left(  \frac{\theta}{1!}-\frac{\theta^3}{3!}+\frac{\theta^5}{5!} + \ldots \right) p_\theta+ \left(  \frac{\theta^2}{2!}-\frac{\theta^4}{4!}+ \ldots \right) p_\theta^2 \\
  &=& I + \sin(\theta)p_\theta + (1-\cos(\theta))p_\theta^2 .
\end{eqnarray*}
The block structure of $p$ indicates that $p_\theta$ and $p_\theta^2$ do not have any common places of nonzero blocks. Therefore, the first $d$ entries in the last column of $\exp(p)$ are exactly as in $\sin(\theta)p_\theta$ and the last entry (on the diagonal) is affected by $I$ and $p_\theta^2$ and equals to $1+(1-\cos(\theta))(p_\theta^2)_{d+1,d+1}$, which is actually just $\cos(\theta)$ since the normalization implies $(p_\theta^2)_{d+1,d+1}=-1$. The last column is of particular interest as by \eqref{eqn:CartanRotation} we deduce that the last columns of $\exp(p)$ and $Q$ are identical.

We summarize the above in Algorithm~\ref{alg:CartanRotation}. Note that we assume $\theta = \cos^{-1}(Q_{d+1,d+1})<\pi$ as $\theta=\pi$ implies we are on the boundary of the diffeomorphism domain of the exponential map and there is no unique decomposition (which indicates that $\lambda$ was wrongly chosen initially).
\begin{algorithm}[ht]
\caption{Global Cartan decomposition for special orthogonal matrices}
\label{alg:CartanRotation}
\begin{algorithmic}[1]
\REQUIRE  The matrix $Q \in \so(d+1)$. 
\ENSURE The matrices $P$ and $R$ of the Cartan decomposition $Q=PR$ of \eqref{eqn:CartanRotation}.
\STATE $\theta \gets \cos^{-1}\left(Q(d+1,d+1)\right)$
\IF{$0<\theta<\pi$} 
\STATE {$ b \gets \lambda \frac{\theta}{\sin(\theta)} Q(1:d\, ,d+1)$} 
\ELSIF{$\theta = 0$} 
\STATE{$b \gets 0$} 
\ENDIF
\RETURN $ P = \exp(\left[ \begin{smallmatrix} \mathbf{0}_{d\times d} & b \\ -b^T & 0\end{smallmatrix} \right])$, $R = P^TQ$.
\end{algorithmic}
\end{algorithm}

One other general approach for the Cartan decomposition is to approximate $p$ and $r$ using the observation that the right hand side of \eqref{eqn:BCH_4_cartan} can be divided into two types of terms: ones from $\mathfrc{p}$, that are $p,\frac{1}{2}[p,r] ,\frac{1}{12} [r,[r,p]],\ldots$ and the other terms from $\mathfrc{t}$. Then, truncated BCH series can be applied which also requires solving a non-linear polynomial equation in either $r$ or $p$, see e.g., \cite{earp2005constructive}. 

Finally, a more practical approach for the general case would be to simply optimize the decomposition. In details, one can define an optimization on the linear subspace $\mathfrc{p}$, and minimize a loss function in the least squares fashion in the vicinity of the zero vector (that is a neighborhood of the identity in the group level). In this neighborhood, there is a unique solution which is guaranteed by the unique global Cartan decomposition. Given an element to decompose $Q$ on the matrix group $\Gb$, an optimization problem can be defined as
\begin{equation} \label{eqn:cartan_opt}
\begin{aligned}
& \underset{{p \in \mathfrc{p} }}{\text{minimize}}
& & \left \| U_1 \exp(-p)Q U_2 \right \|_F^2 \\
\end{aligned}
\end{equation}
where $U_1$ and $U_2$ are matrices that restrict the product $\exp(-p)Q $ to the structure of matrices from $K$. For example, for $\so(d+1)$ we need to zero out the last column, and so one chooses $U_2$ to be the unit vector in $\mathbb{R}^{d+1}$ with one in its last entry, and $U_1 = \left[ \begin{smallmatrix} I_{d\times d} & 0\\ 0 & 0\end{smallmatrix} \right]$. For the case of contraction map over $\mmg(d, \ell)$, we have $\Gb=\on(d+\ell)$ (see Example~\ref{exm:MMG} and Section~\ref{subsec:MMG}) and so we can use $U_1 = \left[ \begin{smallmatrix} I_{\ell \times \ell} & 0_{\ell \times d } \end{smallmatrix} \right] $ and $U_2 = \left[ \begin{smallmatrix} 0_{\ell \times d } \\  I_{d \times d} \end{smallmatrix} \right]  $.
 
By extracting the gradient of the cost function in \eqref{eqn:cartan_opt} there are efficient gradient-based solvers to use for \eqref{eqn:cartan_opt}, such as the trust region algorithm. Testing this approach numerically, the solution $p^\ast$ is typically found, to double precision, in about $20$ iterations. Note that we defined \eqref{eqn:cartan_opt} as unconstrained problem, while we really want to obtain a solution inside $\norm{p} \le \pi$ (to ensure $p$ is inside the injectivity radius of the exponential, see \eqref{eqn:lambda_cond}). Such a solution zeroes out the cost function and thus local minima (if exist) can be avoided. Typically, by starting with an initial guess around the identity, a valid solution is attended. However, to guarantee that, one might use a constraint optimization such as interior point algorithm. A Matlab implementation of both Algorithm~\ref{alg:CartanRotation} and the Cartan decomposition via optimization are available online in \url{https://github.com/nirsharon/Synchronization-over-Cartan-Motion-Groups}.

\end{document}